\documentclass[11.5pt]{amsart}
\usepackage{amsmath,amsthm,amsfonts,amssymb,mathrsfs}
\usepackage[all]{xy}
\usepackage{eufrak}
\usepackage{amssymb}
\usepackage{latexsym}
\usepackage{amsmath,amsthm}
\usepackage[ngerman,french,english]{babel}
\usepackage[utf8]{inputenc}
\usepackage{aeguill}
\usepackage{amssymb}
\usepackage{mathrsfs}
\usepackage{hyperref}
\usepackage{float}
\usepackage{graphicx}
\usepackage{etoolbox}
\usepackage{enumerate}
\usepackage{ textcomp }
\usepackage[total={6.5in,8.75in}, top=1.2in, left=1.0in, includefoot]{geometry}
\usepackage[utf8]{inputenc}
\usepackage{tikz}[2013/12/13]

\DeclareFontFamily{U}{wncy}{}
    \DeclareFontShape{U}{wncy}{m}{n}{<->wncyr10}{}
    \DeclareSymbolFont{mcy}{U}{wncy}{m}{n}
    \DeclareMathSymbol{\Sh}{\mathord}{mcy}{"58} 

\def\@tocline#1#2#3#4#5#6#7{\relax
  \ifnum #1>\c@tocdepth 
  \else
    \par \addpenalty\@secpenalty\addvspace{#2}%
    \begingroup \hyphenpenalty\@M
    \@ifempty{#4}{%
      \@tempdima\csname r@tocindent\number#1\endcsname\relax
    }{%
      \@tempdima#4\relax
    }%
    \parindent\z@ \leftskip#3\relax \advance\leftskip\@tempdima\relax
    \rightskip\@pnumwidth plus4em \parfillskip-\@pnumwidth
    #5\leavevmode\hskip-\@tempdima
      \ifcase #1
       \or\or \hskip .5em \or \hskip 1em \else \hskip 1.5em \fi%
      #6\nobreak\relax
    \dotfill\hbox to\@pnumwidth{\@tocpagenum{#7}}\par
    \nobreak
    \endgroup
  \fi}
\makeatother
\newcommand{\C}{\mathbb{C}}
\newcommand{\N}{\mathbb{N}}
\newcommand{\Z}{\mathbb{Z}}
\newcommand{\R}{\mathbb{R}}
\newcommand{\Q}{\mathbb{Q}}
\newcommand{\F}{\mathbb{F}}
\newcommand{\A}{\mathbb{A}}
\newcommand{\X}{\mathbb{X}}

\newcommand{\Gal}{\operatorname{Gal}}

\newcommand{\Fr}{\mathrm{Fr}}
\newcommand{\geo}{\mathrm{geo}}

\newcommand{\rk}{\operatorname{rk}}

\newcommand{\ad}{\operatorname{ad}}

\renewcommand{\sc}{\operatorname{sc}}
\renewcommand{\sp}{\operatorname{sp}}

\newcommand{\ssp}{\operatorname{ssp}}
\renewcommand{\ss}{\operatorname{ss}}
\newcommand{\der}{\operatorname{der}}
\newcommand{\red}{\operatorname{red}}

\newcommand{\Hom}{\operatorname{Hom}}
\newcommand{\Aut}{\operatorname{Aut}}
\newcommand{\Out}{\operatorname{Out}}
\newcommand{\Inn}{\operatorname{Inn}}
\newcommand{\End}{\operatorname{End}}

\newcommand{\GL}{\mathrm{GL}}
\newcommand{\PGL}{\mathrm{PGL}}
\newcommand{\SL}{\mathrm{SL}}

\newcommand{\et}{\mathrm{\text{\'e}t}}

\newcommand{\bB}{\mathbf{B}}
\newcommand{\bC}{\mathbf{C}}
\newcommand{\bD}{\mathbf{D}}
\newcommand{\bG}{\mathbf{G}}
\newcommand{\bH}{\mathbf{H}}
\newcommand{\bM}{\mathbf{M}}
\newcommand{\bN}{\mathbf{N}}
\newcommand{\bS}{\mathbf{S}}
\newcommand{\bT}{\mathbf{T}}
\newcommand{\bU}{\mathbf{U}}

\newcommand{\bY}{\mathbf{Y}}

    \makeatletter
    \renewcommand\paragraph{\@startsection{paragraph}{4}{\z@}%
                                             {-1ex plus -1ex minus -0.2ex}
                                             {-1ex plus 0.2ex}
                                          {\normalfont\normalsize}}
																					
   \renewcommand\subparagraph{\@startsection{subparagraph}{5}{\z@}%
                                          {-3.25ex\@plus -1ex \@minus -.2ex}%
                                          {0.0001pt \@plus .2ex}%
                                          {\normalfont\normalsize\bfseries}}
     
    \counterwithin{paragraph}{subsubsection}
    \counterwithin{subparagraph}{paragraph}
    \makeatother

\newtheorem{innercustomthm}{Theorem}
\newenvironment{customthm}[1]
  {\renewcommand\theinnercustomthm{#1}\innercustomthm}
  {\endinnercustomthm}

\newtheorem{innercustomprop}{Proposition}
\newenvironment{customprop}[1]
  {\renewcommand\theinnercustomprop{#1}\innercustomprop}
  {\endinnercustomprop}
	
\newtheorem{innercustommthm}{Main theorem}
\newenvironment{custommthm}[1]
  {\renewcommand\theinnercustommthm{#1}\innercustommthm}
  {\endinnercustommthm}

\newtheorem{innercustomconj}{Conjecture}
\newenvironment{customconj}[1]
  {\renewcommand\theinnercustomconj{#1}\innercustomconj}
  {\endinnercustomconj}
	
\newtheorem{thm}{Theorem}[section]

\newtheorem{cor}[thm]{Corollary}

\newtheorem{prop}[thm]{Proposition}
\newtheorem{lemma}[thm]{Lemma}

\newtheorem{remark}[thm]{Remark}

\theoremstyle{definition}
\newtheorem{definition}[thm]{Definition}

\begin{document}

\title[On the rationality of algebraic monodromy groups of compatible systems]
{On the rationality of algebraic monodromy groups of compatible systems}

\author{Chun Yin Hui}
\email{chhui@maths.hku.hk, pslnfq@gmail.com}
\address{
Department of Mathematics\\
The University of Hong Kong\\
Pokfulam, Hong Kong
}

\maketitle

\begin{abstract} 
Let $E$ be a number field and $X$ a smooth geometrically connected variety defined over a characteristic $p$ finite field.
Given an $n$-dimensional pure $E$-compatible system 
of semisimple $\lambda$-adic representations of the \'etale fundamental group of $X$ 
with connected algebraic monodromy groups $\bG_\lambda$, 
we construct a common $E$-form $\bG$ of all the groups $\bG_\lambda$ and
in the absolutely irreducible case, a common $E$-form $\bG\hookrightarrow\GL_{n,E}$ 
of all the tautological representations $\bG_\lambda\hookrightarrow\GL_{n,E_\lambda}$ (Theorem \ref{thm1}).
Analogous rationality results in characteristic $p$ assuming the existence of 
crystalline companions in $\mathrm{\textbf{F-Isoc}}^{\dagger}(X)\otimes E_{v}$ for all $v|p$ (Theorem \ref{thmcrys})
and in characteristic zero assuming ordinariness (Theorem \ref{thm2}) are also obtained.
Applications include a construction of $\bG$-compatible system from some $\GL_n$-compatible system and 
some results predicted by the Mumford-Tate conjecture.
\end{abstract}

\selectlanguage{english} 
{\tableofcontents}

\newpage
\section{Introduction}
\subsection{The Mumford-Tate conjecture}
Let $A$ be an abelian variety defined over a number field $K\subset\C$, 
$V_\ell:=H^1(A_{\overline K},\Q_\ell)$ the \'etale cohomology groups for all primes $\ell$,
and $V_\infty=H^1(A(\C),\Q)$ the singular cohomology group.
The famous Mumford-Tate conjecture \cite[$\mathsection4$]{Mu66} asserts that 
the $\ell$-adic Galois representations $\rho_\ell:\Gal(\overline K/K)\to \GL(V_\ell)$
are independent of $\ell$, in the sense that if $\bG_\ell$ denotes
the \emph{algebraic monodromy group} of $\rho_\ell$ 
(the Zariski closure of the image of $\rho_\ell$ in $\GL_{V_\ell}$)
and $\bG_{MT}$ denotes the \emph{Mumford-Tate group} of the pure Hodge structure of $V_\infty$, then 
via the comparison isomorphisms $V_\ell\cong V_\infty\otimes\Q_\ell$ one has
\begin{equation}\label{iso1}
(\bG_{MT}\hookrightarrow\GL_{V_\infty})\times_\Q\Q_\ell\hspace{.1in}\cong\hspace{.1in} 
\bG_\ell^\circ\hookrightarrow\GL_{V_\ell}\hspace{.1in} \mathrm{for~all~\ell.}
\end{equation}
In particular, the representations $\rho_\ell$ are semisimple and the 
identity components $\bG_\ell^\circ$ are reductive with the same absolute root datum.
This conjectural $\ell$-independence of different algebraic monodromy representations 
can be formulated almost identically for projective smooth varieties $Y$ defined over $K$, 
and more generally, for pure motives over $K$ by the universal cohomology theory
envisaged by Grothendieck and some deep conjectures in algebraic and arithmetic geometry
(see \cite[$\mathsection3$]{Se94}).

The same Mumford-Tate type question 
can also be asked for projective smooth varieties $Y$ defined over a global field $K$ of characteristic $p>0$.
Since $V_\ell=H^w(Y_{\overline K},\Q_\ell)$ are \emph{Weil cohomology theories} for $Y_{\overline K}$ only 
when $\ell\neq p$\footnote{When $\ell=p$, one has to consider crystalline cohomology group of $Y$.},
one may ask if the algebraic monodromy representations $\bG_\ell^\circ\hookrightarrow\GL_{V_\ell}$
of the Galois representations $V_\ell$ are independent of $\ell$
for all $\ell\neq p$. 
This expectation is supported by the philosophy of motives (see \cite[$\mathsection$E]{Dr18}).
On the other hand, one can always exploit the fact that the 
system of $\ell$-adic Galois representations $\{V_\ell=H^w(Y_{\overline K},\Q_\ell)\}_\ell$
is a $\Q$-\emph{compatible system} (in the sense of Serre \cite[Chap. I-11 Definition]{Se98}) 
that is \emph{pure of weight} $w$ (proven by Deligne \cite{De74,De80}) to directly argue
$\ell$-independence of the algebraic monodromy representations $\bG_\ell^\circ\hookrightarrow\GL_{V_\ell}$. 
This approach holds,
regardless of the characteristic of the global field $K$.
By utilizing the compatibility and weight conditions of the compatible system, Serre developed the method of \emph{Frobenius tori} \cite{Se81} 
to prove the $\ell$-independence result below (Theorem \ref{thm0}). 

Let us define some notation first.
If $L$ is a subfield of $\overline \Q$, then denote by  $\mathcal{P}_{L}$ the set of places of $L$.
Denote by $\mathcal{P}_{L,f}$ (resp. $\mathcal{P}_{L,\infty}$) the 
set of finite (resp. infinite) places of $L$. 
Then $\mathcal{P}_{L}=\mathcal{P}_{L,f}\cup \mathcal{P}_{L,\infty}$.
Denote by $\mathcal{P}_{L,f}^{(p)}$ the set of elements of $\mathcal{P}_{L,f}$ not extending $p$.
The residue characteristic of the finite place $v\in \mathcal{P}_{L,f}$ is denoted by $p_v$.
Let $V$ and $W$ be free modules of finite rank over a ring $R$. Let 
$\bG_m\subset\cdots\subset\bG_1\subset\GL_V$ and $\bH_m\subset\cdots\subset\bH_1\subset\GL_W$
be two chain of closed algebraic subgroups over $R$. We say that the two chain representations 
(or simply representations if it is clear that they are chains of subgroups of some $\GL_n$) are isomorphic 
if there is an $R$-modules isomorphism $V\cong W$
such that the induced isomorphism $\GL_V\cong\GL_W$ maps $\bG_i$ isomorphically onto $\bH_i$ for $1\leq i\leq m$.

\begin{customthm}{A}(Serre)\label{thm0} \cite{Se81} (see also \cite{LP97})
\begin{enumerate}[(i)]
\item (The component groups) The quotient groups $\bG_\ell/\bG_\ell^\circ$ for all $\ell$ are isomorphic.
\item (Common $\Q$-form of formal characters) 
For all $v$ in a positive Dirichlet density subset of $\mathcal{P}_{K,f}$,
there exist a subtorus $\bT:=\bT_{v}$ of $\GL_{n,\Q}$ 
such that for all $\ell\neq p_v$,
 the representation $(\bT\hookrightarrow \GL_{n,\Q})\times_\Q\Q_\ell$ is isomorphic to 
the  representation $\bT_\ell\hookrightarrow\GL_{V_\ell}$ for some maximal torus $\bT_\ell$ of $\bG_\ell$. 
\end{enumerate}
\end{customthm}

\noindent It follows immediately that the connectedness and the absolute rank of $\bG_\ell$ are
both independent of $\ell$.
Later, Larsen-Pink obtained some $\ell$-independence results
for abstract semisimple compatible system on a Dirichlet density one set of primes $\ell$ \cite{LP92} and for the
geometric monodromy of $\{V_\ell\}_\ell$ if $\mathrm{Char}(K)>0$ \cite{LP95}. 
When $\mathrm{Char}(K)=0$, the author proved that the formal 
bi-character (Definition \ref{fbc}(ii)) of 
$\bG_\ell^\circ\hookrightarrow\GL_{V_\ell}$ is independent of $\ell$ and obtained $\ell$-independence of 
$\bG_\ell^\circ$ under some type A hypothesis \cite{Hu13,Hu18}.
The next result is by far the best result in positive characteristic, in
a setting more general than the above \'etale cohomology case.

Let $X$ be a smooth geometrically connected variety defined 
over a finite field $\F_q$ of characteristic $p$.
Let $E$ be a number field.
For any $\lambda\in\mathcal{P}_{E}$, denote by $E_\lambda$ the $\lambda$-adic completion of $E$. 
Let 
\begin{equation}\label{rhobullet}
\rho_\bullet:=\{\rho_\lambda:
\pi_1^{\et}(X,\overline x)\to\GL(V_\lambda)\}_{\lambda\in\mathcal{P}_{E,f}^{(p)}}
\end{equation}
be an $E$-compatible 
system of $n$-dimensional semisimple $\lambda$-adic representations of the 
\emph{\'etale fundamental group} $\pi_1^{\et}(X,\overline x)$ of $X$ (with base point $\overline x$)
that is pure of integral weight $w$. Denote by $\bG_\lambda\subset\GL_{V_\lambda}$ 
the algebraic monodromy group of the representation $V_\lambda$.
For simplicity, set $\pi_1(X)=\pi_1^{\et}(X,\overline x)$ and 
for all $\lambda\in\mathcal{P}_{E,f}^{(p)}$, choose coordinates for $V_\lambda$ so that $\bG_\lambda$ 
is identified as a subgroup of $\GL_{n,E_\lambda}$. The following theorem was obtained by 
Chin when $X$ is a curve \cite{Ch04}\footnote{\cite{Ch04} used pivotally Serre's Frobenius tori 
and Lafforgue's work \cite{La02} on the Langlands' conjectures. In case $X$ is a curve, 
Theorem \ref{Chin}(i),(ii),(iii) 
follow, respectively, from Lemma 6.4, Thm. 1.4, Thm. 6.8 and Cor. 6.9 of the paper.
} and is true in general 
by reducing to the curve case by finding a suitable curve $S$ in some covering 
$X'$ of $X$ \cite[$\mathsection3.3$]{BGP19}, see also \cite[$\mathsection$4.3]{D'Ad20}.

\begin{customthm}{B}\label{Chin}
Let $\rho_\bullet$ be an $E$-compatible system of $n$-dimensional $\lambda$-adic semisimple 
representations of $\pi_1(X)$ that is 
pure of integer weight $w$. The following assertions hold in some coordinates of $V_\lambda$.
\begin{enumerate}[(i)]
\item (Common $E$-form of formal characters): There exists a subtorus $\bT$ of $\GL_{n,E}$ such that for all 
$\lambda\in\mathcal{P}_{E,f}^{(p)}$,
 $\bT_\lambda:=\bT\times_E E_\lambda$ is a maximal torus  of $\bG_\lambda$.
\item ($\lambda$-independence over an extension): 
There exist a finite extension $F$ of $E$ and a chain of subgroups $\bT^{\sp}\subset\bG^{\sp}\subset\GL_{n,F}$
such that $\bG^{\sp}$ is connected split reductive, $\bT^{\sp}$ is a split maximal torus of $\bG^{\sp}$, and for all 
$\lambda\in\mathcal{P}_{E,f}^{(p)}$, 
if $F_\lambda$ is a completion of $F$ extending $\lambda$ on $E$, then there exists an isomorphism of chain representations:
\begin{equation*}
f_{F_\lambda} :(\bT^{\sp}\subset\bG^{\sp}\hookrightarrow\GL_{n,F})\times_F F_\lambda\stackrel{\cong}{\rightarrow} 
(\bT_{\lambda}\subset\bG_\lambda^\circ\hookrightarrow\GL_{n, E_\lambda})\times_{E_\lambda} F_\lambda.
\end{equation*}
\item (Rigidity) The isomorphisms $f_{F_\lambda}$ in (ii) can be chosen such that the restriction isomorphisms $f_{F_\lambda}:\bT^{\sp}\times_F F_\lambda\to\bT_{\lambda}\times_{E_\lambda} F_\lambda$ admit a common $F$-form $f_F:\bT^{\sp}\to \bT\times_E F$ 
for all $\lambda\in\mathcal{P}_{E,f}^{(p)}$ and $F_\lambda$.
\end{enumerate}
\end{customthm}

\subsection{The results of the paper}
\subsubsection{Characteristic $p$}
\paragraph{} Theorem \ref{Chin}(ii) asserts that the algebraic monodromy 
representations $\bG_\lambda^\circ\hookrightarrow\GL_{n, E_\lambda}$ have a common (split) $F$-model
after finite extensions $F_\lambda$ of $E_\lambda$. 
The main theme of this article is to remove these extensions.
Base on Theorem \ref{Chin}(i)--(iii) and some ideas seeded in \cite{Hu18}, 
we prove the following $E$-rationality result (Theorem \ref{thm1}).
In case the representations $V_\lambda$ are absolutely 
irreducible\footnote{In general, 
we expect a common $E$-form of the faithful representations
$\bG_\lambda\hookrightarrow\GL_{V_\lambda}$ for all $\lambda\in\mathcal{P}_{E,f}^{(p)}$ exists.}, 
it answers the Mumford-Tate type question in positive characteristic.

\begin{thm}\label{thm1}
Let $\rho_\bullet:=\{\rho_\lambda:
\pi_1(X)\to\GL(V_\lambda)\}_{\lambda\in\mathcal{P}_{E,f}^{(p)}}$ be an $E$-compatible system of $n$-dimensional $\lambda$-adic semisimple 
representations of $\pi_1(X)$ that is 
pure of integer weight $w$.
Then the following assertions hold.
\begin{enumerate}[(i)]
\item There exists a connected reductive group $\bG$ defined over $E$ such that $\bG\times_E E_\lambda\cong \bG_\lambda^\circ$ for all $\lambda\in\mathcal{P}_{E,f}^{(p)}$.
\item If moreover $\bG_\lambda^\circ\hookrightarrow \GL_{V_\lambda}$ is absolutely irreducible for some $\lambda$, 
then there exists
a connected reductive subgroup $\bG$ of $\GL_{n,E}$
such that for all $\lambda\in\mathcal{P}_{E,f}^{(p)}$, the representations are isomorphic: 
\begin{equation*}
(\bG\hookrightarrow \GL_{n,E})\times_E E_\lambda\cong (\bG_\lambda^\circ\hookrightarrow\GL_{V_\lambda}).
\end{equation*}
\end{enumerate}
\end{thm}

\paragraph{} 
Let $\mathcal{O}_\lambda$ be the ring of integers of $E_\lambda$, 
 $\mathcal{O}_E$ be the ring of integers of $E$, $\mathcal{O}_{E,S}$ be the localization
for some finite subset $S\subset \mathcal{P}_{E,f}$, and $\A_E^{(p)}$ be the adele ring of $E$ without factors above $p$.
We construct an adelic representation $\rho_\A^{\bG}$ in Corollary \ref{cor1} and in the absolutely irreducible case,
a common model $\mathcal{G}\subset \GL_{n,\mathcal{O}_{E,S}}$ of the group schemes 
 $\mathcal{G}_\lambda\hookrightarrow \GL_{n,\mathcal{O}_\lambda}$ (with respect to some $\mathcal O_\lambda$-lattice in $V_\lambda$) 
for all but finitely many $\lambda$ in Corollary \ref{cor1.5}.

\begin{cor}\label{cor1}
Let $\rho_\bullet$ be a $\lambda$-adic compatible system of $\pi_1(X)$ as above.
Suppose $\bG_\lambda$ is connected for all $\lambda$. Then the following assertions hold.
\begin{enumerate}[(i)]
\item There exist a connected reductive group $\bG$ defined over $E$ and an isomorphism 
$\bG\times_E E_\lambda \stackrel{\phi_\lambda}{\rightarrow}\bG_\lambda$ for each $\lambda\in\mathcal{P}_{E,f}^{(p)}$
such that the direct product representation 
$$\prod_{\lambda\in\mathcal{P}_{E,f}^{(p)}}\rho_\lambda:\pi_1(X)\to \prod_{\lambda\in\mathcal{P}_{E,f}^{(p)}}\bG_\lambda(E_\lambda)$$
 factors through a $\bG$-valued adelic representation via $\phi_\lambda$:
$$\rho_{\A}^{\bG}:\pi_1(X)\to \bG(\A_E^{(p)}).$$
\item If the representations $V_\lambda$ are absolutely irreducible, then there exist
a connected reductive subgroup $\bG$ of $\GL_{n,E}$ and an isomorphism of representations 
$(\bG\hookrightarrow \GL_{n,E})\times_E E_\lambda \stackrel{\phi_\lambda}{\rightarrow}(\bG_\lambda\hookrightarrow\GL_{V_\lambda}$) 
for each $\lambda\in\mathcal{P}_{E,f}^{(p)}$
such that the direct product representation 
$$\prod_{\lambda\in\mathcal{P}_{E,f}^{(p)}}\rho_\lambda:
\pi_1(X)\to \prod_{\lambda\in\mathcal{P}_{E,f}^{(p)}}\bG_\lambda(E_\lambda)
\subset\prod_{\lambda\in\mathcal{P}_{E,f}^{(p)}}\GL_n(E_\lambda)$$ factors through 
a $\bG$-valued adelic representation via $\phi_\lambda$:
$$\rho_{\A}^{\bG}:\pi_1(X)\to \bG(\A_E^{(p)})\subset \GL_{n,E}(\A_E^{(p)}).$$
\end{enumerate}
\end{cor}

\begin{cor}\label{cor1.5}
Let $\rho_\bullet$ be a $\lambda$-adic compatible system of $\pi_1(X)$ as above.
Suppose $V_\lambda$ is absolutely irreducible and $\bG_\lambda$ is connected for all $\lambda$. 
Then there exists a smooth reductive group scheme $\mathcal{G}\subset \GL_{n,\mathcal{O}_{E,S}}$
defined over $\mathcal{O}_{E,S}$ (for some finite $S$) whose generic fiber 
is $\bG\subset\GL_{n,E}$ such that for all $\lambda\in\mathcal{P}_{E,f}^{(p)}\backslash S$,
 the representations $(\mathcal{G}\hookrightarrow \GL_{n,\mathcal{O}_{E,S}})\times\mathcal{O}_\lambda$
and $\mathcal{G}_\lambda\hookrightarrow \GL_{n,\mathcal{O}_\lambda}$
are isomorphic, where $\mathcal G_\lambda$
is the Zariski closure of $\rho_\lambda(\pi_1(X))$ in $\GL_{n,\mathcal O_\lambda}$
after some choice of $\mathcal{O}_\lambda$-lattice in $V_\lambda$. 
\end{cor}

\vspace{.1in}
\noindent
For almost all $\lambda$, $\bG(\mathcal{O}_\lambda)$ is a \emph{hyperspecial maximal compact subgroup} 
of $\bG(E_\lambda)$ \cite[$\mathsection3.9.1$]{Ti79}. Hence, 
Corollary \ref{cor1}(i) implies that for almost all $\lambda$, the image $\rho_\lambda(\pi_1(X))$
is contained in hyperspecial maximal compact 
subgroup of $\bG_\lambda(E_\lambda)$ (see Proposition \ref{hyper}). 
Next corollary is about the \emph{$\bG$-valued compatibility} of the system,
motivated by the papers \cite{BHKT19},\cite{Dr18} on Langlands conjectures.
As obtained in \cite[$\mathsection6$]{BHKT19}, the results in \cite[$\mathsection4$]{D'Ad20} 
(\cite[$\mathsection6$]{Ch04} when $X$ is a curve) imply that 
the $E$-compatible system $\rho_\bullet$ (assume connectedness of $\bG_\lambda$), after some finite extension $F/E$, 
factors through an $F$-compatible system  $\rho_\bullet^{\bG^{\sp}}$ of 
$\bG^{\sp}$-representations for some connected split reductive group $\bG^{\sp}$ defined over $F$.
In some situation, we prove that the extension $F/E$ can be omitted.
This shows evidence to the \emph{motivic hope} in \cite[$\mathsection$E]{Dr18}
that the Tannakian categories $\mathcal{T}_\lambda(X)$ of semisimple (weight $0$) $E_\lambda$-representations of $\pi_1(X)$,
at least for all $\lambda$ not extending $p$, should come from a canonical category $\mathcal{T}(X)$ over $E$ 
$$\mathcal{T}(X)\otimes_E E_\lambda\stackrel{\approx}{\rightarrow}\mathcal{T}_\lambda(X)$$
in a compatible way (see \cite[Thm. 1.4.1]{Dr18}).
The definition of an $E$-compatible system of $\bG$-representations will be recalled in $\mathsection3.2$.
Let $\pi_\lambda:\A_E^{(p)}\to E_\lambda$ be the natural surjection to the $\lambda$-component.

\begin{cor}\label{cor2}
Let 
$\rho_\bullet$ be a $\lambda$-adic compatible system of $\pi_1(X)$ as above.
Suppose $V_\lambda$ is absolutely irreducible and $\bG_\lambda$ is connected for all $\lambda$. 
Let $\bG\hookrightarrow\GL_{n,E}$ be the $E$-form and $\rho_\lambda^{\bG}$ be the adelic representation in Corollary \ref{cor1}(ii). 
Let $N_{\GL_{n,E}}\bG$ the normalizer of $\bG$ in $\GL_{n,E}$. 
Then for each $\lambda\in \mathcal{P}_{E,f}^{(p)}$, there exists (a change of coordinates)
$\beta_\lambda\in (N_{\GL_{n,E}}\bG)(E_\lambda)$ such that the system 
$$\rho_\bullet^{\bG}:=\{\rho_\lambda^{\bG}:\pi_1(X)\stackrel{\rho_\A^{\bG}}{\rightarrow} 
\bG(\A_E^{(p)})\stackrel{\pi_\lambda}{\rightarrow} \bG(E_\lambda)
\stackrel{\beta_\lambda}{\rightarrow} \bG(E_\lambda)\}_{\lambda\in\mathcal{P}_{E,f}^{(p)}}$$ 
is an $E$-compatible system of $\bG$-representations when one of the following holds.
\begin{enumerate}[(i)]
\item The group $\bG_\lambda$ is split for all $\lambda$.
\item The outer automorphism group of the derived group $\bG^{\der}\times_E\overline E$ is trivial ($\beta_\lambda=id$).
\end{enumerate}
Hence, for any $E$-representation $\alpha:\bG\to\GL_{m,E}$, the system of $m$-dimensional $\lambda$-adic semisimple representations 
$\{\alpha\circ\rho_\lambda^{\bG}\}_{\lambda\in\mathcal{P}_{E,f}^{(p)}}$ is also $E$-compatible.
\end{cor}

\paragraph{} 
Denote by $\mathcal{P}_{E,p}$ the set 
of finite places of $E$ extending $p$.
Let $\Q_{p^k}$ be a degree $k$ unramified extension of $\Q_p$, $v\in \mathcal{P}_{E,p}$, and 
$E_{v,p^k}$ the composed fields $E_v\cdot\Q_{p^k}$. Let  $\rho_\bullet$ be in Theorem \ref{thm1}.
The \emph{semisimple crystalline companion object}  of $\rho_\bullet$ at $v$ 
(whose existence\footnote{Recent works of Kedlaya \cite{Ke22b,Ke22c} establish the existence of crystalline companion
when $X$ is smooth.} is conjectured by Deligne \cite[Conjecture 1.2.10]{De80})
is an object $M_v$ in the Tannakian category $\mathrm{\textbf{F-Isoc}}^{\dagger}(X)\otimes E_{v,q}$ 
of \emph{overconvergent $F$-isocrystals of $X$ with coefficients in $E_{v,q}$} (see \cite[$\mathsection2$]{Ke22a} for definition). 
Any $t\in X(\F_{q^k})$ induces 
a fiber functor to the category of vector spaces over $E_{v,q^k}$ given by the composition
$$w_t: \mathrm{\textbf{F-Isoc}}^{\dagger}(X)\otimes E_{v,q} 
\to \mathrm{\textbf{F-Isoc}}^{\dagger}(x)\otimes E_{v,q^k}\to \mathrm{\textbf{Vec}}_{E_{v,q^k}},$$
where the first one is via the pull-back of $i:t\to X$ and the second one is the forgetful functor.
The image $V_{t,v}:=w_t(M_v)$ is an $n$-dimensional vector space.
The Tannakian group of the subcategory generated by $M_v$ with respect to $w_t$
can be identified as a reductive subgroup $\bG_{t,v}\subset\GL_{V_{t,v}}\cong\GL_{n,E_{v,q^k}}$
and is called the algebraic monodromy group of $(M_v,w_t)$.
For different closed points $t$ and $t'$ in $X(\F_{q^k})$, $\bG_{t,v}$ and $\bG_{t',v}$ differ by an inner twist \cite[Theorem 3.2]{DS82}. 
Let $\lambda$ be a finite place of $E$ not extending $p$.
The absolute root data of $\bG_{t,v}^\circ$ and $\bG_\lambda^\circ$ 
(resp. the component groups of $\bG_{t,v}$ and $\bG_\lambda$) are proven to be isomorphic independently by
Pal \cite{Pa15}  and D'Addezio \cite{D'Ad20} (relying on \cite{La02} and \cite{Ab18}). 
Moreover, given the closed point $t$ one can define the \emph{Frobenius tori} $\bT_{t,v}$ 
in $\bG_{t,v}$ (see \cite[$\mathsection4.2$]{D'Ad20}) and 
$\bT_{\bar t,\lambda}$ in $\bG_\lambda$ (up to conjugation, see $\mathsection3.3$).
Assume the crystalline companions of $\rho_\bullet$ exist for all $v\in\mathcal{P}_{E,p}$ and certain conditions, 
we prove an $E$-rationality result (existence + uniqueness) for 
the above algebraic monodromy groups at all finite places of $E$.

\begin{thm}\label{thmcrys}
Let $\rho_\bullet:=\{\rho_\lambda:\pi_1(X)\to\GL(V_\lambda)\}_{\lambda\in\mathcal{P}_{E,f}^{(p)}}$ 
be an $E$-compatible system of $n$-dimensional $\lambda$-adic semisimple representations of $\pi_1(X)$ that is pure of integer weight $w$
and $t\in X(\F_{q^k})$ a closed point of $X$.
Suppose the semisimple crystalline companion object $M_v$ of $\rho_\bullet$ exists in $\mathrm{\textbf{F-Isoc}}^{\dagger}(X)\otimes E_{v,q}$ for each $v\in \mathcal{P}_{E,p}$
and the following conditions hold.
\begin{enumerate}[(a)]
\item The Frobenius torus $\bT_{\bar t,\lambda}$ is a maximal torus of $\bG_\lambda$ for some $\lambda$.
\item For all $v\in \mathcal{P}_{E,p}$, the field $\Q_{q^k}$ is contained in $E_v$.
\item The number field $E$ has at least one real place\footnote{This condition is 
needed to ensure that the $E$-torus in Main Theorem \ref{meta2}(d) is anisotropic at some place $v$ of $E$.}.
\end{enumerate} 
Then the following assertions hold.
\begin{enumerate}[(i)]
\item There exists a chain (of a connected reductive group together with a maximal torus) $\bT\subset\bG$ 
defined over $E$ that is the unique common $E$-form of the chains
$\bT_{\bar t,\lambda}\subset\bG_\lambda^\circ$ 
for all $\lambda\in\mathcal{P}_{E,f}^{(p)}$
and the chains $\bT_{t,v}\subset\bG_{t,v}^\circ$ for all $v\in\mathcal{P}_{E,p}$.
\item If moreover $\bG_\lambda^\circ\hookrightarrow \GL_{V_\lambda}$ is absolutely irreducible for some $\lambda$, 
 then there exist
an inner form $\GL_{m,D}$ (for some division algebra $D$ over $E$) of $\GL_{n,E}$ over $E$ containing a chain of subgroups 
$\bT\subset\bG$ such that $\bT\subset\bG\hookrightarrow \GL_{m,D}$ 
is the unique common $E$-form of the chain representations $\bT_{\bar t,\lambda}\subset\bG_\lambda\hookrightarrow\GL_{V_\lambda}$
for all $\lambda\in\mathcal{P}_{E,f}^{(p)}$ and the chain representations 
$\bT_{t,v}\subset\bG_{t,v}\hookrightarrow\GL_{V_{t,v}}$
for all $v\in\mathcal{P}_{E,p}$.
When $E$ has exactly one real place, we have $\GL_{m,D}\cong\GL_{n,E}$.
\end{enumerate}
\end{thm}

\subsubsection{Characteristic zero}
\paragraph{}
It turns out that the strategy for proving Theorem \ref{thm1} retains
in characteristic zero if ordinary representations enter the picture. 
This part is influenced by the work of Pink \cite{Pi98}.
To keep things simple, we only consider the $\Q$-compatible system (with exceptional set $S$) 
of $n$-dimensional $\ell$-adic Galois representations
$V_\ell:=H^w(Y_{\overline K},\Q_\ell)$:
\begin{equation}\label{cs}
\rho_\bullet:=\{\rho_\ell:\Gal(\overline K/K)\to \GL(V_\ell)\}_{\ell\in\mathcal{P}_{\Q,f}},
\end{equation}
arising from a smooth projective variety $Y$ defined over a number field $K$.
The set $S$ consists of the finite places of $K$ such that $Y$ does not have good reduction.
Let $\bG_\ell$ be the algebraic monodromy group at $\ell$.
The Grothendieck-Serre semisimplicity conjecture asserts that 
the representation $\rho_\ell$ is semisimple (see \cite{Tat65}), which is equivalent to the algebraic group
$\bG_\ell^\circ$ being reductive.
Choose coordinates for $V_\ell$ and identify $\bG_\ell$ as a subgroup of $\GL_{n,\Q_\ell}$ for all $\ell$.
Embed $\Q_\ell$ into $\C$ for all $\ell$.

Let $v\in\mathcal{P}_{K,f}\backslash S$ with  $p:=p_v$.
Let $K_v$ be the completion of $K$ at $v$, $\mathcal{O}_v$ the ring of integers, 
and $Y_v$ the special fiber of a smooth model of $Y$ over $\mathcal O_v$.
The local representation $V_p=H^w(Y_{\overline K},\Q_p)$ of $\Gal(\overline{K}_v/K_v)$ is \emph{crystalline} 
and corresponds, via a mysterious functor of Fontaine \cite{Fo79,Fo82,Fo83}, to the crystalline cohomology 
$M_v:=H^w(Y_v/\mathcal O_v)\otimes_{\mathcal O_v}K_v$ \cite{FM87}, \cite{Fa89}.
The local representation $V_p$ is said to be \emph{ordinary} if the Newton and Hodge polygons of $M_v$ coincide \cite{Ma72}.
This notion originates from ordinary abelian varieties defined over finite fields.
It is conjectured by Serre that if $K$ is large enough, then
the set of places $v$ in $\mathcal{P}_{K,f}$
for which the local representations $V_p$ are ordinary
is of Dirichlet density one,
for abelian varieties of low dimensions, see Serre \cite{Se98}, Ogus \cite{O82}, Noot \cite{No95,No00}, Tankeev \cite{Ta99};
for abelian varieties in general, see Pink \cite{Pi98}; and for K3 surfaces, see Bogomolov-Zarhin \cite{BZ09}.

\begin{thm}\label{thm2}
Let $\rho_\bullet$ be the $\Q$-compatible system \eqref{cs} arising from 
the  $\ell$-adic cohomology (of degree $w$) of a smooth projective variety $Y$
defined over a number field $K$. 
Suppose $\bG_\ell$ is connected for all $\ell$ and the following conditions hold.
\begin{enumerate}[(a)]
\item (Ordinariness): The set of places $v$ in $\mathcal{P}_{K,f}$
for which the local representations $V_p$ of $\Gal(\overline{K}_v/K_v)$ are ordinary is of positive Dirichlet density.
\item ($\ell$-independence absolutely): 
There exists a connected reductive subgroup $\bG_\C$ of $\GL_{n,\C}$ such that 
the representations $\bG_\C\hookrightarrow\GL_{n,\C}$ and 
$(\bG_\ell\hookrightarrow\GL_{n,\Q_\ell})\times_{\Q_\ell}\C$ are isomorphic for all $\ell$.
\item (Invariance of roots): Let $\bT_\C^{\ss}$ be a maximal torus of the derived group $\bG_\C^{\der}$. Then the normalizer
$N_{\GL_{n,\C}}(\bT_\C^{\ss})$ is invariant on the roots of $\bG_\C^{\der}$ with respect to $\bT_\C^{\ss}$.
\end{enumerate}
Then the following assertions hold.
\begin{enumerate}[(i)]
\item There exists a connected reductive group $\bG$ defined over $\Q$ such that 
$\bG\times_\Q \Q_\ell\cong \bG_\ell$ for all $\ell$. In particular, $\bG_\ell$ is unramified for $\ell\gg0$.
\item If moreover $\bG_\C$ is irreducible on $\C^n$, then there exists a connected reductive subgroup $\bG$ of $\GL_{n,\Q}$ 
such that $\bG\hookrightarrow \GL_{n,\Q}$ is a common $\Q$-form of the representations $\bG_\ell\hookrightarrow\GL_{n,\Q_\ell}$
for all $\ell$.
\end{enumerate}
\end{thm}

\begin{remark}
The conditions \ref{thm2}(a),(b),(c) are to be compared with Theorem \ref{Chin}(i),(ii),(iii).
Since Theorem \ref{thm0}(ii) only gives a common $\Q$-form of formal characters for all but one $\ell$,
the condition (a) is needed if one aims at a $\Q$-common form for all $\ell$. Given \ref{thm2}(a) and \ref{Chin}(i),
then \ref{thm2}(b) and \ref{Chin}(ii) are easily seen to be equivalent ($E=\Q$).
The rigidity assertion \ref{Chin}(iii) is not known to hold in characteristic zero, 
and is now replaced with the invariance of roots condition \ref{thm2}(c), which 
holds if $\bG_\C^{\der}$ is of certain root system \cite[Thm. A1, A2]{Hu20}. 
\end{remark}

\begin{remark}
If $\rho_\ell$ is abelian at one $\ell$, then the rationality of $\bG_\ell\hookrightarrow \GL_{n,\Q_\ell}$ 
for all $\ell$ is obtained by Serre
via Serre group $\bS_{\mathfrak{m}}$ \cite{Se98}.
\end{remark}

\paragraph{}
Suppose $\bG_\ell$ is connected reductive for all $\ell\in\mathcal{P}_{\Q,f}$.\\

\noindent\textbf{Hypothesis H.} For $\ell\gg0$, the image of $\rho_\ell$ is contained in 
a hyperspecial maximal compact subgroup of $\bG_\ell(\Q_\ell)$.\\

\noindent This hypothesis follows from a Galois maximality conjecture of Larsen \cite{Lar95} (see Theorem \ref{QH}),
which has been established for type A representations \cite{HL16}, abelian varieties and 
hyper-\"Kahler varieties (degree $w=2$) \cite{HL20}.
Further assuming the hypothesis,
we obtain the following corollaries which are analogous to Corollaries \ref{cor1} and \ref{cor1.5}.

\begin{cor}\label{cor2.1}
Let $\rho_\bullet$ be an $\ell$-adic compatible system of $\Gal(\overline K/K)$ as above.
Suppose $\bG_\ell$ is connected for all $\ell$ and Hypothesis H holds. 
Then the following assertions hold.
\begin{enumerate}[(i)]
\item There exist a connected reductive group $\bG$ defined over $\Q$ and an isomorphism 
$\bG\times_\Q \Q_\ell \stackrel{\phi_\ell}{\rightarrow}\bG_\ell$ for each $\ell\in\mathcal{P}_{\Q,f}$
such that the direct product representation 
$$\prod_{\ell\in\mathcal{P}_{\Q,f}}\rho_\ell:\Gal(\overline K/K)\to \prod_{\ell\in\mathcal{P}_{\Q,f}}\bG_\ell(\Q_\ell)$$
 factors through a $\bG$-valued adelic representation via $\phi_\ell$:
$$\rho_{\A}^{\bG}:\Gal(\overline K/K)\to \bG(\A_\Q).$$
\item If the representations $V_\ell$ are absolutely irreducible, then there exist 
 a connected reductive subgroup $\bG\subset\GL_{n,\Q}$ and an isomorphism of representations 
$(\bG\hookrightarrow \GL_{n,\Q})\times_\Q \Q_\ell \stackrel{\phi_\ell}{\rightarrow}(\bG_\ell\hookrightarrow\GL_{V_\ell}$) 
for each $\ell\in\mathcal{P}_{\Q,f}$
such that the direct product representation 
$$\prod_{\ell\in\mathcal{P}_{\Q,f}}\rho_\ell:
\Gal(\overline K/K)\to \prod_{\ell\in\mathcal{P}_{\Q,f}}\bG_\ell(\Q_\ell)
\subset\prod_{\ell\in\mathcal{P}_{\Q,f}}\GL_n(\Q_\ell)$$ factors through 
a $\bG$-valued adelic representation via $\phi_\ell$:
$$\rho_{\A}^{\bG}:\Gal(\overline K/K)\to \bG(\A_\Q)\subset \GL_{n,\Q}(\A_\Q).$$
\end{enumerate}
\end{cor}

\begin{cor}\label{cor2.2}
Let $\rho_\bullet$ be an $\ell$-adic compatible system of $\Gal(\overline K/K)$ as above.
Suppose $V_\ell$ is absolutely irreducible, $\bG_\ell$ is connected for all $\ell$,
and Hypothesis H holds. 
Then there exists a smooth reductive group scheme $\mathcal{G}\subset \GL_{n,\Z_{S}}$
defined over $\Z_{S}$ (for some finite $S\subset\mathcal{P}_{\Q,f}$) whose generic fiber 
is $\bG\subset\GL_{n,\Q}$ such that for all $\ell\in\mathcal{P}_{\Q,f}\backslash S$,
 the representations $(\mathcal{G}\hookrightarrow \GL_{n,\Z_{S}})\times\Z_\ell$
and $\mathcal{G}_\ell\hookrightarrow \GL_{n,\Z_\ell}$
are isomorphic, where $\mathcal G_\ell$
is the Zariski closure of $\rho_\ell(\Gal(\overline K/K))$ in $\GL_{n,\Z_\ell}$
after some choice of $\Z_\ell$-lattice in $V_\ell$. 
\end{cor}

\paragraph{}
Suppose $Y=A$ is an abelian variety defined over $K$ of dimension $g$ and $w=1$.
We say that $A$ has ordinary reduction at $v$
if the local representation $V_p$ of $\Gal(\overline K_v/K_v)$ is ordinary.
The following results are due to Pink.

\begin{customthm}{C}\cite[Thm. 5.13(a),(c),(d), Thm. 7.1]{Pi98}\label{thm3}
Let $A$ be an abelian variety defined over a number field $K$ with $\End(A_{\overline K})=\Z$
and suppose $\bG_\ell$ is connected for all $\ell$.
There exists a connected reductive subgroup $\bG$ of $\GL_{2g,\Q}$ such that the following assertions hold.
\begin{enumerate}[(i)]
\item $(\bG\hookrightarrow\GL_{2g,\Q})\times\Q_\ell$ is isomorphic to $\bG_\ell\hookrightarrow\GL_{V_\ell}$ for all $\ell$
in set $\mathcal{L}$ of primes of Dirichlet density one.
\item The derived group $\bG^{\der}$ is $\Q$-simple.
\item If the root system of $\bG$ is determined uniquely by its formal character, 
i.e., if $\bG$ does not have an ambiguous factor (in Theorem \ref{LP1}), then 
we can take $\mathcal{L}$ in (i) to contain all but finitely many primes.
\item If $\bG\times_\Q\overline\Q$ does not have any type $C_r$ simple factors with $r\geq 3$,
then the abelian variety $A$ has ordinary reduction at a Dirichlet density one set of places $v$ of $K$.
\end{enumerate}
\end{customthm}

\noindent By the Tate conjecture of abelian varieties proven by Faltings \cite{Fa83}
and $\End(A_{\overline K})=\Z$, the representations $V_\ell$ are absolutely irreducible.
The $\Q_\ell$-representation $V_\ell=H^1(A_{\overline K},\Q_\ell)$ has a natural 
$\Z_\ell$-model $H^1(A_{\overline K},\Z_\ell)$.
Consider the representation $\Gal(\overline K/K)\to \GL(H^1(A_{\overline K},\Z_\ell))$ and
let $\mathcal{G}_\ell$ be the Zariski closure of the image in $\GL_{H^1(A_{\overline K},\Z_\ell)}$.
Combining the previous results, we obtain Theorem \ref{thm4} below
which extends Theorem \ref{thm3}(iii) to all $\ell$ assuming ordinariness.

\begin{thm}\label{thm4}
Let $A$ be an abelian variety defined over a number field $K$ with $\End(A_{\overline K})=\Z$
and suppose $\bG_\ell$ is connected for all $\ell$ and the following conditions hold.
\begin{enumerate}[(a)]
\item The set of places $v$ in $\mathcal{P}_{K,f}$
for which the local representations $V_p$ of $\Gal(\overline{K}_v/K_v)$ are ordinary is of positive Dirichlet density.
\item The root system of $\bG_\ell$ is determined uniquely by its formal character.
\end{enumerate}
Then there exists a smooth group subscheme $\mathcal{G}\subset\GL_{2g,\Z_S}$ over $\Z_S$ (for some finite $S\subset\mathcal{P}_{\Q,f}$)
with generic fiber $\bG\subset\GL_{2g,\Q}$ and an isomorphism of representations
$(\bG\hookrightarrow \GL_{2g,\Q})\times_\Q \Q_\ell \stackrel{\phi_\ell}{\rightarrow}(\bG_\ell\hookrightarrow\GL_{V_\ell}$) 
for each $\ell\in\mathcal{P}_{\Q,f}$
such that the direct product representation 
$$\prod_{\ell\in\mathcal{P}_{\Q,f}}\rho_\ell:
\Gal(\overline K/K)\to \prod_{\ell\in\mathcal{P}_{\Q,f}}\bG_\ell(\Q_\ell)
\subset\prod_{\ell\in\mathcal{P}_{\Q,f}}\GL_{2g}(\Q_\ell)$$ factors through 
a $\bG$-valued adelic representation via $\phi_\ell$:
$$\rho_{\A}^{\bG}:\Gal(\overline K/K)\to \bG(\A_\Q)\subset \GL_{2g,\Q}(\A_\Q).$$
Moreover, for $\ell\gg0$, the representations $(\mathcal{G}\hookrightarrow\GL_{2g,\Z_S})\times\Z_\ell$
and $\mathcal{G}_\ell\hookrightarrow \GL_{H^1(A_{\overline K},\Z_\ell)}$ are isomorphic.
\end{thm}

\begin{remark}
By Theorem \ref{thm3}(iv), Theorem \ref{LP1}, and 
the fact that for every $\ell$, 
every simple factor  of $\bG_{\ell}\times_{\Q_{\ell}}\overline\Q_{\ell}$ 
is of type A, B, C, or D \cite[Cor. 5.11]{Pi98},
the conditions \ref{thm4}(a),(b) hold if for some prime $\ell'$, 
every simple factor  of $\bG_{\ell'}\times_{\Q_{\ell'}}\overline\Q_{\ell'}$
is of type $A_r$ with $r> 1$. 
\end{remark}

\subsection{The structure of the paper}
The paper is structured on the purely algebraic main theorems \ref{meta1} and \ref{meta2} in next section.
Roughly speaking, it states that if a 
family of connected reductive algebraic subgroups $\bG_\lambda\hookrightarrow\GL_{n,E_\lambda}$
indexed by $\lambda\in\mathcal{P}_{E,f}^{(p)}$ (resp. $\mathcal{P}_{E,f}$)
satisfies some conditions, then there exist a common $E$-form 
of the family of the subgroups (resp. the representations).
The results in $\mathsection1.2$ are established in two big steps.
Firstly, we state and prove the main theorems in $\mathsection2$
which require different techniques from representation theory and Galois cohomology.
The notation and diagrams we developed in $\mathsection2$ are very much influenced by the work \cite{Hu18}.
A crucial step to the existence of a common $E$-form in the main theorem is based on 
the local-global aspects of Galois cohomology $\mathsection2.5$.
Secondly, we prove Theorems \ref{thm1}, \ref{thmcrys}, and \ref{thm2} in $\mathsection3$ by checking that
the conditions of the main theorems are satisfied for the corresponding family of algebraic monodromy groups
of the $E$-compatible systems and applying the main theorems.
For the characteristic $p$ case, 
to prove Theorem \ref{thm1} (resp. Theorem \ref{thmcrys}) by main theorem \ref{meta1} (resp. \ref{meta2}), 
the required conditions are ensured by Theorem \ref{Chin} (resp. recent work \cite{D'Ad20}, see Theorem \ref{DAd}).
The characteristic zero case is more involved. 
It requires the results of formal bi-character ($\mathsection2.2$c'-bi)
and invariance of roots to compensate for the lack of the rigidity condition \ref{Chin}(iii).
The information at the real place (Proposition \ref{anisoreal}) and a finite place (ordinary representation $V_p$) are also needed.
The other results in $\mathsection1.2$ will also be established in $\mathsection3$.
The statements that we quote are named using alphabets (e.g., Theorem \ref{thm0})
and the statements that we prove are named using numbers (e.g., Theorem \ref{thm1}).

\section{Main theorems}
\subsection{Statements}
\begin{custommthm}{I}\label{meta1}
Suppose a connected reductive subgroup $\bG_\lambda\subset\GL_{n,E_\lambda}$ is given 
for each $\lambda\in\mathcal{P}_{E,f}^{(p)}$ such that
the following conditions hold.
\begin{enumerate}[(a)]
\item (Common $E$-form of formal characters): There exists a subtorus $\bT$ of $\GL_{n,E}$ such that for all $\lambda\in\mathcal{P}_{E,f}^{(p)}$,
 $\bT_\lambda:=\bT\times_E E_\lambda$ is a maximal torus of $\bG_\lambda$.
\item ($\lambda$-independence absolutely): There exists a chain of subgroups $\bT^{\sp}\subset\bG^{\sp}\subset\GL_{n,E}$
such that $\bG^{\sp}$ is connected split reductive, $\bT^{\sp}$ is a split maximal torus of $\bG^{\sp}$, 
and for all $\lambda\in\mathcal{P}_{E,f}^{(p)}$, 
 if $\overline E_{\lambda}$ is a completion of $\overline E$ extending $\lambda$ on $E$, 
then there exists an isomorphism of chain representations:
\begin{equation*}
f_{\overline E_{\lambda}} :(\bT^{\sp}\subset\bG^{\sp}\hookrightarrow\GL_{n,E})\times_E \overline E_\lambda\stackrel{\cong}{\rightarrow} 
(\bT_{\lambda}\subset\bG_\lambda\hookrightarrow\GL_{n, E_\lambda})\times_{E_\lambda} \overline E_\lambda.
\end{equation*}
\item 
(Rigidity): The isomorphisms $f_{\overline E_{\lambda}}$ in (b) can be chosen such that the restriction isomorphisms 
$f_{\overline E_{\lambda}}:\bT^{\sp}\times_E \overline E_\lambda\to\bT_{\lambda}\times_{E_\lambda} \overline E_\lambda$
 admit a common $\overline E$-form $f_{\overline E}:\bT^{\sp}\times_E \overline E\to \bT\times_E \overline E$ 
for all $\lambda\in\mathcal{P}_{E,f}^{(p)}$ and $\overline E_{\lambda}$.
\item (Quasi-split): The groups $\bG_\lambda$ are quasi-split for all but finitely many $\lambda\in\mathcal{P}_{E,f}^{(p)}$.
\end{enumerate}
\vspace{.1in}
Then the following assertions hold.
\begin{enumerate}[(i)]
\item There exists a connected reductive group $\bG$ defined over $E$ such that 
$\bG\times_E E_\lambda\cong \bG_\lambda$ for all $\lambda\in\mathcal{P}_{E,f}^{(p)}$.
In particular, $\bG_\lambda$ is unramified for all but finitely many $\lambda$.
\item If moreover $\bG^{\sp}\hookrightarrow \GL_{n,E}$ is irreducible, then there exists
a connected reductive subgroup $\bG$ of $\GL_{n,E}$ 
such that $\bG\hookrightarrow \GL_{n,E}$ is a common $E$-form of the representations $\bG_\lambda\hookrightarrow\GL_{n,E_\lambda}$
for all $\lambda\in\mathcal{P}_{E,f}^{(p)}$.
\end{enumerate}
\end{custommthm}

\vspace{.1in}
For any $E$-algebra $B$, define $\GL_{m,B}$ to be the affine algebraic group over $E$ 
such that for any $E$-algebra $C$ the group of $C$-points is $\GL_{m}(B\otimes_E C)$.

\begin{custommthm}{II}\label{meta2}
Suppose a connected reductive subgroup $\bG_\lambda\subset\GL_{n,E_\lambda}$ is given for each $\lambda\in\mathcal{P}_{E,f}$ 
such that the following conditions hold.
\begin{enumerate}[(a)]
\item (Common $E$-form of formal characters): There exists a subtorus $\bT$ of $\GL_{n,E}$ such that for all $\lambda\in\mathcal{P}_{E,f}$,
 $\bT_\lambda:=\bT\times_E E_\lambda$ is a maximal torus of $\bG_\lambda$.
\item ($\lambda$-independence absolutely): There exists a chain of subgroups $\bT^{\sp}\subset\bG^{\sp}\subset\GL_{n,E}$
such that $\bG^{\sp}$ is connected split reductive, $\bT^{\sp}$ is a split maximal torus of $\bG^{\sp}$, 
and for all $\lambda\in\mathcal{P}_{E,f}$, 
 if $\overline E_{\lambda}$ is a completion of $\overline E$ extending $\lambda$ on $E$, 
then there exists an isomorphism of chain representations:
\begin{equation*}
f_{\overline E_{\lambda}} :(\bT^{\sp}\subset\bG^{\sp}\hookrightarrow\GL_{n,E})\times_E \overline E_\lambda\stackrel{\cong}{\rightarrow} 
(\bT_{\lambda}\subset\bG_\lambda\hookrightarrow\GL_{n, E_\lambda})\times_{E_\lambda} \overline E_\lambda.
\end{equation*}
\item  
(Rigidity): The isomorphisms $f_{\overline E_{\lambda}}$ in (b) 
can be chosen such that the restriction isomorphisms 
$f_{\overline E_{\lambda}}:\bT^{\sp}\times_E \overline E_\lambda\to\bT_{\lambda}\times_{E_\lambda} \overline E_\lambda$
 admit a common $\overline E$-form $f_{\overline E}:\bT^{\sp}\times_E \overline E\to \bT\times_E \overline E$ 
for all $\lambda\in\mathcal{P}_{E,f}$ and $\overline E_{\lambda}$.
\item (Anisotropic torus): The twisted $E$-torus ${}_\mu\! (\bT^{\sp}/\bC)$ 
is anisotropic at some place of $E$ and all real places of $E$, where 
$\bC$ is the center of $\bG^{\sp}$ and $\mu\in Z^1(E,\Aut_{\overline E}\bT^{\sp})$ the cocycle 
defined by $f_{\overline E}$ in (c). 
\end{enumerate}
\vspace{.1in}
Then the following assertions hold.
\begin{enumerate}[(i)]
\item There exists a unique connected reductive group $\bG$ defined over $E$ containing $\bT$ such that 
$(\bT\subset\bG)\times_E E_\lambda\cong (\bT_\lambda\subset\bG_\lambda)$ for all $\lambda\in\mathcal{P}_{E,f}$.
In particular, $\bG_\lambda$ is unramified for all but finitely many $\lambda$.
\item If moreover $\bG^{\sp}\hookrightarrow \GL_{n,E}$ is irreducible, then there exist
an inner form $\GL_{m,D}$ (for some division algebra $D$ over $E$) of $\GL_{n,E}$ over $E$ containing a chain of subgroups 
$\bT\subset\bG$ such that $\bT\subset\bG\hookrightarrow \GL_{m,D}$ 
is a common $E$-form of the chain representations $\bT_\lambda\subset\bG_\lambda\hookrightarrow\GL_{n,E_\lambda}$
for all $\lambda\in\mathcal{P}_{E,f}$. Such a chain of $E$-groups is unique.
\end{enumerate}
\end{custommthm}

\begin{remark}\label{simdiff} There are similarities and differences between the two main theorems.
\begin{enumerate}[(1)]
\item The index set for main theorem \ref{meta1} is $\mathcal{P}_{E,f}^{(p)}$ and for main theorem \ref{meta2} is $\mathcal{P}_{E,f}$.
\item Conditions (a), (b), (c) of the two main theorems are identical except for the index sets.
\item If we embed $\overline E_\lambda$ into $\C$ for all $\lambda$, then 
 condition (b) is equivalent to asking that the $\C$-representation $(\bG_\lambda\to \GL_{n,E_\lambda})\times_{E_\lambda}\C$ 
is independent of $\lambda$.
\item The rigidity condition (c) rigidifies the isomorphisms $f_{\overline E_\lambda}$ in (b)
by requiring them to be extensions of an $\overline E$-isomorphism $\bT^{\sp}\times_E \overline E\to \bT\times_E\overline E$
where $\bT^{\sp}$ (resp. $\bT$) is the torus in (b) (resp. (a)).
\item An $F$-torus $\bT$ is said to be anisotropic if it does not have non-trivial $F$-character.
If $F$ is a number field, $\bT$ is said to be anisotropic at a place $\lambda$ of $F$ if it is anisotropic over $F_\lambda$.
The twisted $E$-torus ${}_\mu\! (\bT^{\sp}/\bC)$ in main theorem \ref{meta2}(d) will be defined in $\mathsection\ref{cocycle}$. 
\item The conclusion of main theorem \ref{meta2} is stronger than that of main theorem \ref{meta1} as
the $E$-torus $\bT$ in condition (a) can be found in the common $E$-form $\bG$ in main theorem \ref{meta2}.
Moreover, if $E$ has only one real place, then the inner form $\GL_{m,D}$ in main theorem \ref{meta2} is equal to $\GL_{n,E}$ by class field theory.
\end{enumerate}
\end{remark}

\subsection{The rigidity condition}
The rigidity condition (c) is important for the construction of 
the $E$-form $\bG$ in the main theorems. It does not come for free.
In this section, we would like to prove that the rigidity condition follows from conditions (a),(b)
and (c') below.

\begin{enumerate}
\item[(c')] Both the following hold.\\
(c'-bi)=(Common $E$-form of formal bi-characters): 
There exists a subtorus $\bT^{\ss}$ of $\bT$ such that $\bT^{\ss}\times_E E_\lambda$ is a maximal torus
of the derived group $\bG_\lambda^{\der}$ of $\bG_\lambda$ for all $\lambda\in\mathcal{P}_{E,f}$;\\
(c'-inv)=(Invariance of roots): 
The normalizer $N_{\GL_{n,E}}(\bT^{\ssp})$ is invariant on the roots of 
the derived group $(\bG^{\sp})^{\der}$ of $\bG^{\sp}$ with respect to
the maximal torus $\bT^{\ssp}:=\bT^{\sp}\cap (\bG^{\sp})^{\der}$.
 \end{enumerate}

\subsubsection{Formal character and bi-character} \label{Formal}
Let $F$ be a field and $\bG$ a connected reductive subgroup of $\GL_{n,F}$.
If $\bT$ is a maximal torus of $\bG$, then $\bT^{\ss}:=\bT\cap \bG^{\der}$ 
is a maximal torus of the derived group $\bG^{\der}$ of $\bG$.

\begin{definition}\label{fbc}\label{fbc}\cite[Def. 2.2, 2.3]{Hu18}
\begin{enumerate}[(i)]
\item The inclusion $\bT\subset\GL_{n,F}$ is said to be a formal character of $\bG\subset\GL_{n,F}$.
\item The chain $\bT^{\ss}\subset\bT\subset\GL_{n,F}$ is said to be a formal bi-character of $\bG\subset\GL_{n,F}$.
\end{enumerate}
\end{definition}

\begin{remark}
Given a chain of subtori $\bT^{\ss}\subset\bT\subset\GL_{n,F}$, 
it is a formal bi-character of $\bG\subset\GL_{n,F}$ if and only if
$\bT\subset\GL_{n,F}$ is a formal character of $\bG\subset\GL_{n,F}$
and $\bT^{\ss}\subset\GL_{n,F}$ is a formal character of $\bG^{\der}\subset\GL_{n,F}$.
It is clear that (c'-bi) together with (a) in the main theorems mean that
there exist a chain of subtori, denoted $\bT^{\ss}\subset\bT\subset\GL_{n,E}$, such that 
$$(\bT^{\ss}\subset\bT\subset\GL_{n,E})\times_E E_\lambda$$ 
is a formal bi-character of $\bG_\lambda\subset\GL_{n,E_\lambda}$ for all $\lambda$.
\end{remark}

\begin{prop}
If conditions (a) and (b) in the main theorems hold and $\bG^{\sp}$ is irreducible on $E^n$, then (c'-bi) holds.
\end{prop}

\begin{proof}
Let $\bT\subset\GL_{n,E}$ be in (a) and let $\bT^{\ss}$ be 
the identity component of the kernel of the determinant map
$\bT\to\GL_{n,E}\stackrel{\det}{\rightarrow} \mathbb{G}_m$. 
Since $\bG_\lambda$ is connected and the representation $\bG_\lambda\subset\GL_{n,E_\lambda}$ is absolutely irreducible
for all $\lambda$ by the assumptions, $\bG_\lambda$ is either $\bG_\lambda^{\der}$ or $\bG_\lambda^{\der}\cdot\mathbb{G}_m$
by Schur's lemma.
Hence by counting dimension, $\bT^{\ss}\times_E E_\lambda$ is a maximal torus of $\bG_\lambda^{\der}$ for all $\lambda$.
\end{proof}

\subsubsection{Invariance of roots} \label{Invofroots}
Let $F$ be a field of characteristic zero and $\bG$ a connected split semisimple subgroup of $\GL_{n,F}$.
Fix a split maximal torus $\bT$ of $\bG$ and denote by $\mathbb{X}$ the character group 
of $\bT$. Let $R\subset \mathbb{X}$ be the set of \emph{roots} of $\bG$ with respect to 
$\bT$. Let $\bN:=N_{\GL_{n,F}}(\bT)$ be the normalizer of $\bT$ in $\GL_{n,F}$.
Since $\bN$ acts on $\bT$, it also acts on $\mathbb{X}$.
We would like to know when $R$ is invariant under $\bN$.
It is easy to see that this invariance of roots condition (i.e., $\bN\cdot R=R$) is independent of the choice of the
maximal torus $\bT$ and is invariant under field extension. 
So, we take $F=\C$ for simplicity. If $\bH$ is an almost simple factor of $\bG$,
then by the Cartan-Killing classification the root system of $\bH$ is one of the following:
$A_r$ $(r\geq1)$, $B_r$ $(r\geq 2)$, $C_r$ $(r\geq 3)$, $D_r$ $(r\geq 4)$, $E_6$, $E_7$, $E_8$, $F_4$, $G_2$.
We also use the convention that $C_2=B_2, D_2=A_1^2$, and $D_3=A_3$.

\paragraph{} Here are some examples for the invariance of roots condition.

\begin{customthm}{D}\label{c2-invex}\cite[Thm. 3.10]{Hu18},\cite[Thm. A2]{Hu20}
The following $\C$-connected semisimple groups $\bG$ satisfy 
the invariance of roots condition for all representations $\bG\subset\GL_{n,\C}$.
\begin{enumerate}[(a)]
\item (Hypothesis A): $\bG$ has at most one $A_4$ almost simple factor and if $\bH$ is
an almost simple factor of $\bG$, then $\bH$ is of type $A_r$ for some $r\in\N\backslash\{1,2,3,5,7,8\}$.
\item (Almost simple):  $\bG$ is almost simple of type different from $\{A_7,A_8,B_4,D_8\}$.
\end{enumerate}
\end{customthm}

Suppose $\bG$ is irreducible on the ambient space $\C^n$. If $\bG_1$ is a connected normal subgroup of $\bG$,
then there exists an unique complementary connected normal subgroup $\bG_2$ of $\bG$ such that 
the natural map $\bG_1\times\bG_2\to \bG$ is an isogeny of semisimple groups.
Moreover, there exist unique irreducible representations $V_1$ and $V_2$ of respectively $\bG_1$ and $\bG_2$ such that 
the composition representation $\bG_1\times\bG_2\to\bG\to\GL_{n,\C}$ is equal to the tensor product representation 
$(\bG_1\times\bG_2, V_1\otimes V_2)$ (see \cite{FH91}).
We say that the representation $(\bG_1,V_1)$ is a \emph{factor} of the representation $(\bG,\C^n)$.

\begin{customthm}{E}\label{LP1}(by \cite[Thm. 4]{LP90})
If $\bG,\bG'\subset\GL_{n,\C}$ are two connected semisimple subgroups with the same formal character $\bT\subset\GL_{n,\C}$ and are both irreducible on the ambient space $\C^n$. 
Then the roots $R$ and $R'$ of respectively $\bG$ and $\bG'$ (with respect to $\bT$)
are identical in $\X$ and the two representations are isomorphic unless one of the following conditions holds.
\begin{enumerate}[(a)]
\item For $r_1,...,r_m,r\in\N$ such that $r_1+\cdots+r_m=r$, the spin representation of $B_r$ is a factor of $(\bG,\C^n)$
and the tensor product of  the spin representations of $B_{r_j}$ for all $1\leq j\leq m$ is a factor of $(\bG',\C^n)$.
\item For $1\leq k\leq r-1$ and $r\geq 2$, the representation of $C_r$ (resp. $D_r$)
with highest weight $(k,k-1,..,2,1,...0)$ is a factor of $(\bG,\C^n)$ (resp. $(\bG',\C^n)$).
\item The unique dimension $27$ irreducible representation of $A_2$ (resp. $G_2$) is a factor of $(\bG,\C^n)$ (resp. $(\bG',\C^n)$).
\item Pick two out of the three unique dimension $4096=2^{12}$ irreducible representations of $C_4$, $D_4$, and $F_4$.
Then one is a factor of  $(\bG,\C^n)$ and the other one is a factor of $(\bG',\C^n)$.
\end{enumerate}
\end{customthm}

The following corollary follows directly by taking $\bG'=g\bG g^{-1}$, where $g\in\bN$.

\begin{cor}\label{LP2}
If $\bG\subset\GL_{n,\C}$ is a connected semisimple subgroup that is irreducible on the ambient space $\C^n$,
then the invariance of roots condition holds if the following conditions are satisfied.
\begin{enumerate}[(a)]
\item For $r_1,...,r_m,r\in\N$ such that $r_1+\cdots+r_m=r$,  the spin representation of $B_r$ 
and the tensor product of  the spin representations of $B_{r_j}$ for all $1\leq j\leq m$ are not both factors of $(\bG,\C^n)$.
\item For $1\leq k\leq r-1$ and $r\geq 2$, the representations of $C_r$ and $D_r$ 
with highest weight $(k,k-1,..,2,1,...0)$ are not both factors of $(\bG,\C^n)$.
\item The unique dimension $27$ irreducible representations of $A_2$ and $G_2$ are not both factors of $(\bG,\C^n)$.
\item Any two of the unique dimension $4096$ irreducible representations of $C_4$, $D_4$, and $F_4$ are not both factors of $(\bG,\C^n)$.
\end{enumerate}
\end{cor}

\paragraph{} Inspired by Theorem \ref{LP1}, we give more examples for the invariance of roots condition.

\begin{thm}\label{c2-invex2}
Suppose $\bG\subset\GL_{n,\C}$ is a connected adjoint semisimple subgroup that 
satisfies the following Lie type assumptions:
\begin{enumerate}[(a)]
\item $\bG$ does not have a factor of type $B_r$ ($r\geq 2$).
\item If $\bG$ has a factor of type $C_3$, then it cannot have a factor of type $A_3$.
\item If $\bG$ has a factor of type $C_r$, then it cannot have a factor of type $D_r$ ($r\geq 4$). 
\item If $\bG$ has a factor of type $F_4$, then it cannot have a factor of type $D_4$.
\item If $\bG$ has a factor of type $G_2$, then it cannot have a factor of type $A_2$.
\end{enumerate}
Then the invariance of roots condition holds.
\end{thm}

\begin{proof}
Let $\bG_1,...,\bG_k$ be the almost simple factors of $\bG$.
Then $\bT_i=\bG_i\cap\bT$ is a maximal torus of $\bG_i$ for all $i$.
Let $\X_i$ be the character group of $\bT_i$ and $R_i$ the roots of $\bG_i$ with respect to $\bT_i$.
Let $\Phi\subset\X$ (resp. $\Phi_i\subset\X_i$) be the subgroup (root lattice) generated by $R$ (resp. $R_i$).
One can impose a metric on the real vector space $\X_\R:=\X\otimes_\Z\R$ such that $(R, \X_\R)$
is a root system, the normalizer $\bN$ is isometric on $\X_\R$, and
the decomposition 
\begin{equation}\label{decomp}
R=\coprod_{i=1}^k R_i\subset\bigoplus_{i=1}^k \Phi_i\otimes\R=\bigoplus_{i=1}^k \X_i\otimes\R =\X_\R
\end{equation}
is orthogonal (see e.g., \cite[Appendix A]{Hu20}). 
The root subsystem $(R_i,\X_{i,\R}:=\X_i\otimes\R)$ is
irreducible for all $i$. The lemma below is needed.

\begin{lemma}
Suppose $\bG\subset\GL_{n,\C}$ is a connected semisimple subgroup that 
that satisfies the assumptions of Theorem \ref{c2-invex2}.
The following assertions are equivalent.
\begin{enumerate}[(i)]
\item $R$ is invariant under $\bN$.
\item If $g\in\bN$, then $g\cdot R\subset \Phi$.
\item If $g\in\bN$, then $g$ induces an automorphism of $\Phi$.
\end{enumerate}
\end{lemma}

\begin{proof}
(i) $\Rightarrow$ (ii): trivial.\\
(ii) $\Rightarrow$ (iii): (ii) is equivalent to $\bN\cdot\Phi\subset\Phi$. 
Since $g$ induces an automorphism of $\X$ and $\X/\Phi$ is finite, $g\cdot\Phi\subset\Phi$ implies that $g\cdot\Phi=\Phi$.\\
(iii) $\Rightarrow$ (i): The set of non-zero elements of $\Phi_i$ with the shortest length is equal 
to the set of short roots $R_i^\circ$ of $R_i$ \cite[$\mathsection4$ Lemma]{LP90}, which 
also spans $\X_i\otimes\R$.
The decomposition in \eqref{decomp} is orthogonal and $\Phi=\oplus_{i=1}^k \Phi_i$ in $\X_\R$. 
Since $g$ is isometric on $X_\R$ and induces an automorphism of $\Phi$ by (iii), $g$ permutes 
the union $R_1^\circ\cup R_2^\circ\cup \cdots\cup R_m^\circ$.
Note that $R_i^\circ= R_i$ if $R_i$ is of type $A,D,E$ and the following \cite[p.395]{LP90}:
\begin{align*}\label{Liefact}
 B_r^\circ= A_1^r\hspace{.05in}(r\geq 2),\hspace{.1in}C_3^\circ=A_3,\hspace{.1in} C_r^\circ=D_r\hspace{.05in}(r\geq 4),\hspace{.1in} F_4^\circ=D_4,\hspace{.1in} G_2^\circ=A_2.
\end{align*}
These facts and assumption (a) imply that $R_i^\circ$ remains irreducible for all $i$.
Then the orthogonality of the decomposition \eqref{decomp} 
and the fact that $g$ is isometric on $X_\R$ imply that $g$ permutes the set $\{R_1^\circ,R_2^\circ,...,R_m^\circ\}$.
Since $g$ is isometric on $X_\R$, the Lie type assumptions (a)--(e)
and the above facts about short roots imply that $R_i$ and $R_j$ ($1\leq i, j\leq m$) are of
the same type if $g\cdot R_i^\circ=R_j^\circ$.
By observing how the $R_i^\circ$ generate $R_i$ \cite[Table 1]{GOV94},
we obtain $g\cdot R_i=R_j$. Hence, $g$ actually permutes the union of roots $R_1\cup R_2\cup \cdots\cup R_m$.
By the orthogonality of the decomposition \eqref{decomp}, the fact that $g$ is isometric on $X_\R$, and induction, we conclude
that $g$ permutes $R$. 
\end{proof}

Back to the theorem, we have $\Phi=\X$ because $\bG$ is adjoint. 
Since $\X$ is invariant under $\bN$ by definition, $\Phi$ is invariant under $\bN$. 
Therefore, $R$ is invariant under $\bN$ by the lemma.
\end{proof}

\subsubsection{Conditions for rigidity}

\begin{prop}\label{c2c1}
If conditions (a), (b) in the main theorem(s) and (c') hold, then condition (c) in the main theorem(s) also holds.
\end{prop}

\begin{proof}
By (a) and (c'-bi), we have a chain of subtori $\bT^{\ss}\subset\bT\subset\GL_{n,E}$ such that for all $\lambda$,
$$\bT^{\ss}_\lambda\subset\bT_\lambda\subset\GL_{n,E_\lambda}:=(\bT^{\ss}\subset\bT\subset\GL_{n,E})\times_E E_\lambda$$ 
is a formal bi-character of $\bG_\lambda\subset\GL_{n,E_\lambda}$.
By (b), we have the field extensions diagram
\begin{displaymath}
  \xymatrix@C=2em{
    & \overline E_\lambda \ar@{-}[dr]& \\
 \overline E \ar@{-}[ur] & & {E_\lambda} \ar@{-}[dl]\\
     &  \ar@{-}[ul] {E}     }
\end{displaymath}
and a chain $\bT^{\sp}\subset\bG^{\sp}$ (over $E$) such that for all $\lambda$, there exists an 
$\overline E_\lambda$-isomorphism of representations $f_{\overline E_\lambda}$ taking 
$\bT^{\sp}\subset\bG^{\sp}$ to $\bT_\lambda\subset\bG_\lambda$ (omitting the extension field for simplicity).
This implies that $f_{\overline E_\lambda}$ maps 
$\bT^{\ssp}:=\bT^{\sp}\cap(\bG^{\sp})^{\der}$ to $\bT^{\ss}_\lambda=\bT_\lambda\cap\bG_\lambda^{\der}$
for all $\lambda$. Hence, we conclude that for all $\lambda$, the two chains
\begin{equation}\label{chain1}
\bT^{\ssp}\subset\bT^{\sp}\subset\bG^{\sp} \hspace{.2in}\text{and}\hspace{.2in} \bT^{\ss}_\lambda\subset\bT_\lambda\subset\bG_\lambda
\end{equation}
are conjugate in $\GL_n(\overline E_\lambda)$. In particular, the two $E$-chains
\begin{equation}\label{chain2}
\bT^{\ssp}\subset\bT^{\sp} \hspace{.2in}\text{and}\hspace{.2in} \bT^{\ss}\subset\bT
\end{equation}
are conjugate in $\GL_n(\overline E)$. So we choose $M\in \GL_n(\overline E)$ such that
\begin{equation}\label{chain3}
\bT^{\ssp}\subset\bT^{\sp}\hspace{.1in} = \hspace{.1in}M(\bT^{\ss}\subset\bT)M^{-1}.
\end{equation}
To finish the proof, it suffices to find for all $\lambda$, a matrix $B_\lambda\in\GL_n(\overline E_\lambda)$, such that 
the conjugation map by $B_\lambda$ takes $M\bG_\lambda M^{-1}$ to $\bG^{\sp}$
and is identity on $\bT^{\sp}=M \bT M^{-1}$. Such $B_\lambda$ exists. 
Indeed, there exists $A_\lambda\in\GL_n(\overline E_\lambda)$ such that 
\begin{equation}\label{chain4}
\bT^{\ssp}\subset\bT^{\sp}\subset\bG^{\sp}\hspace{.1in} = \hspace{.1in}A_\lambda M(\bT^{\ss}_\lambda\subset\bT_\lambda\subset\bG_\lambda)M^{-1}A_\lambda^{-1}
\end{equation}
because the chains in \eqref{chain1} are conjugate in $\GL_n(\overline E_\lambda)$.
Then \eqref{chain3} and \eqref{chain4} imply that $A_\lambda\in N_{\GL_n}(\bT^{\ssp})$ and 
conjugation by $A_\lambda$ takes the roots of  $M \bG_\lambda^{\der} M^{-1}$ to the roots of $(\bG^{\sp})^{\der}$.
By (c'-inv), the roots of the two semisimple (derived) groups are identical (in the character group of $\bT^{\ssp}$).
Hence, \cite[Thm. 3.8]{Hu18} implies that the absolute root data of $M \bG_\lambda M^{-1}$ and $\bG^{\sp}$ are identical 
with respect to the common maximal torus $M \bT_\lambda M^{-1}=\bT^{\sp}$. 
By \cite[Thm. 16.3.2]{Sp08}, there exists an $\overline E_\lambda$-isomorphism $b_\lambda$ 
taking the pair $(M \bG_\lambda M^{-1}, M \bT_\lambda M^{-1})$
to the pair $(\bG^{\sp},\bT^{\sp})$ inducing the identity map between their root data.
Let $i_1$ and $i_2$ be the tautological representation of $M \bG_\lambda M^{-1}$
and $\bG^{\sp}$ into $\GL_n$. Then the two representations $i_1$ and  $i_2\circ b_\lambda$ are isomorphic.
Therefore, $b_\lambda$ is just a conjugation by a matrix $B_\lambda\in\GL_n(\overline E_\lambda)$
that is identity on $M \bT_\lambda M^{-1}=\bT^{\sp}$.
\end{proof}

\subsection{Forms of reductive chains} \label{Chainform}
This section is foundational to the proofs of the main theorems
and is developed from \cite[$\mathsection4$]{Hu18}.

\subsubsection{Galois cohomology}\label{Galcoho}

Let $F$ be a field of characteristic zero, $\bG_1$ and $\bG'_1$ be linear algebraic groups defined over $F$.
The Galois group $\Gal(\overline F/F)$ acts (on the left) on the set of $\overline F$-homomorphisms
$\phi: \bG_1\times_F\overline F\to \bG'_1\times_F\overline F$ as follows: 
if $\sigma\in\Gal(\overline F/F)$, then ${}^\sigma\! \phi$ is the homomorphism such that
\begin{equation*}\label{gal}
{}^\sigma\! \phi(x)=\sigma(\phi(\sigma^{-1} x))\hspace{.1in}\forall x\in\bG_1(\overline F).
\end{equation*}

\noindent Let $\bG_k\subset\cdots\subset\bG_2\subset\bG_1$ be a chain of linear algebraic groups defined over $F$.
An \emph{$F$-form of the chain} $\bG_k\subset\cdots\subset\bG_2\subset\bG_1$ is
 a chain of reductive groups $\bG'_k\subset\cdots\subset\bG'_2\subset\bG'_1$ defined over $F$
that is isomorphic to $\bG_k\subset\cdots\subset\bG_2\subset\bG_1$ over $\overline F$, i.e.,
there exists a $\overline F$-homomorphism $\phi:\bG_1\times_F\overline F\to \bG'_1\times_F\overline F$
such that $\phi(\bG_i\times_F\overline F)\subset\bG'_i\times_F\overline F$ and the restriction 
$\phi|_{\bG_i\times_F\overline F}$ is an isomorphism for all $1\leq i\leq k$.
Since the groups are defined over $F$, the $\overline F$-homomorphism ${}^\sigma\! \phi$ is also a $\overline F$-isomorphism 
between the two chains.
In particular, the \emph{automorphism group $\Aut_{\overline F}(\bG_1,\bG_2,...,\bG_k)$ of the chain}  (i.e,
the subgroup of the automorphism group $\Aut_{\overline F}\bG_1$ of $\bG_1\times_F\overline F$ 
preserving the chain $\bG_k\subset\cdots\subset\bG_2\subset\bG_1$)
is a $\Gal(\overline F/F)$-group.
Let $\phi:\bG_1\times_F\overline F\to \bG'_1\times_F\overline F$ be a $\overline F$-isomorphism 
from $\bG_k\subset\cdots\subset\bG_2\subset\bG_1$ to $\bG'_k\subset\cdots\subset\bG'_2\subset\bG'_1$.
Then the association 
\begin{equation}\label{assoc}
\sigma\mapsto a_\sigma:=\phi^{-1}\circ{}^\sigma\! \phi\in \Aut_{\overline F}(\bG_1,\bG_2,...,\bG_k)
\end{equation}
for all $\sigma\in\Gal(\overline F/F)$ satisfies the $1$-cocycle condition:
$$a_{\sigma\sigma'}=a_\sigma {}^{\sigma}\! a_{\sigma'},$$
producing a bijective correspondence (see \cite[Ch. 3.1, Prop. 5 and its proof]{Se97}) between 
the set of isomorphism classes of $F$-forms of the chain $\bG_k\subset\cdots\subset\bG_2\subset\bG_1$ and 
the Galois cohomology pointed set $H^1(F,\Aut_{\overline F}(\bG_1,\bG_2,...,\bG_k))$ in which
the \emph{neutral element} is the trivial class $[a_\sigma=id]$ corresponding to
the $F$-isomorphism class of $\bG_k\subset\cdots\subset\bG_2\subset\bG_1$.

Let $\Inn_{\overline F}\bG_1$ be the inner automorphism group of $\bG_1\times_F\overline F$.
It is a ($\Gal(\overline F/F)$-) normal subgroup of $\Aut_{\overline F}\bG_1$. Denote the \emph{inner automorphism group of 
the chain} by 
$$\Inn_{\overline F}(\bG_1,\bG_2,...,\bG_k):=\Aut_{\overline F}(\bG_1,\bG_2,...,\bG_k)\cap \Inn_{\overline F}\bG_1.$$
and the \emph{outer automorphism group of the chain} by 
$$\Out_{\overline F}(\bG_1,\bG_2,...,\bG_k):=\Aut_{\overline F}(\bG_1,\bG_2,...,\bG_k)/\Inn_{\overline F}(\bG_1,\bG_2,...,\bG_k).$$
Then we obtain a short exact sequence of $\Gal(\overline F/F)$-groups
\begin{equation}\label{inout1}
1\to \Inn_{\overline F}(\bG_1,\bG_2,...,\bG_k)\to \Aut_{\overline F}(\bG_1,\bG_2,...,\bG_k)\to \Out_{\overline F}(\bG_1,\bG_2,...,\bG_k)\to 1.
\end{equation}
and an exact sequence of pointed set \cite[Ch. 1.5.5, Prop. 38]{Se97}
\begin{equation}\label{inout2}
H^1(F,\Inn_{\overline F}(\bG_1,\bG_2,...,\bG_k))\stackrel{i}{\rightarrow}
 H^1(F,\Aut_{\overline F}(\bG_1,\bG_2,...,\bG_k))\stackrel{\pi}{\rightarrow} H^1(F,\Out_{\overline F}(\bG_1,\bG_2,...,\bG_k)).
\end{equation}
The exactness means that the preimage 
$\pi^{-1}([id])$ is equal to the image $\mathrm{Im}(i)$.

An $F$-form $\bG'_k\subset\cdots\subset\bG'_2\subset\bG'_1$ of $\bG_k\subset\cdots\subset\bG_2\subset\bG_1$
is called an \emph{inner $F$-form (or inner form)} if there exists an $\overline F$-isomorphism $\phi$ such that in \eqref{assoc},
the element $a_\sigma$ belongs to $\Inn_{\overline F}(\bG_1,\bG_2,...,\bG_k)$ for all $\sigma$.
In general, the isomorphism classes of inner $F$-forms do not form a subset of 
the isomorphism classes of $F$-forms since the map $i$ in \eqref{inout2}
is not injective. However, the sequence \eqref{inout2} is a short exact sequence 
of pointed sets (and thus $i$ is injective) if \eqref{inout1} splits.
We will see in later sections that the splitting of \eqref{inout1} holds for some chains (e.g., $\bT^{\sp}\subset\bG^{\sp}$).
The following simple lemma is useful to study the conjugacy class of a subgroup in $\GL_{n,F}$.

\begin{lemma}\label{equivrepn}
Let $D$ be a central division algebra over $F$.
Let $\bU=\GL_{m,D}$ be an $F$-inner form of $\GL_{n,F}$, $\bT\subset\bG\subset\GL_{n,F}$ and 
$\bT'\subset\bG'\subset\bU$ be two chains.
If the two chains of $\overline F$-representations $(\bT\subset\bG\hookrightarrow\GL_{n,F})\times_F\overline F$ 
and $(\bT'\subset\bG'\hookrightarrow\bU)\times_F\overline F$ are isomorphic, then the following hold.
\begin{enumerate}[(i)]
\item The chain $\bT'\subset\bG'\subset\bU$ is an inner form of $\bT\subset\bG\subset\GL_{n,F}$.
\item If the cohomology class $[\bT'\subset\bG'\subset\bU]\in H^1(F,\Inn_{\overline F}(\GL_{n,F},\bG,\bT))$ is the neutral class,
then $D=F$ and the two $F$-representations $\bT\subset\bG\hookrightarrow \GL_{n,F}$ and $\bT'\subset\bG'\hookrightarrow\bU=\GL_{n,F}$ are isomorphic.
\end{enumerate}
\end{lemma}

\begin{proof}
Identify 
$\bU\times_F\overline F$ with $\GL_{n,\overline F}$.
The condition implies that there exists an $\overline F$-inner automorphism $\psi$ of $\GL_{n,\overline F}$ 
such that $\psi(\bG\times_F\overline F)=\bG'\times_F\overline F$ and $\psi(\bT\times_F\overline F)=\bT'\times_F\overline F$.
This defines a $1$-cocycle 
$$\sigma\mapsto a_\sigma:=\psi^{-1}\circ{}^\sigma\! \psi\in\Inn_{\overline F}(\GL_{n,F},\bG,\bT),$$
which proves (i). If the cocycle is neutral, then there exists $\gamma\in \Inn_{\overline F}(\GL_{n,F},\bG,\bT)\subset\PGL_n(\overline F)$
such that $a_\sigma=\gamma^{-1}\circ{}^\sigma\! \gamma$ for all $\sigma\in\Gal(\overline F/F)$.
This is equivalent to 
$$\psi\circ\gamma^{-1}={}^\sigma\! (\psi\circ\gamma^{-1})\hspace{.1in}\forall\sigma\in\Gal(\overline F/F).$$
Hence, $\psi\circ\gamma^{-1}\in\PGL_n(F)$ and $\GL_{n,F}$ and $\GL_{m,D}$ are $F$-isomorphic.
Therefore, $D=F$, $\bU=\GL_{n,F}$, and $\psi\circ\gamma^{-1}$ is an $F$-inner automorphism of $\GL_{n,F}$ taking $\bG$ to $\bG'$
as well as $\bT$ to $\bT'$,
which prove (ii).
\end{proof}

\subsubsection{Some diagrams}\label{Diag}
In this section, some diagrams of groups and Galois cohomology will be presented.
Let $F$ be a field. Denote by 

\begin{itemize}
\item $\bG^{\mathrm{sp}}$ a connected split reductive group defined over $F$,
\item $\bT^{\mathrm{sp}}$ a split maximal torus of $\bG^{\mathrm{sp}}$,
\item $\bN$ the normalizer of $\bT^{\mathrm{sp}}$ in $\bG^{\mathrm{sp}}$,
\item $W:=\bN/\bT^{\mathrm{sp}}$ the Weyl group, 
\item $\bB$ a Borel subgroup of $\bG^{\sp}$ containing $\bT^{\mathrm{sp}}$, 
\item $\bC$ the center of $\bG^{\mathrm{sp}}$,
\item $(\bG^{\sp})^{\ad}:=\bG^{\sp}/\bC$ the adjoint quotient of $\bG^{\sp}$,
\item $\Theta_{\overline F}:=\Out_{\overline F}\bG^{\sp}$ the outer automorphism group of $\bG^{\sp}$.
\item $Z^k(F,\bM):=Z^k(F,\bM(\overline F))$ the cocycles if $\bM$ is a linear algebraic group defined over $F$.
\item $H^k(F,\bM):=H^k(F,\bM(\overline F))$ the cohomology if $\bM$ is a linear algebraic group defined over $F$.
\end{itemize}

\paragraph{} Consider the following diagram of $\Gal(\overline F/F)$-groups:

\begin{equation}\label{ses1}
\begin{aligned}
\xymatrix{
1\ar[r] &\bN/\bC(\overline F) \ar@{^{(}->}[d] \ar[r]^{i\hspace{.2in}} &\Aut_{\overline F}(\bG^{\sp},\bT^{\sp}) \ar@{^{(}->}[d]_{\mathrm{Res}_{\bG^{\sp}}}\ \ar[r]^{\hspace{.3in}\pi}& \Theta_{\overline F}\ar[d]_{=}\ar[r]& 1\\
1\ar[r] & (\bG^{\sp})^{\ad}(\overline F)\ar[r]^i & \Aut_{\overline F}\bG^{\sp}\ar[r]^\pi & \Theta_{\overline F} \ar[r] &1}
\end{aligned}
\end{equation}

\vspace{.1in}
\noindent  where the top (resp. bottom) row is \eqref{inout1} for $\bT^{\sp}\subset\bG^{\sp}$ by \cite[Prop. 4.3]{Hu18} (resp. $\bG^{\sp}$)
and the vertical arrows are all natural inclusions induced by restricting automorphisms to $\bG^{\sp}$:
\begin{equation}\label{resG} 
\mathrm{Res}_{\bG^{\sp}}:\Aut_{\overline F}(\bG^{\sp},\bT^{\sp})\to \Aut_{\overline F}\bG^{\sp}.
\end{equation}
Since $\bG^{\mathrm{sp}}$ is split, the Galois group $\Gal(\overline F/F$) acts trivially on the outer automorphism group $\Theta_{\overline F}$. The proposition below is well-known.

\begin{customprop}{F}\label{split}(see e.g. \cite[Prop. 4.1]{Hu18})
The automorphism group $\Aut_{\overline F}\bG^{\sp}$ contains a $\Gal(\overline F/F)$-invariant subgroup
that preserves $\bT^{\sp}$ and $\bB$ and is mapped isomorphically onto $\Out_{\overline F}\bG^{\sp}$.
Hence, the top (resp. bottom) row in \eqref{ses1} is a split short exact sequence of $\Gal(\overline F/F)$-groups:
\vspace{-1em}
\begin{equation}\label{jsplit}
\xymatrix{1\ar[r] &\bN/\bC(\overline F) \ar[r]^{i\hspace{.2in}} &\Aut_{\overline F}(\bG^{\sp},\bT^{\sp}) \ar[r]^{\hspace{.3in}\pi}& 
\Theta_{\overline F}\ar@/_1.5pc/[l]_j \ar[r]& 1.}
\end{equation}
\end{customprop}

\vspace{.1in}
\noindent Denote by 
\begin{itemize}
\item $\Omega_{\overline F}:=\mathrm{Im}(\mathrm{Res}_{\bT^{\sp}})$, where $\mathrm{Res}_{\bT^{\sp}}$ 
restricts automorphisms to $\bT^{\sp}$:
\end{itemize}
\begin{equation}\label{resT}
\mathrm{Res}_{\bT^{\sp}}:\Aut_{\overline F}(\bG^{\sp},\bT^{\sp})\to \Aut_{\overline F}\bT^{\sp}.
\end{equation}

\noindent Then the first row in \eqref{ses1} also fits into the following diagram of $\Gal(\overline F/F)$-groups with exact rows and columns 
by \cite[Prop. 4.3]{Hu18} and $j$ denotes a splitting induced by \eqref{jsplit}.

\begin{equation}\label{ses2}
\begin{aligned}
\xymatrix{
&1\ar[d] &1\ar[d] &1\ar[d] &\\
1\ar[r] &\bT^{\mathrm{sp}}/\bC(\overline F) \ar[d] \ar[r] &\bT^{\mathrm{sp}}/\bC(\overline F) \ar[d]\ar[r] &1\ar[d]\ar[r] & 1\\
1\ar[r] &\bN/\bC(\overline F) \ar[d] \ar[r]^{i\hspace{.2in}} &\Aut_{\overline F}(\bG^{\mathrm{sp}},\bT^{\mathrm{sp}}) \ar[d]_{\mathrm{Res}_{\bT^{\sp}}} \ar[r]^{\hspace{.3in}\pi}& \Theta_{\overline F}\ar@/_1.5pc/[l]_j\ar[d]_{=}\ar[r]& 1\\
1\ar[r] & W\ar[r]^i\ar[d] & \Omega_{\overline F}\ar[r]^\pi\ar[d] & \Theta_{\overline F}\ar[d] \ar[r] &1\\
 &1&1&1&}
\end{aligned}
\end{equation}

\paragraph{}
Suppose given a faithful (absolutely) irreducible representation $\bG^{\sp}\hookrightarrow\GL_{n,F}$.
Then we have the chain $\bT^{\sp}\subset\bG^{\sp}\subset\GL_{n,F}$.
The irreducibility condition implies that $\bC$ is contained in the subgroup of scalars in $\GL_{n,F}$ and 
the following inclusions hold .
 
\begin{equation}\label{inclu1}
\begin{aligned}
\xymatrix{
\bN/\bC(\overline F)\ar@{^{(}->}[d] \ar@{^{(}->}[r ]& \Inn_{\overline F}(\GL_{n,F},\bG^{\sp},\bT^{\sp})\ar@{^{(}->}[d]_{\mathrm{Res}_{(\GL_{n,F},\bG^{\sp})}}\ar@{^{(}->}[r]&\Aut_{\overline F}(\bG^{\sp},\bT^{\sp})\ar@{^{(}->}[d]_{\mathrm{Res}_{\bG^{\sp}}}\\
(\bG^{\sp})^{\ad}(\overline F)\ar@{^{(}->}[r]&\Inn_{\overline F}(\GL_{n,F},\bG^{\sp})\ar@{^{(}->}[r] &\Aut_{\overline F}\bG^{\sp}}
\end{aligned}
\end{equation}

\noindent In diagram \eqref{ses2}, denote by
\begin{itemize}
\item $\theta_{\overline F}:=\pi(\Inn_{\overline F}(\GL_{n,F},\bG^{\sp},\bT^{\sp}))\in\Theta_{\overline F}$, 
\item $\omega_{\overline F}:=\mathrm{Res}_{\bT^{\sp}}(\Inn_{\overline F}(\GL_{n,F},\bG^{\sp},\bT^{\sp}))\in\Omega_{\overline F}$. 
\end{itemize}
\vspace{.1in}

\noindent By diagrams \eqref{ses1}, \eqref{ses2}, \eqref{inclu1} and the fact that the squares in \eqref{inclu1} 
are Cartesian, we obtain the following two diagrams with exact rows and columns.
Moreover,  \eqref{ses3} injects naturally into \eqref{ses1}, \eqref{ses4} injects naturally into \eqref{ses2},
and $j$ denotes the splitting induced by \eqref{jsplit}.

\begin{equation}\label{ses3}
\begin{aligned}
\xymatrix{
1\ar[r] &\bN/\bC(\overline F) \ar@{^{(}->}[d] \ar[r]^{i\hspace{.5in}} &\Inn_{\overline F}(\GL_{n,F},\bG^{\sp},\bT^{\sp}) \ar@{^{(}->}[d]_{\mathrm{Res}_{(\GL_{n,F},\bG^{\sp})}}\ \ar[r]^{\hspace{.5in}\pi}& \theta_{\overline F}\ar@/_1.5pc/[l]_j\ar[d]_{=}\ar[r]& 1\\
1\ar[r] & (\bG^{\sp})^{\ad}(\overline F)\ar[r]^{i\hspace{.3in}} & \Inn_{\overline F}(\GL_{n,F},\bG^{\sp})\ar[r]^{\hspace{.5in}\pi} & \theta_{\overline F}\ar@/_1.5pc/[l]_j \ar[r] &1}
\end{aligned}
\end{equation}

\vspace{.2in}
\begin{equation}\label{ses4}
\begin{aligned}
\xymatrix{
&1\ar[d] &1\ar[d] &1\ar[d] &\\
1\ar[r] &\bT^{\mathrm{sp}}/\bC(\overline F) \ar[d] \ar[r] &\bT^{\mathrm{sp}}/\bC(\overline F) \ar[d]\ar[r] &1\ar[d]\ar[r] & 1\\
1\ar[r] &\bN/\bC(\overline F) \ar[d] \ar[r]^{i\hspace{.5in}} &\Inn_{\overline F}(\GL_{n,F},\bG^{\sp},\bT^{\sp}) \ar[d]_{\mathrm{Res}_{\bT^{\sp}}} \ar[r]^{\hspace{.5in}\pi}& \theta_{\overline F}\ar@/_1.5pc/[l]_j\ar[d]_{=}\ar[r]& 1\\
1\ar[r] & W\ar[r]^i\ar[d] & \omega_{\overline F}\ar[r]^\pi\ar[d] & \theta_{\overline F}\ar[d] \ar[r] &1\\
 &1&1&1&}
\end{aligned}
\end{equation}

\vspace{.2in}
\paragraph{}
By taking Galois cohomology on diagrams \eqref{ses1},\eqref{ses2},\eqref{ses3},\eqref{ses4}, the splitting $j$, 
and Hilbert's Theorem 90: $H^1(F,\bT^{\mathrm{sp}}/\bC)=H^1(F,\mathbb{G}_m)^{\oplus k}=0$, we obtain the following diagrams of pointed sets
such that the rows and columns are all exact. Moreover, there are natural maps from \eqref{ses7} 
to \eqref{ses5}, \eqref{ses8} to \eqref{ses6}, and $j$ denotes again the splitting.

\begin{equation}\label{ses5}
\begin{aligned}
\xymatrix{
0\ar[r] & H^1(F,\bN/\bC) \ar[d] \ar[r]^{i\hspace{.4in}} & H^1(F,\Aut_{\overline F}(\bG^{\sp},\bT^{\sp})) 
\ar[d]_{\mathrm{Res}_{\bG^{\sp}}}\ \ar[r]^{\hspace{.5in}\pi}& H^1(F,\Theta_{\overline F})\ar@/_1.5pc/[l]_j\ar[d]_{=}\ar[r]& 0\\
0\ar[r] & H^1(F,(\bG^{\sp})^{\ad})\ar[r]^i & H^1(F,\Aut_{\overline F}\bG^{\sp})\ar[r]^\pi & H^1(F,\Theta_{\overline F})\ar@/_1.5pc/[l]_j \ar[r] &0}
\end{aligned}
\end{equation}

\vspace{.2in}

\begin{equation}\label{ses6}
\begin{aligned}
\xymatrix{
&0\ar[d] &0\ar[d] &0\ar[d] &\\
0\ar[r] &H^1(F,\bN/\bC) \ar[d] \ar[r]^{i\hspace{.4in}} &H^1(F,\Aut_{\overline F}(\bG^{\mathrm{sp}},\bT^{\mathrm{sp}})) \ar[d]_{\mathrm{Res}_{\bT^{\sp}}} \ar[r]^{\hspace{.5in}\pi}& H^1(F,\Theta_{\overline F})\ar@/_1.5pc/[l]_j\ar[d]_{=}\ar[r]& 0\\
0\ar[r] & H^1(F,W)\ar[r]^i & H^1(F,\Omega_{\overline F})\ar[r]^\pi & H^1(F,\Theta_{\overline F}) \ar[r] &0
}
\end{aligned}
\end{equation}

\begin{equation}\label{ses7}
\begin{aligned}
\xymatrix{
0\ar[r] &H^1(F,\bN/\bC) \ar[d] \ar[r]^{i\hspace{.7in}} &
H^1(F,\Inn_{\overline F}(\GL_{n,F},\bG^{\sp},\bT^{\sp})) \ar[d]_{\mathrm{Res}_{(\GL_{n,F},\bG^{\sp})}}\ \ar[r]^{\hspace{.7in}\pi}& 
H^1(F,\theta_{\overline F})\ar@/_1.5pc/[l]_j\ar[d]_{=}\ar[r]& 0\\
0\ar[r] & H^1(F,(\bG^{\sp})^{\ad})\ar[r]^{i\hspace{.5in}} & H^1(F,\Inn_{\overline F}(\GL_{n,F},\bG^{\sp}))\ar[r]^{\hspace{.5in}\pi} & 
H^1(F,\theta_{\overline F})\ar@/_1.5pc/[l]_j \ar[r] &0}
\end{aligned}
\end{equation}

\vspace{.2in}
\begin{equation}\label{ses8}
\begin{aligned}
\xymatrix{
&0\ar[d] &0\ar[d] &0\ar[d] &\\
0\ar[r] &H^1(F,\bN/\bC) \ar[d] \ar[r]^{i\hspace{.7in}} &
H^1(F,\Inn_{\overline F}(\GL_{n,F},\bG^{\sp},\bT^{\sp})) \ar[d]_{\mathrm{Res}_{\bT^{\sp}}} \ar[r]^{\hspace{.7in}\pi}
& H^1(F,\theta_{\overline F})\ar@/_1.5pc/[l]_j\ar[d]_{=}\ar[r]& 0\\
0\ar[r] & H^1(F,W)\ar[r]^i & H^1(F,\omega_{\overline F})\ar[r]^\pi& H^1(F,\theta_{\overline F}) \ar[r] &0
}
\end{aligned}
\end{equation}

\subsection{Twisting}\label{Twist}

Let $G$ be a profinite group and $A$ be a $G$-group (a discrete group on which $G$ acts continuously).
The Galois cohomology $H^1(G,A)$ is a pointed set with neutral element given by the trivial class $[id_A]$.
Let $1\to A\stackrel{i}{\rightarrow} B\stackrel{\pi}{\rightarrow} C\to 1$ be a 
short exact sequence of $G$-groups. Then
one obtains an exact sequence of pointed sets 
$$H^1(G,A)\stackrel{i}{\to} H^1(G,B)\stackrel{\pi}{\to} H^1(G,C),$$
meaning that the image of $i$ is equal to $\pi^{-1}([id_C])=\pi^{-1}(\pi([id_B]))$,
the fiber of $\pi([id_B])$. 
Let $[\beta]\in H^1(G,B)$ be a cohomology class. To study the image of $\pi$ as well as the fiber of $\pi([\beta])$,
that is, the set $\pi^{-1}(\pi([\beta]))$, one uses the method of \emph{twisting} in \cite[Ch. 1.5.3--1.5.7]{Se97}.
This technique will be applied to some short exact sequences in $\mathsection\ref{fiberofpi}$.

\subsubsection{Definition}
Let $G$ be a group, $M$ a (left) $G$-group, and $A$ (resp. $B$) be a $M$-group on which $G$ acts compatibly on the left,
i.e., $g(m(a))=g(m)(g(a))$ for $g\in G$, $m\in M$, and $a\in A$.
Suppose $\mu:=(m_g)\in Z^1(G,M)$ is a $1$-cocycle.
Then one can define a $G$-group ${}_\mu\! A$ \emph{twisted by} $\mu$, which can be viewed as $A$ with a new $G$-action:
as a group ${}_\mu\! A=A$ and the $G$-action is defined by 
\begin{equation}
\begin{aligned}
G\times {}_\mu\! A &\longrightarrow {}_\mu\! A\\
(g,a)&\mapsto m_g(g(a)).
\end{aligned}
\end{equation}

\noindent As $M$ acts on itself by inner automorphism (conjugation): $(-)\mapsto m(-)m^{-1}$,
denote by ${}_\mu\! M$ the twisted $G$-group. Then ${}_\mu\! A$ is a  ${}_\mu\! M$-group 
under the identification 
\begin{equation}\label{ident1}
\begin{aligned}
\xymatrix{
{}_\mu\! M\times {}_\mu\! A\ar[r]\ar[d]_{=} &{}_\mu\! A\ar[d]_{=}\\
M\times A\ar[r] & A}
\end{aligned}
\end{equation}
\noindent on which $G$ acts compatibly on the left.
If $\mu,\mu'\in Z^1(G,M)$ are cohomologous, then  ${}_\mu\! A$ and ${}_{\mu'}\! A$ are isomorphic. 
The association $A\mapsto {}_\mu\! A$ is functorial: if $f:A\to B$ is a $G$-, $M$-group homomorphism, then 
${}_\mu\!f:{}_\mu\! A\to {}_\mu\! B$ is a $G$-, ${}_\mu\!M$-group homomorphism \cite[Ch. 1.5.3]{Se97}.
Since $A$ acts on itself by inner automorphism: $A\to \Inn(A)$, it acts on $B$ via the map
$A\to B\to \Inn(B)$ such that  $A\to B$ is an $A$-group homomorphism.
The following correspondences are crucial. 

\begin{customprop}{G}\label{twistid}\cite[Ch. 1.5.3 Prop. 35 bis]{Se97}
Let $f:A\to B$ be a $G$-group homomorphism, $\alpha=(a_g)\in Z^1(G,A)$ be a cocycle, and $\beta=(b_g)\in Z^1(G,B)$ the image of $\alpha$.
Write $A'={}_\alpha\! A$, $B'={}_\beta\! B$, and $f':A'\to B'$ the map. To each cocycle $(a'_g)\in Z^1(G,A')$ (resp. $(b'_g)\in Z^1(G,B')$),
associate the cocycle $(a'_ga_g)\in Z^1(G,A)$ (resp. $(b'_gb_g)\in Z^1(G,B)$).
This induces the following commutative diagrams such that the vertical arrows are bijective correspondence 
taking neutral cocycles (resp. classes) to $\alpha,\beta$ (resp. $[\alpha],[\beta]$).
\begin{equation}\label{ident2}
\begin{aligned}
\xymatrix{
Z^1(G,A') \ar[r]^{f'}\ar[d]_{t_\alpha} & Z^1(G,B') \ar[d]_{t_\beta} &&H^1(G,A') \ar[r]^{f'}\ar[d]_{\tau_\alpha} & H^1(G,B') \ar[d]_{\tau_\beta}\\
Z^1(G,A) \ar[r]^f & Z^1(G,B)&& H^1(G,A) \ar[r]^f & H^1(G,B)
}
\end{aligned}
\end{equation}
\end{customprop}

\noindent Therefore, $\tau_\alpha: (f')^{-1}( f'([id_{A'}]))\to f^{-1} (f([\alpha]))$  is a
bijective correspondence between  the fibers of classes. 

\subsubsection{Fibers of $\pi$}\label{fiberofpi}
Given a split short exact sequence of $G$-groups:
\begin{equation}\label{ses9}
\begin{aligned}
\xymatrix{
1\ar[r] &A \ar[r]^i &
B \ar[r]^\pi
& C\ar@/_1.5pc/[l]_j\ar[r]& 1.
}
\end{aligned}
\end{equation}
Then we obtain a split short exact sequence of pointed sets:
\vspace{-1em}
\begin{equation}\label{ses10}
\begin{aligned}
\xymatrix{
0\ar[r] &H^1(G,A) \ar[r]^i &
H^1(G,B) \ar[r]^\pi
& H^1(G,C)\ar@/_1.5pc/[l]_j\ar[r]& 0.
}
\end{aligned}
\end{equation}
\noindent 
Since $C$ acts on itself by inner automorphism, it also acts on $B$ and $A$ by the splitting $j$.
Let $\chi\in Z^1(G,C)$ be a cocycle. It can also be seen as a cocycle in $B$ via $j$. Hence, we let 
\vspace{-1em}
\begin{equation}\label{ses11}
\begin{aligned}
\xymatrix{
1\ar[r] &A' \ar[r]^{i'} &
B' \ar[r]^{\pi'}
& C'\ar@/_1.5pc/[l]_{j'}\ar[r]& 1
}
\end{aligned}
\end{equation} 
\noindent be the split short exact sequence of $G$-groups constructed by twisting \eqref{ses9} by $\chi$.
We obtain  the corollary below by Proposition \ref{twistid}.

\begin{cor}\label{twistid2}
In the diagram below, the rows are split short exact sequence of pointed sets and
the vertical arrows are bijective
with $\tau_{j(\chi)}([id_{B'}])=[j(\chi)]$, $\tau_\chi([id_{C'}])=[\chi]$, and $\tau_\chi\circ\pi'=\pi\circ\tau_{j(\chi)}$.
\vspace{-.5em}
\begin{equation}\label{ident2}
\begin{aligned}
\xymatrix{
0\ar[r] &H^1(G,A') \ar[r]^{i'} &H^1(G,B') \ar[r]^{\pi'}\ar[d]_{\tau_{j(\chi)}} & H^1(G,C')\ar@/_1.5pc/[l]_{j'}\ar[r]\ar[d]_{\tau_\chi}& 0 \\
0\ar[r] &H^1(G,A) \ar[r]^i &
H^1(G,B) \ar[r]^\pi
& H^1(G,C)\ar@/_1.5pc/[l]_j\ar[r]& 0
}
\end{aligned}
\end{equation}
\end{cor}

\paragraph{}
Let $\bG^{\sp}$ be a connected split reductive group defined over $F$.
By Proposition \ref{split}, there is a split short exact sequence of $\Gal(\overline F/F)$-groups
\vspace{-0.5em}
\begin{equation*}
\begin{aligned}
\xymatrix{
0\ar[r] &(\bG^{\sp})^{\ad}(\overline F) \ar[r]^{i} &
\Aut_{\overline F}\bG^{\sp} \ar[r]^{\hspace{.2in}\pi}
& \Theta_{\overline F}\ar@/_1.5pc/[l]_j\ar[r]& 0,
}
\end{aligned}
\end{equation*}

\noindent inducing a split short exact sequence of pointed sets
\vspace{-1em}
\begin{equation}\label{ses12}
\begin{aligned}
\xymatrix{
0\ar[r] &H^1(F,(\bG^{\sp})^{\ad}) \ar[r]^{i\hspace{.1in}} &
H^1(F,\Aut_{\overline F}\bG^{\sp}) \ar[r]^{\hspace{.2in}\pi}
& H^1(F,\Theta_{\overline F})\ar@/_1.5pc/[l]_j\ar[r]& 0.
}
\end{aligned}
\end{equation}

\noindent A reductive group $\bG/F$ is said to be quasi-split if $\bG$ has a Borel subgroup defined over $F$.
The group $\Theta_{\overline F}$ via $j$ is a group of $F$-automorphisms of $\bG^{\sp}/F$.
The image of $j$ in \eqref{ses12} can be characterized.

\begin{customthm}{I} (see e.g. \cite[Thm. 4.2]{Hu18} and its proof)\label{quasisplit} 
The set $j(H^1(F,\Theta_{\overline F}))$ in (\ref{ses12}) is equal to
the set of isomorphism classes of quasi-split $F$-forms of $\bG^{\sp}$. 
Moreover, if $\chi\in Z^1(F,\Theta_{\overline F})$, then the $\Gal(\overline F/F)$-group ${}_\chi\! \bG^{\sp}(\overline F)$
is the $\overline F$-points of a quasi-split connected reductive group $\bG'$ over $F$ corresponding to the
$\overline F$-isomorphism class $[\bG']=j([\chi])$.
\end{customthm}

\noindent Since the twisted automorphism group ${}_\chi\!\Aut_{\overline F}\bG^{\sp}$ acts on ${}_\chi\! \bG^{\sp}(\overline F)=\bG'(\overline F)$ by Theorem \ref{quasisplit}, the twisted group ${}_\chi\!\Aut_{\overline F}\bG^{\sp}$
is naturally isomorphic to $\Aut_{\overline F}\bG'$.
Denote by $\bG^{'\hspace{-0.2em}\ad}$ the adjoint quotient of $\bG'$.
By Corollary \ref{twistid2}, the following diagram has split short exact rows of pointed sets and
the vertical arrows are bijective with $\tau_{j(\chi)}([id])=[\bG']$ and  $\tau_\chi\circ\pi'=\pi\circ\tau_{j(\chi)}$.

\vspace{-1.5em}
\begin{equation}\label{ident3}
\begin{aligned}
\xymatrix{
0\ar[r] &H^1(F,\bG^{'\hspace{-0.2em}\ad}) \ar[r]^{i'\hspace{.1in}} &H^1(F,\Aut_{\overline F}\bG') \ar[r]^{\hspace{.2in}\pi'}\ar[d]_{\tau_{j(\chi)}} & H^1(F,\Theta_{\overline F}')\ar@/_1.5pc/[l]_{j'}\ar[r]\ar[d]_{\tau_\chi}& 0 \\
0\ar[r] &H^1(F,(\bG^{\sp})^{\ad}) \ar[r]^{i\hspace{.1in}} &
H^1(F,\Aut_{\overline F}\bG^{\sp}) \ar[r]^{\hspace{.2in}\pi}
& H^1(F,\Theta_{\overline F})\ar@/_1.5pc/[l]_j\ar[r]& 0
}
\end{aligned}
\end{equation}

\vspace{.1in}
\begin{remark}\label{inner}~
\begin{enumerate}[(1)]
\item The middle vertical correspondence $\tau_{j(\chi)}$ in \eqref{ident3} is the identity map if
we identify the set of isomorphism classes of $F$-forms of $\bG'$ with that of $\bG^{\sp}$ in a natural way.
\item The twisted group $\Theta'_{\overline F}$ is naturally isomorphic to $\Out_{\overline F}\bG'$
and corresponds via $j'$ to the set of isomorphism classes of quasi-split $F$-forms of $\bG'$.
\item Let $\bG_1$ and $\bG_2$ be two $F$-forms of $\bG^{\sp}$. The form $\bG_1$ is said to be 
an inner form of $\bG_2$ if $\pi([\bG_1])=\pi([\bG_2])$. By Theorem \ref{quasisplit}, any $F$-form 
$\bG_1$ is an inner form of a unique quasi-split $F$-form $\bG'$.
\end{enumerate}
\end{remark}

\paragraph{} Similarly, let $\chi\in Z^1(F,\theta_{\overline F})$ and twist 
the second row of \eqref{ses3} by $\chi$. 
Then we obtain an $F$-form $\bG'\subset\GL_{m',D'}$ of the chain $\bG^{\sp}\subset\GL_{n,F}$,
where $\bG'$ is a quasi-split $F$-form of $\bG^{\sp}$ and
$\GL_{m',D'}$ is an inner form of $\GL_{n,F}$ (for some central division algebra $D'$ over $F$).
Since $\bG'$ is quasi-split and the tautological representation is absolutely irreducible,
it follows that $\GL_{m',D'}=\GL_{n,F}$ \cite[Thm. 3.3]{Ti71} and the $F$-form is
\begin{equation}\label{chainform}
\bG'\subset\GL_{n,F}
\end{equation} 
such that the following diagram has split short exact rows of pointed sets and the vertical 
arrows are bijective with $\tau_{j(\chi)}([id])=[\bG'\subset \GL_{n,F}]$ and $\tau_\chi\circ\pi'=\pi\circ\tau_{j(\chi)}$.

\vspace{-1em}
\begin{equation}\label{ident4}
\begin{aligned}
\xymatrix{
0\ar[r] &H^1(F,\bG^{'\hspace{-0.2em}\ad}) \ar[r]^{i'\hspace{.4in}} &H^1(F,\Inn_{\overline F}(\GL_{n,F},\bG')) \ar[r]^{\hspace{.5in}\pi'}\ar[d]_{\tau_{j(\chi)}} & H^1(F,\theta_{\overline F}')\ar@/_1.5pc/[l]_{j'}\ar[r]\ar[d]_{\tau_\chi}& 0 \\
0\ar[r] &H^1(F,(\bG^{\sp})^{\ad}) \ar[r]^{i\hspace{.4in}} &
H^1(F,\Inn_{\overline F}(\GL_{n,F},\bG^{\sp})) \ar[r]^{\hspace{.5in}\pi}
& H^1(F,\theta_{\overline F})\ar@/_1.5pc/[l]_j\ar[r]& 0
}
\end{aligned}
\end{equation}

\begin{cor}\label{fiber}
The fiber $\pi^{-1}([\chi])$ in \eqref{ident3} (resp. \eqref{ident4}) can be identified with $H^1(F,\bG^{'\hspace{-0.2em}\ad})$.
\end{cor}

\subsubsection{Image of $\pi$}
Given a short exact sequence of $G$-groups with $A$ abelian:
\begin{equation}\label{ses13}
\begin{aligned}
\xymatrix{
1\ar[r] &A \ar[r]^i & B \ar[r]^\pi & C\ar[r]& 1.
}
\end{aligned}
\end{equation}

\noindent Then $C$ acts on $A$ naturally and there is the twisted group ${}_\chi\! A$ for every $\chi\in Z^1(G,C)$.
One associates to $\chi$ a cohomology class $\Delta(\chi)\in H^2(G,{}_\chi\! A)$ as follows.
Lift $\chi$ to a continuous map $g\mapsto b_g$ of $G$ into $B$ and define
\begin{equation}
a_{g,g'}=b_gg(b_{g'})b_{gg'}^{-1},
\end{equation}
which is a $2$-cocycle with values in ${}_\chi\! A$ \cite[Ch. 1.5.6]{Se97}.

\begin{customprop}{J}\cite[Ch. 1.5.6 Prop. 41]{Se97}
The cohomology class $[\chi]$ belongs to the image of $\pi: H^1(G,B)\to H^1(G,C)$ 
if and only if $\Delta(\chi)$ vanishes in $H^2(G,{}_\chi\! A)$.
\end{customprop}

Since the middle columns of \eqref{ses2} and \eqref{ses4}
 are short exact sequence of $\Gal(\overline F/F)$-groups with $\bT^{\sp}/\bC$ abelian,
we obtain the following.

\begin{cor}\label{image}
Let $\mu\in Z^1(F,\Omega_{\overline F})$ (resp. $Z^1(F,\omega_{\overline F})$). 
The cohomology class $[\mu]$ belongs to the image of 
$\mathrm{Res}_{\bT^{\sp}}$ in \eqref{ses6} (resp. \eqref{ses8})
if and only if $\Delta(\mu)$ vanishes in $H^2(F,{}_\mu\! (\bT^{\mathrm{sp}}/\bC))$.
\end{cor}

\subsection{Local-global aspects}\label{local-global}
\subsubsection{The localization map}
Let $E$ be a number field and $\mathcal{P}_E$ be the set of places of $E$.
Let $\bG$ be a linear algebraic group 
(or more generally an automorphism group of a reductive chain in $\mathsection\ref{Galcoho}$) defined over $E$ .
For any $\lambda\in \mathcal{P}_E$, denote by $E_\lambda$ the completion of $E$ 
with respect to $\lambda$ and by $i_\lambda:E\to E_\lambda$ the embedding.
Let $i_{\overline \lambda}:\overline E\to \overline E_\lambda$ be an embedding extending $i_\lambda$.
Then it induces homomorphisms $\Gal(\overline E_\lambda/E_\lambda)\to \Gal(\overline E/E)$
and $\bG(\overline E)\to \bG(\overline E_\lambda)$
for which  the $\Gal(\overline E_\lambda/E_\lambda)$-module  $\bG(\overline E_\lambda)$
and  $\Gal(\overline E/E)$-module  $\bG(\overline E)$ are compatible.
We obtain a map of cocycles ($k=0,1$ if $\bG$ non-abelian)
\begin{equation}\label{locbarv}
\mathrm{loc}_{\overline\lambda}: Z^k(E,\bG)\to Z^k(E_\lambda,\bG).
\end{equation}
The associated map of Galois cohomology 
\begin{equation}\label{locv}
\mathrm{loc}_\lambda: H^k(E,\bG)\to H^k(E_\lambda,\bG)
\end{equation}
is called the \emph{localization map at $\lambda$}. It is functorial and does not depend on $i_{\overline \lambda}$ \cite[Ch. 2.1.1]{Se97}.

\subsubsection{Some results} We would like to present some results for the map
\begin{equation}\label{sumloc}
\prod_{\lambda\in\mathcal{P}_E}\mathrm{loc}_\lambda: H^k(E,\bG)\to \prod_{\lambda\in\mathcal{P}_E} H^k(E_\lambda,\bG)
\end{equation}
when $\bG$ is connected reductive and $k=1$ and when $\bG$ is a torus and $k=2$. 
For simplicity, we use the notation and formulation of \cite{Bo98} although the results were 
obtained earlier by Harder \cite{Ha66}, Kneser \cite{Kn69}, Sansuc \cite{Sa81}, Kottwitz \cite{Ko86}.
Let $\Sh^k(E,\bG)$ be the kernel of the map \eqref{sumloc}. 
The reductive group $\bG$ is said to satisfy the \emph{Hasse principle} if
the \emph{Shafarevich-Tate group} $\Sh^1(E,\bG)$ of $\bG$ vanishes.

Denote  by $\overline{\bG}$ the group $\bG\times_E\overline E$, by $\overline{\bG}^{\der}$ the derived group of $\overline{\bG}$, 
 by $\overline{\bG}^{\sc}$ the simply-connected cover of $\overline{\bG}^{\der}$, by $\rho:\overline{\bG}^{\sc}\to \overline{\bG}$ the natural map, by $\overline{\bT}$ a maximal torus of $\overline{\bG}$,
 and by $\X_*$ the cocharacter functor for a torus.
The \emph{algebraic fundamental group} of $\overline{\bG}$ \cite[Def. 1.3]{Bo98}
is a $\Gal(\overline E/E)$-module defined as 
$$M:=\X_*(\overline{\bT})/\rho_*(\X_*(\rho^{-1}(\overline{\bT}))).$$
For each $\lambda\in\mathcal{P}_E$, one has a map \cite[5.15]{Bo98}
\begin{equation}\label{abmap}
\mu_\lambda: H^1(E_\lambda,\bG)\stackrel{\mathrm{ab}^1}{\longrightarrow} H^1_{\mathrm{ab}}(E_\lambda,\bG)
=\mathcal{T}^{-1}_\lambda(M)\stackrel{\mathrm{cor}^{-1}_\lambda}{\longrightarrow}\mathcal{T}^{-1}(M)=(M_{\Gal(\overline E/E)})_{\mathrm{tor}},
\end{equation}
where $H^1_{\mathrm{ab}}(E_\lambda,\bG)$ is the \emph{first abelian Galois cohomology group of} $\bG$ \cite[Definition 2.2]{Bo98} and
$(M_{\Gal(\overline E/E)})_{\mathrm{tor}}$ denotes the the torsion subgroup of the Galois coinvariants of $M$.
The  surjectivity of
\emph{abelianization map} $\mathrm{ab}^1$ is by \cite[Thm. 5.4]{Bo98}.
If $E_\lambda$ is non-Archimedean, then  $\mathcal{T}^{-1}_\lambda(M)=(M_{\Gal(\overline E_\lambda/E_\lambda)})_{\mathrm{tor}}$ 
\cite[Propositions 2.8 and 4.1(i)]{Bo98} and $\mathrm{cor}^{-1}_\lambda$ is the natural map \cite[4.7]{Bo98}.

\begin{customthm}{K}\label{local-globalred}\cite[Thm. 5.16]{Bo98}
When $k=1$, the map in \eqref{sumloc} factors through $\bigoplus_{\lambda\in\mathcal{P}_E} H^1(E_\lambda,\bG)$ and 
$$0\to \Sh^1(E,\bG)\to H^1(E,\bG)\to \bigoplus_\lambda H^1(E_\lambda,\bG)\stackrel{\oplus \mu_\lambda}{\longrightarrow}(M_{\Gal(\overline E/E)})_{\mathrm{tor}} $$
is exact.
\end{customthm}

As $M$ is finite for  semisimple $\bG$, we obtain the following.

\begin{prop}\label{local-globalss}
If $\bG$ is semisimple and $E_\lambda$ is non-Archimedean, then $\mu_\lambda$ in \eqref{abmap} is surjective.
\end{prop}

\noindent We have the following result for torus $\bG=\bT$ by class field theory and \cite[Lemma 5.6.2]{Bo98}.

\begin{customprop}{L}\label{local-globaltorus}
Suppose $\bT$ is a direct product of a split torus $\bT^{\sp}$ and a torus $\bT'$ such that $\bT'$ is anisotropic over $E_\lambda$
for some place $\lambda$ of $E$. Then $\Sh^2(E,\bT)=\Sh^2(E,\bT^{\sp})\oplus\Sh^2(E,\bT')=0$.
\end{customprop}

\subsection{Proofs of main theorems}\label{PfMeta}

\subsubsection{The $1$-cocycles $\mu$ and $\chi$}\label{cocycle}
According to conditions (a),(b),(c) of the main theorem(s),
we have a chain $\bT^{\sp}\subset\bG^{\sp}\subset\GL_{n,E}$,
a chain $\bT\subset\GL_{n,E}$, and an $\overline E$-isomorphism of representations
$$f_{\overline E}:(\bT^{\sp}\times_E\overline E\hookrightarrow\GL_{n,\overline E})\stackrel{\cong}{\rightarrow}
(\bT\times_E\overline E\hookrightarrow \GL_{n,\overline E}).$$
This produces a $1$-cocycle (as well as a Galois representation since $\Gal(\overline E/E)$ acts trivially on $\Aut_{\overline E}\bT^{\sp}$):
\begin{equation}\label{defmu}
\mu=(\mu_\sigma):=(f_{\overline E}^{-1}\circ 
{}^\sigma\!f_{\overline E})\in Z^1(E,\Aut_{\overline E}\bT^{\sp})=\Hom(\Gal(\overline E/E),\Aut_{\overline E}\bT^{\sp}).
\end{equation}
As $\Omega_{\overline E}$ (resp. $\omega_{\overline E}$) is a subgroup of $\Aut_{\overline E}\bT^{\sp}$ ($\mathsection\ref{Diag}$),
we first show the following.

\begin{prop}\label{Chebo}
The image of the Galois representation $\mu:\Gal(\overline E/E)\to\Aut_{\overline E}\bT^{\sp}$ is contained in $\Omega_{\overline E}$ 
(resp. $\omega_{\overline E}$ if $\bG^{\sp}$ is irreducible on $E^n$).
Thus, it defines a class $\mu\in Z^1(E,\Omega_{\overline E})$ (resp. $Z^1(E,\omega_{\overline E})$).
\end{prop}

\begin{proof}
For every $i_{\overline\lambda}:\overline E\to \overline E_\lambda$ with $\lambda\in\mathcal{P}_{E,f}$,
$$\mathrm{loc}_{\overline\lambda}(\mu)=\mathrm{Res}_{\bT^{\sp}}\circ\mathrm{loc}_{\overline\lambda}((f_{\overline E_\lambda}^{-1}\circ 
{}^\sigma\!f_{\overline E_\lambda})) \in \Hom(\Gal(\overline E_\lambda/E_\lambda),\Omega_{\overline E_\lambda})=\Hom(\Gal(\overline E_\lambda/E_\lambda),\Omega_{\overline E})$$ 
(resp. $\Hom(\Gal(\overline E_\lambda/E_\lambda),\omega_{\overline E})$) by 
\eqref{locbarv}, condition (b),
and diagram \eqref{ses2} (resp. diagrams \eqref{inclu1} and \eqref{ses4}) for $F=E_\lambda$.
Hence, all the local representations land on $\Omega_{\overline E}$ (resp. $\omega_{\overline E}$).
Since $\Aut_{\overline E}\bT^{\sp}$ is discrete, the image of $\mu$ is finite.
We are done by the Chebotarev density theorem.
\end{proof}

\noindent So it makes sense to define by  diagram \eqref{ses2} (resp. \eqref{ses4}) for $F=E$ the twisted torus 
\begin{equation}\label{twisttorus}
{}_\mu\! (\bT^{\sp}/\bC)
\end{equation}
for main theorem \ref{meta2}(d) and the $\Theta_{\overline E}$-valued (resp. $\theta_{\overline E}$-valued) $1$-cocycle
\begin{equation}\label{chi}
\chi:=\pi(\mu).
\end{equation}

\subsubsection{Proof of main theorem \ref{meta1}(i)}
By condition (b) and diagram \eqref{ses6} for $F=E_\lambda$, in $H^1(E_\lambda,\Theta_{\overline E_\lambda})$
the cohomology class $\pi([\bT_\lambda\subset\bG_\lambda])$ is equal to $\mathrm{loc}_\lambda[\chi]$.
Then by applying $\mathrm{Res}_{\bG^{\sp}}$ in diagram \eqref{ses5}, the class $\pi([\bG_\lambda])=\mathrm{loc}_\lambda[\chi]$
for all $\lambda\in\mathcal{P}_{E,f}^{(p)}$. 
By Theorem \ref{quasisplit} for $F=E$, we obtain a quasi-split connected reductive group $\bG'$ over $E$
such that $[\bG']=j[\chi]$ in \eqref{ses12}.
On the one hand, for all $\lambda\in\mathcal{P}_{E,f}^{(p)}$, 
$[\bG'\times_E E_\lambda]$ and $[\bG_\lambda]$ belong to same fiber of $\pi$ in \eqref{ses12} for $F=E_\lambda$.
On the other hand,  for almost all $\lambda\in\mathcal{P}_{E,f}^{(p)}$
\begin{equation}
[\bG'\times_E E_\lambda]= j(\mathrm{loc}_\lambda[\chi])=[\bG_\lambda]
\end{equation}
by Theorem \ref{quasisplit} for $F=E_\lambda$ and condition (d).
Hence, by Corollary \ref{fiber} for $F=E_\lambda$ for all $\lambda\in\mathcal{P}_{E,f}^{(p)}$
to identify $[\bG_\lambda]$ as an element in $H^1(E_\lambda,\bG^{'\hspace{-0.2em}\ad}\times_E E_\lambda)$,
we obtain that $[\bG_\lambda]=0$ for almost all $\lambda\in\mathcal{P}_{E,f}^{(p)}$.
Let $\lambda'$ be a place of $E$ extending $p$. Then $\lambda'\notin\mathcal{P}_{E,f}^{(p)}$.
Since $\bG^{'\hspace{-0.2em}\ad}$ is semisimple, there exists 
an element $[\bG]\in H^1(E,\bG^{'\hspace{-0.2em}\ad})$ such that $\mathrm{loc}_\lambda[\bG]=[\bG_\lambda]$
for all $\lambda\in\mathcal{P}_{E,f}^{(p)}$ by Theorem \ref{local-globalred} and Proposition \ref{local-globalss}.
Here $\bG$ is an inner form of $\bG^{'\hspace{-0.2em}\ad}$ (Remark \ref{inner}(3)).
Therefore, we conclude that $\bG\times_E E_\lambda\cong \bG_\lambda$ for all $\lambda\in\mathcal{P}_{E,f}^{(p)}$
and $\bG_\lambda$ is unramified for all but finitely many $\lambda$.\qed

\begin{remark}
Besides $\mathrm{loc}_\lambda[\bG]=[\bG_\lambda]$
for all $\lambda\in\mathcal{P}_{E,f}^{(p)}$, we can impose conditions at other places of $E$ except $\lambda'$.
For example, we can require
that $\mathrm{loc}_\lambda[\bG]=[\bG'\times_E E_\lambda]$
for all $\lambda\in\mathcal{P}_E\backslash(\mathcal{P}_{E,f}^{(p)}\cup\{\lambda'\})$.
\end{remark}

\subsubsection{Proof of main theorem \ref{meta1}(ii)}
By condition (b) and diagram \eqref{ses8} for $F=E_\lambda$, 
the cohomology class $\pi([\bT_\lambda\subset\bG_\lambda\subset\GL_{n,E_\lambda}])$ 
is equal to $\mathrm{loc}_\lambda[\chi]$ in $H^1(E,\theta_{\overline E_\lambda})$.
Then by $\mathrm{Res}_{(\GL_{n,E_\lambda},\bG^{\sp})}$ in diagram \eqref{ses7}, 
the class $\pi([\bG_\lambda\subset\GL_{n,E_\lambda}])=\mathrm{loc}_\lambda[\chi]$
for all $\lambda\in\mathcal{P}_{E,f}^{(p)}$. 
By \eqref{chainform} for $F=E$, we obtain an $E$-form $\bG'\subset\GL_{n,E}$ of $\bG^{\sp}\subset \GL_{n,E}$
where $\bG'$ is quasi-split 
such that $[\bG'\subset\GL_{n,E}]=j[\chi]$ in \eqref{ident4}.
On the one hand, for all $\lambda\in\mathcal{P}_{E,f}^{(p)}$, 
$[(\bG'\subset\GL_{n,E})\times_E E_\lambda]$ and $[\bG_\lambda\subset\GL_{n,E_\lambda}]$ 
belong to same fiber of $\pi$ in \eqref{ident4} for $F=E_\lambda$.
On the other hand,  for almost all $\lambda\in\mathcal{P}_{E,f}^{(p)}$
\begin{equation}
[(\bG'\subset\GL_{n,E})\times_E E_\lambda]= j(\mathrm{loc}_\lambda[\chi])=[\bG_\lambda\subset\GL_{n,E_\lambda}]
\end{equation}
by Theorem \ref{quasisplit} for $F=E_\lambda$, condition (d), and the proposition below.

\begin{prop}\cite[Lemma 3.2, Thm. 3.3]{Ti71}
Let $F$ be a field of characteristic zero and $D_i$ ($i=1,2$) be central simple 
algebras over $F$. Let $\bH$ be a connected reductive group over $F$ and $\rho_i:\bH\to \GL_{m_i,D_i}$ ($i=1,2$)
be two $F$-representations that are absolutely irreducible.  
If $\rho_1\times_F\overline F\cong\rho_2\times\overline F$, then $\rho_1\cong \rho_2$.
\end{prop}

Hence, by Corollary \ref{fiber} for $F=E_\lambda$ for all $\lambda\in\mathcal{P}_{E,f}^{(p)}$
to identify $[\bG_\lambda\subset\GL_{n,E_\lambda}]$ as an element in $H^1(E_\lambda,\bG^{'\hspace{-0.2em}\ad}\times_E E_\lambda)$,
we obtain that $[\bG_\lambda\subset\GL_{n,E_\lambda}]=0$ for almost all $\lambda\in\mathcal{P}_{E,f}^{(p)}$.
Let $\lambda'$ be a place of $E$ extending $p$. Then $\lambda'\notin\mathcal{P}_{E,f}^{(p)}$.
Since $\bG^{'\hspace{-0.2em}\ad}$ is semisimple, there exists 
an element $[\bG\subset\GL_{m,D}]\in H^1(E,\bG^{'\hspace{-0.2em}\ad})$ 
such that 
\begin{equation}\label{require}
\begin{split}
\mathrm{loc}_\lambda[\bG\subset\GL_{m,D}]&=[\bG_\lambda\subset\GL_{n,E_\lambda}],\hspace{.2in}\forall\lambda\in\mathcal{P}_{E,f}^{(p)}\\
\mathrm{loc}_\lambda[\bG\subset\GL_{m,D}]&=[(\bG'\subset\GL_{n,E})\times E_\lambda],\hspace{.2in}
\forall\lambda\in\mathcal{P}_E\backslash(\mathcal{P}_{E,f}^{(p)}\cup\{\lambda'\})
\end{split}
\end{equation} 
by Theorem \ref{local-globalred} and Proposition \ref{local-globalss}.
Here $\bG$ (resp. $\GL_{m,D}$) is an inner form of $\bG^{'\hspace{-0.2em}\ad}$ (resp. $\GL_{n,E}$)
and  $\GL_{m,D}=\GL_{n,E}$ by \eqref{require} and class field theory.
By Lemma \ref{equivrepn}, we conclude that $(\bG\hookrightarrow\GL_{n,E})\times_E E_\lambda\cong (\bG_\lambda\hookrightarrow\GL_{n,E_\lambda})$ as representations for all $\lambda\in\mathcal{P}_{E,f}^{(p)}$
and $\bG_\lambda$ is unramified for all but finitely many $\lambda$.\qed

\subsubsection{Proof of main theorem \ref{meta2}}
Consider the cocycle $\mu\in Z^1(E,\Omega_{\overline E})$ (resp. $Z^1(E,\omega_{\overline E})$). 
By condition (b), 
$\mathrm{loc}_\lambda[\mu]=\mathrm{Res}_{\bT^{\sp}}[\bT_\lambda\subset\bG_\lambda]$ 
(resp. $\mathrm{Res}_{\bT^{\sp}}[\bT_\lambda\subset\bG_\lambda\subset\GL_{n,E_\lambda}]$) for all $\lambda\in\mathcal{P}_{E,f}$.
It suffices to show that $[\mu]$ belongs to the image of the injection $\mathrm{Res}_{\bT^{\sp}}$ (ensuring uniqueness)
in  diagram \eqref{ses6} (resp. \eqref{ses8}) for $F=E$.
By Corollary \ref{image}, this is equivalent to $\Delta(\mu)=0$ in $H^2(E, {}_\mu\! (\bT^{\sp}/\bC))$.
By condition (d) and Proposition \ref{local-globaltorus}, 
it remains to prove that $\mathrm{loc}_\lambda(\Delta(\mu))=0$ for all places $\lambda$ of $E$.
For a finite place $\lambda$, this is true by the fact that the image of $\mathrm{Res}_{\bT^{\sp}}$ in
\eqref{ses6} (resp. \eqref{ses8}) contains
$\mathrm{loc}_\lambda[\mu]$ and Corollary \ref{image} for $F=E_\lambda$.
For a real place, this is true by (d) and $H^2(\R,\bS_\R)=0$ if $\bS_\R$ is an $\R$-anisotropic torus (see \cite[Lemma 10.4]{Ko86}).
Therefore, we obtain a common $E$-form $\bT\subset\bG$ (resp. $\bT\subset\bG\hookrightarrow\GL_{m,D}$ by Lemma \ref{equivrepn})
of the chain $\bT_\lambda\subset\bG_\lambda$ (resp. the chain representation $\bT_\lambda\subset\bG_\lambda\hookrightarrow \GL_{n,E_\lambda}$) 
for all finite places $\lambda$ of $E$.\qed

\section{Rationality of algebraic monodromy groups}
This section is devoted to the proofs of the statements in $\mathsection1.2$. 
Fix a number field $E$ and denote by $p_\lambda$ the residue characteristic of the finite place $\lambda\in\mathcal{P}_{E,f}$.

\subsection{Profinite group $\Pi$ and Frobenius elements $\Fr$.} \label{twocases}
Consider two cases.

\subsubsection{(Characteristic zero).} 
 In this case, $\Pi$ denotes the absolute Galois group $\Gal(\overline K/K)$ of a number field $K$ and
$\mathcal{P}:=\mathcal{P}_{E,f}$. 
Equip $\Pi$ with a subset $\Fr\subset\Pi$ of \emph{Frobenius elements} as follows.

For all $v\in \mathcal{P}_{K,f}$, let $q_v$ be the size of the residue field $\F_{q_v}$ of $K_v$ and consider the natural surjection
$$\pi_v:\Gal(\overline K_v/K_v)\to \Gal(\overline \F_{q_v}/\F_{q_v}).$$
For each $v$, fix a lift $\phi_v\in \pi_v^{-1}(\Fr_{q_{v}}^{-1})$, where 
$\Fr_{q_{v}}^{-1}\in \Gal(\overline \F_{q_v}/\F_{q_v})$ is the geometric Frobenius.
Each $\overline v\in \mathcal{P}_{\overline K,f}$ determines an embedding 
$\iota_{\overline v}:\Gal(\overline K_v/K_v)\to \Gal(\overline K/K)$.
For $\overline v\in \mathcal{P}_{\overline K,f}$, define $\Fr_{\overline v}$ to be 
$\iota_{\overline v}(\phi_v)$ where $v$ is the restriction of $\overline v$ to $K$. Define
$$\Fr_v:=\bigcup_{\overline v|v} [\Fr_{\overline v}]\hspace{.2in}\mathrm{and} \hspace{.2in}
\Fr:=\bigcup_{v\in\mathcal{P}_{K,f}} \Fr_{v}.$$
For any Galois extension $L/K$ that is unramified except finitely many $v\in\mathcal{P}_{K,f}$ and any finite subset $S\subset \mathcal{P}_{K,f}$, 
the image of $\bigcup_{v\in\mathcal{P}_{K,f}\backslash S} \Fr_{v}$ in $\Gal(L/K)$ is dense \cite[Chap. I, $\mathsection2.2$ Cor. 2]{Se98}. 
Assign the number $q_v$ to the elements in $\Fr_{v}$.

\subsubsection{(Characteristic $p$).} 
 In this case, $\Pi$ denotes the \'etale fundamental group $\pi^{\et}_1(X,\bar x)$ (with some base point $\bar x$) 
of a smooth geometrically connected variety $X/\F_q$ in characteristic $p$ and
$\mathcal{P}:=\mathcal{P}_{E,f}^{(p)}$. 
Equip $\Pi$ with a subset $\Fr\subset\Pi$ of \emph{Frobenius elements} as follows.

Let $X^{cl}$ be the set of closed points of $X$.
For any geometric point $\overline x'$ over $x'\in X^{cl}$, let $\Fr_{\overline x'}$ be the image of the geometric Frobenius 
$\Fr_{q_{x'}}^{-1}\in\Gal(\overline \F_{q_{x'}}/\F_{q_{x'}})=\pi_1(x',\overline x')$ under the natural map
$$\pi_1(x',\overline x')\to \pi_1(X,\overline x')\stackrel{\sigma_{xx'}}{\longrightarrow}\pi_1(X,\overline x),$$
where $q_{x'}$ is the size of the residue field of $x'$.
Note that the change of base point isomorphism $\sigma_{xx'}$ is unique up to an inner automorphism of $\pi_1(X,\overline x)$.
Since the conjugacy class $[\Fr_{\overline x'}]$ depends only on $x'$, write $\Fr_{x'}:=[\Fr_{\overline x'}]$ and define
$$\Fr:=\bigcup_{x'\in X^{cl}} \Fr_{x'}.$$
The subset $\Fr$ is dense in $\Pi$ \cite{Se65}. Assign the number $q_{x'}$ to the elements in $\Fr_{x'}$.

\subsection{$E$-compatible systems} 
Let $(\Pi,\Fr,\mathcal P)$ be one of the two cases in $\mathsection\ref{twocases}$. 
In the characteristic zero case, denote by $S$ a finite subset of $\mathcal{P}_{K,f}$.
Otherwise, $S$ is the empty set.

\subsubsection{$\GL_n$-valued compatible systems}
A system of $n$-dimensional $\lambda$-adic (continuous) representations 
$$\rho_\bullet:=\{\rho_\lambda:\Pi\to \GL_n(E_\lambda)\}_{\lambda\in \mathcal{P}}$$ 
of $\Pi$ is said to be \emph{semisimple} (resp. \emph{irreducible, absolutely irreducible}) 
if for all $\lambda\in\mathcal{P}$, $\rho_\lambda$ is semisimple (resp. irreducible, absolutely irreducible).
The system $\rho_\bullet$ is said to be $E$-\emph{compatible} (with exceptional set $S$) 
if 
\begin{itemize}
\item in the characteristic zero case, $\rho_\lambda$ is unramified outside $S\cup\{t\in\mathcal{P}_{K,f}: p_\lambda|q_t\}$ for each $\lambda\in\mathcal{P}$;
\item for each Frobenius element $\Fr_{\overline t}\in\Fr$ satisfying $t\notin S$
and for each $\lambda$ satisfying $p_\lambda\nmid q_t$,
the characteristic polynomial 
\begin{equation}\label{poly}
P_t(T):=\det(\rho_\lambda(\Fr_{\overline t})-T\cdot I_n)\in E_\lambda[T]
\end{equation} 
has coefficients in $E$ and depends only on $t$ (independent of $\lambda\in\mathcal{P}$).
\end{itemize}
The compatible system $\rho_\bullet$ is said to be \emph{pure of weight} $w\in\R$ (resp. \emph{mixed of weights})
if for each $\Fr_{\overline t}\in\Fr$ with $t\notin S$ and each root $\alpha\in\overline E$ of $P_t(T)$,
the absolute value $|i(\alpha)|$ is equal to $q_t^{w/2}$ for all
complex embedding $i:\overline E\to\C$ (resp. is independent of the complex embedding $i:\overline E\to\C$).

\subsubsection{Coefficient extension and the Weil restriction} Let $\rho_\bullet$ be an $n$-dimensional (semisimple) $E$-compatible system of $\Pi$ that is pure of weight $w$ (resp. mixed of weights).
For a number field $E'$, denote by $\mathcal{P}'=\mathcal{P}_{E',f}$ in characteristic zero case
and by $\mathcal{P}'=\mathcal{P}_{E',f}^{(p)}$ in characteristic $p$ case.

If $E'$ is an extension of $E$, then we obtain by \emph{coefficient extension} 
a (semisimple) system $\rho_\bullet\otimes_E E'$ of $n$-dimensional $\lambda'$-adic representations:
\begin{equation}
(\rho_\bullet\otimes_E E')_{\lambda'}:=(\Pi\stackrel{\rho_\lambda}{\rightarrow}\GL_n(E_\lambda)\subset\GL_{n}(E'_{\lambda'})),
\end{equation}
where $\lambda$ is the restriction of $\lambda'$ to $E$. The system is $E'$-compatible (with exceptional set $S$), 
pure of weight $w$ (resp. mixed of weights),
and called the \emph{coefficient extension of $\rho_\bullet$ to $E'$} (see \cite[Definition 3.2]{BGP19}).

If $E'$ is a subfield of $E$,
then we obtain by the \emph{Weil restriction of scalars} a (semisimple) system $\mathrm{Res}_{E/E'}\rho_\bullet$ 
of $n[E:E']$-dimensional $\lambda'$-adic representations:
\begin{equation}
(\mathrm{Res}_{E/E'}\rho_\bullet)_{\lambda'}:=\bigoplus_{\lambda|\lambda'}\rho_\lambda:\Pi\to 
\prod_{\lambda|\lambda'}\GL_n(E_\lambda)=(\mathrm{Res}_{E/E'}\GL_{n,E})(E'_{\lambda'})\subset\GL_{n[E:E']}(E'_{\lambda'}).
\end{equation}
\noindent 
The system is $E'$-compatible (with exceptional set $S$), pure of weight $w$ (resp. mixed of weights),
and called the \emph{Weil restriction of $\rho_\bullet$} (see \cite[Definition 3.4]{BGP19}).

\subsubsection{$\bG$-valued compatible systems} Let $\bG$ be a linear algebraic group defined over $E$ with affine coordinate ring $R$.
Since $\bG$ acts on itself by conjugation, $\bG$ acts on $R$. The subring of invariant functions is denoted by $R^\bG$. 
For all $g\in \bG$, let $g_s$ be the semisimple part of $g$. If $g$ is defined over a field extension $F/E$, 
then $g_s$ is also defined over $F$.
A system of  $\lambda$-adic $\bG$-representations 
$\{\rho_\lambda:\Pi\to \bG(E_\lambda)\}_{\lambda\in \mathcal{P}}$ of $\Pi$
is said to be \emph{$E$-compatible} (with exceptional set $S$) 
if 
\begin{itemize}
\item in the characteristic zero case, $\rho_\lambda$ is unramified outside $S\cup\{t\in\mathcal{P}_{K,f}: p_\lambda|q_t\}$ for each $\lambda\in\mathcal{P}$;
\item for each Frobenius element $\Fr_{\overline t}\in\Fr$ satisfying $t\notin S$, each $\lambda$ satisfying $p_\lambda\nmid q_t$,
and each $f\in R^\bG$
the number  
\begin{equation}
f(\rho_\lambda(\Fr_{\overline t})_s)\in E_\lambda
\end{equation} 
belongs to $E$ and depends only on $t$ and $f$ \cite[Chap. I, $\mathsection2.4$]{Se98} 
(independent of $\lambda\in\mathcal P$)\footnote{This is equivalent to the conjugacy class of $\rho_\lambda(\Fr_{\overline t})_s$ in $\bG$ being defined over $E$ and depends only on $t\notin S$ (independent of $\lambda$).}.
\end{itemize}
It follows that an $n$-dimensional $E$-compatible system is the same as an $E$-compatible system of $\GL_{n,E}$-representations.

\subsubsection{Algebraic monodromy groups and connectedness}
For all $\lambda\in\mathcal{P}$, the \emph{algebraic monodromy group} of $\rho_\lambda$,
i.e., the Zariski closure of the image of $\rho_\lambda$ in $\GL_{n,E_\lambda}$, is denoted by $\bG_\lambda$.
It is an closed subgroup of $\GL_{n,E_\lambda}$.
The image $\rho_\lambda(\Pi)$ is a compact subgroup of the $\lambda$-adic Lie group $\bG_\lambda(E_\lambda)$.
The following result is well-known by using the compatibility condition, see \cite[Prop. 6.14]{LP92}.

\begin{customprop}{M}\label{connect}
The component groups $\bG_\lambda/\bG_\lambda^\circ$ are isomorphic for all $\lambda\in\mathcal{P}$. 
In particular, the connectedness of $\bG_\lambda$ is independent of $\lambda$.
\end{customprop}

\subsubsection{Group schemes}
Suppose the algebraic monodromy group $\bG_\lambda$ is connected reductive for all $\lambda$. 
Let $\mathcal{O}_\lambda$ be the ring of 
integers of $E_\lambda$ with residue field $k_\lambda$ of characteristic $p_\lambda$.
Let $\Lambda_\lambda$ be an $\mathcal{O}_\lambda$-lattice of $E_\lambda^n$ that is invariant under 
the image $\rho_\lambda(\Pi)$. Let $\mathcal{G}_\lambda$ be the Zariski closure 
of $\rho_\lambda(\Pi)$ in $\GL_{\Lambda_\lambda}\cong\GL_{n,\mathcal{O}_\lambda}$, endowed with 
the unique structure of reduce closed subscheme. The generic fiber of $\mathcal{G}_\lambda$
is $\bG_\lambda$. The special fiber, denoted by $\mathcal{G}_{k_\lambda}$, is identified as a 
subgroup of $\GL_{n,k_\lambda}$.
When $p_\lambda\geq n$, the subgroup $\mathcal{G}_{k_\lambda}\subset\GL_{n,k_\lambda}$ is said to be \emph{saturated}
if for any unipotent element $u\in \mathcal{G}_{k_\lambda}(\overline k_\lambda)$,
the one parameter subgroup $\{u^a: a\in\overline k_\lambda\}\subset\GL_n(\overline k_\lambda)$
belongs to  $ \mathcal{G}_{k_\lambda}(\overline k_\lambda)$ \cite[$\mathsection4.2$]{Se94}.

\begin{customprop}{N}\label{gpscheme}\cite[Prop. 1.3]{LP95},\cite[Prop. 5.51, Thm. 5.52]{BGP19}
For all but finitely many $\lambda\in\mathcal{P}$, the following assertions hold.
\begin{enumerate}[(i)]
\item The group scheme $\mathcal{G}_\lambda$ is smooth with constant absolute rank over $\mathcal{O}_\lambda$.
\item The identity component of the special fiber $\mathcal{G}_{k_\lambda}\subset\GL_{n,k_\lambda}$ is saturated.
\end{enumerate}
\end{customprop} 

\subsection{Frobenius torus}

\subsubsection{Frobenius torus and maximal torus}
For all $\lambda\in\mathcal{P}$, let $\bG_\lambda$ be the algebraic monodromy group of $\rho_\lambda$.
The identity component of $\bG_\lambda$ is reductive since $\rho_\lambda$ is semisimple.
Let $\rho_\lambda$ be a member of the system and $\Fr_{\overline t}\in\Fr$ be a 
Frobenius element with $t\notin S$.
If $p_\lambda\nmid q_t$, then 
the \emph{Frobenius torus} $\bT_{\overline t,\lambda}$ of  $\Fr_{\overline t}$ 
is defined to be the identity component of the smallest (diagonalizable) 
algebraic subgroup $\bS_{\overline t,\lambda}$ in $\GL_{n,E_\lambda}$ 
containing the semisimple part of $\rho_\lambda(\Fr_{\overline t})$.
It follows that $\bT_{\overline t,\lambda}\subset \bS_{\overline t,\lambda}\subset \bG_\lambda$.
The following theorem is due to Serre.

\begin{customthm}{O}\label{FMT}(see \cite[Thm. 1.2 and its proof]{LP97}, \cite[Thm. 5.7]{Ch04}, \cite[Thm. 2.6]{Hu18})
Suppose the algebraic monodromy group $\bG_{\lambda'}$ is connected for some $\lambda'\in\mathcal{P}$. 
Suppose there exists a finite subset $Q\subset\Q$ such that for all $\Fr_{\overline t}\in \Fr$ with $t\notin S$,
the following conditions are satisfied for every root $\alpha$ of the characteristic polynomial $P_t(T)$ in \eqref{poly}:
\begin{enumerate}[(a)]
\item the absolute values of $\alpha$ in all complex embeddings are equal;
\item $\alpha$ is a unit at any finite place not extending $p_t$;
\item for any finite place $w$ of $\overline \Q$ such that $w(p_t)>0$, the ratio $w(\alpha)/w(q_t)$
belongs to $Q$.
\end{enumerate}
Then there exists a proper closed subvariety $\bY_{\lambda'}$ of $\bG_{\lambda'}$ such that $\bT_{\overline t,\lambda'}$
is a maximal torus of $\bG_{\lambda'}$ whenever $\rho_{\lambda'}(\Fr_{\overline t})\in\bG_{\lambda'}\backslash \bY_{\lambda'}$.
\end{customthm}

\begin{remark}\label{Dirichlet}
~
\begin{enumerate}[(1)]
\item If $\bG_\lambda$ is connected and the Frobenius torus $\bT_{\overline t,\lambda}$ is maximal, 
then $\bT_{\overline t,\lambda}=\bS_{\overline t,\lambda}$.
\item Conditions \ref{FMT}(a),(b) hold for our mixed compatible system $\rho_\bullet$.
\item Condition \ref{FMT}(c) holds in the characteristic $p$ case by replacing $X$ 
with a non-empty open subset $U$ \cite[Thm. 1.3.3(i), Remark 1.3.5]{DK17}.
\item Condition \ref{FMT}(c) holds in the characteristic zero case if we assume the system is 
$\{H^w(Y_{\overline K},\Q_\ell)\}_{\ell\in\mathcal{P}_{\Q,f}}$ for some smooth projective variety $Y/K$ \cite[Thm. 1.1]{LP97}.
\item In the characteristic $p$ case, the subset of elements $\Fr_{\overline t}$ 
of $\Fr$ whose Frobenius tori $\bT_{\overline t,\lambda'}$ are maximal in $\bG_{\lambda'}$
is dense in $\Pi$. 
\item In the characteristic zero case, the subset of places $v\in\mathcal{P}_{K,f}$ such that $\bT_{\overline v,\lambda'}$
is a maximal torus of $\bG_{\lambda'}$ is of Dirichlet density one (see \cite[Cor. 2.7]{Hu18}).
\end{enumerate}
\end{remark}

Let $\Fr_{\overline t}$ be a Frobenius element. 
There is a semisimple matrix $M_t$ of $\GL_n(E)$ with $P_t(T)$ \eqref{poly} as characteristic polynomial.
For all $\lambda\in\mathcal{P}$ with $p_\lambda\nmid q_t$, $M_t$ is conjugate to the semisimple part $\rho_\lambda(\Fr_{\overline t})_s$
in $\GL_n(E_\lambda)$ by $E$-compatibility. Hence, if we let $\bS_t$ be the smallest algebraic subgroup of $\GL_{n,E}$ containing $M_t$ 
and $\bT_t$ be the identity component of $\bS_t$, then the chain representations 
$(\bT_t\subset\bS_t\hookrightarrow\GL_{n,E})\times_E E_\lambda$
and $\bT_{\overline t,\lambda}\subset\bS_{\overline t,\lambda}\hookrightarrow\GL_{n,E_\lambda}$
are isomorphic for all $\lambda\in\mathcal{P}$ with $p_\lambda\nmid q_t$.

\begin{cor}\label{FMTcor}
Assume the conditions of Theorem \ref{FMT}. Then the following assertions hold.
\begin{enumerate}[(i)]
\item (Common $E$-form of formal characters) If the Frobenius torus $\bT_{\overline t,\lambda'}$ is a maximal torus of $\bG_{\lambda'}$,
then the Frobenius torus $\bT_{\overline t,\lambda}$ is also a maximal torus of $\bG_{\lambda}$
for all $\lambda\in\mathcal{P}$ with $p_{\lambda}\nmid q_t$.
Moreover, the representation 
$(\bT_{t}\hookrightarrow \GL_{n,E})\times_E E_{\lambda}$ is isomorphic to 
$\bT_{\overline t,\lambda}\hookrightarrow\GL_{n,E_{\lambda}}$ for all $\lambda\in\mathcal{P}$ with $p_{\lambda}\nmid q_t$.
\item (Absolute rank) The absolute rank of $\bG_\lambda$ is independent of $\lambda$.
\end{enumerate}
\end{cor}

\begin{proof}
Assertion (i) is straight forward by Theorem \ref{FMT} and the above construction of $\bT_t$.
Assertion (ii) is obvious by (i) in the characteristic $p$ case and follows from (i) and Remark \ref{Dirichlet}(6) in the characteristic zero case.
\end{proof}

\subsubsection{Anisotropic subtorus}
In this subsection, $\bG_\lambda$ is connected for all $\lambda\in\mathcal{P}$.
The subtorus $\bT_t\subset\GL_{n,E}$ in Corollary \ref{FMTcor}(i) is studied
under the following hypothesis. 
Let $k$ be the order of $\bS_t/\bT_t$. Then the Zariski closure of $M_t^{k\Z}$ in $\GL_{n,E}$ is $\bT_t$.\\

\noindent\textbf{Hypothesis R}: Assume for each real embedding $E\to \R$, the set of powers $\det(M_t)^\Z\subset\R$
contains some non-zero integral power of the absolute value $|i(\alpha)|$ for every root $\alpha$ of $P_t(T)$
and every complex embedding $i:\overline E\to \C$ extending $E\to \R$.

\begin{prop}\label{anisoreal}
If Hypothesis R holds, then the subtorus $(\bT_t\cap \SL_{n,E})^\circ$ of $\bT_t$ is anisotropic at all real places of $E$.
\end{prop}

\begin{proof}
Embed $E$ into $\R$ and let $\bT_{t,\R}\subset\bS_{t,\R}\subset\GL_{n,\R}$ be the base change to $\R$. 
If $\chi:\bT_{t,\R}\to\mathbb{G}_{m,\R}$ is a $\R$-character, then $\chi(M_t^k)\in \mathbb{G}_{m}(\R)=\R^*$
Let $i:\overline E\to \C$ be an embedding extending $E\to \R$. Then $\chi(M_t^k)$
is the product of some integral powers of the roots $i(\alpha)$ of the polynomial $i(P_t(T))\in\R[T]$.
Hence, there exist integers $h\neq 0$ and $m$ such that
$$\chi^{2h}(M_t^k)= \det(M_t)^{2m}\in \R^*_{>0}$$
 by Hypothesis R. This implies $\chi^{2hk}=\det^{2m}$ on $\bT_{t,\R}$ since $(M_t^k)^\Z$ is Zariski dense in $\bT_{t,\R}$.
Hence, $\chi^{2hk}$ is trivial on the subtorus $(\bT_{t,\R}\cap\SL_{n,\R})^\circ$ for some $2hk\neq 0$.
We conclude that the torus $(\bT_{t,\R}\cap\SL_{n,\R})^\circ$ is anisotropic.
\end{proof}

\begin{cor}\label{anisopadic}
If Hypothesis R holds and $E$ has a real place, then the subtorus $(\bT_t\cap \SL_{n,E})^\circ$ of $\bT_t$ 
is anisotropic at a positive Dirichlet density subset $\mathcal{P}'$ of $\mathcal{P}_{E,f}$.
\end{cor}

\begin{proof}
Let $r$ be the absolute rank of the $E$-torus $(\bT_{t}\cap\SL_{n,E})^\circ$.
Then it is an $E$-form of the split torus $\mathbb{G}_{m,E}^r$ with automorphism group $\GL_r(\Z)$. 
The isomorphism class of $(\bT_{t}\cap\SL_{n,E})^\circ$ is represented by an element of $H^1(E, \GL_r(\Z))$,
which is a continuous group homomorphism $\phi:\Gal(\overline E/E)\to \GL_r(\Z)$ up to conjugation.
Let $c\in\Gal(\overline E/E)$ be a complex conjugation corresponding to a real place of $E$. 
Since $(\bT_{t}\cap\SL_{n,E})^\circ$ is anisotropic over $\R$ by Proposition \ref{anisoreal} and $c$ is of order two, 
it follows that $\phi(c)=-I_r$.
Since the image of $\phi$ is finite, there is a positive Dirichlet density set $\mathcal{P}'$ of 
finite places $\lambda$ of $E$ such that $\phi(\Fr_\lambda)=-I_r$ 
by the Chebotarev density theorem.
Therefore, $(\bT_{t}\cap\SL_{n,E})^\circ$ is anisotropic over $E_\lambda$ for all $\lambda\in \mathcal{P}'$.
\end{proof}

\begin{remark}\label{pureR}
~
\begin{enumerate}[(1)]
\item Hypothesis R holds for every $P_t(T)$ if the $E$-compatible system is pure.
\item If $\lambda\in\mathcal{P}'$ in Corollary \ref{anisopadic}, then the $E_\lambda$-subtorus 
$(\bT_{\overline t,\lambda}\cap\bG_\lambda^{\der})^\circ$ of 
$(\bT_{\overline t,\lambda}\cap\SL_{n,E_\lambda})^\circ\cong(\bT_{t}\cap\SL_{n,E})^\circ\times_E E_\lambda$ 
is also anisotropic. If $\bT_{\overline t,\lambda}\subset \bG_\lambda$ is a maximal torus, then 
$\bT_{\overline t,\lambda}\cap\bG_\lambda^{\der}\subset \bG_\lambda^{\der}$ is also a maximal torus.
\item Corollary \ref{anisopadic} is not true for general $E$ since $(\bT_t\cap \SL_{n,E})^\circ$
can be a non-trivial split torus over $E$. This is done by taking a finite extension $E'/E$ such that 
$P_t[T]$ splits and replacing the $E$-compatible system $\rho_\bullet$ 
with its coefficient extension $\rho_\bullet\otimes_E E'$ ($\mathsection3.2.2$). 
\end{enumerate}
\end{remark}

Let $\bG$ be a connected reductive group defined over a field $F$. A torus
$\bT\subset\bG$ is said to be \emph{fundamental} if it is a maximal torus with minimal $F$-rank.  
In the characteristic zero case, let $\mathcal{S}$ be the subset of elements $v\in\mathcal{P}_{K,f}$ 
such that for some $\lambda\in\mathcal{P}_{E,f}$, the Frobenius torus $\bT_{\overline v,\lambda}\cong\bT_v\times_E E_\lambda$
is a fundamental torus of $\bG_\lambda$. 
A Frobenius torus $\bT_{\overline v,\lambda}\subset\bG_\lambda$ 
being fundamental is equivalent to $\bT_{\overline v,\lambda}$ is a maximal torus and 
$\bT_{\overline v,\lambda}\cap \bG_\lambda^{\der}$ is anisotropic \cite[Prop. 5.3.2]{Bo98}.
When Hypothesis R holds and $E$ has a real place, Remark \ref{Dirichlet}(6), Corollary \ref{anisopadic}, 
and Remark \ref{pureR}(2) imply that $\mathcal{S}$ is of Dirichlet density one. \\

\noindent \textbf{Question Q}: Suppose Hypothesis R holds, what is the Dirichlet density of $\mathcal{S}$ in $\mathcal{P}_{K,f}$
when $E$ is totally imaginary?\\

\noindent We do not know the answer; we even do not know if $\mathcal{S}$ is non-empty.
If we want to apply main theorem \ref{meta2} 
to the algebraic monodromy representations $\{\bG_\lambda\hookrightarrow\GL_{n,E_\lambda}\}_{\lambda\in\mathcal{P}_{E,f}}$ 
when $E$ is totally imaginary, then a positive Dirichlet density of $\mathcal{S}$ is necessary.

\subsection{Proofs of characteristic $p$ results}
Let $\mathcal{P}$ be $\mathcal{P}_{E,f}^{(p)}$. 

\subsubsection{Proof of Theorem \ref{thm1}}
By Proposition \ref{connect} and taking a finite Galois covering of $X$, 
we assume that $\bG_\lambda$ is connected for all $\lambda\in\mathcal{P}$.
It suffices to check conditions (a),(b),(c),(d) of main theorem \ref{meta1}
for the system of algebraic monodromy representations
$$\{\bG_\lambda\hookrightarrow\GL_{n,E_\lambda}\}_{\lambda\in\mathcal{P}}.$$
Conditions (a),(b),(c) follow directly from assertions (i),(ii),(iii) of Theorem \ref{Chin}.
Condition (d) holds by \cite[Cor. 7.9]{BGP19}, 
or by Proposition \ref{hyper} below, for almost all $\lambda$, the existence of 
a hyperspecial maximal compact subgroup of $\bG_\lambda(E_\lambda)$ implies that $\bG_\lambda$ is unramified \cite[$\mathsection1$]{Mi92}.
We are done by main theorem \ref{meta1}.\qed

\begin{prop}\label{hyper}
If $\bG_\lambda$ is connected for all $\lambda\in\mathcal{P}$,
 then the image of $\rho_\lambda$ is contained in a hyperspecial maximal compact subgroup $H_\lambda$
of $\bG_\lambda(E_\lambda)$ for almost all $\lambda$.
\end{prop}

\begin{proof}
Since $\pi_1(X)$ is compact, we may assume $\rho_\lambda(\pi_1(X))\subset\GL_n(\mathcal O_\lambda)$ 
after some change of coordinates $V_\lambda\cong E_\lambda^n$ for all $\lambda$.
The geometric \'etale fundamental group $\pi_1^{\geo}(X)$ of $X$ satisfies the short exact sequence
\begin{equation*}\label{fundgp}
1\to \pi_1^{\geo}(X)\to \pi_1(X)\to \Gal(\overline\F_q/\F_q)\to 1.
\end{equation*}
Denote the Zariski closure of $\rho_\lambda(\pi_1^{\geo}(X))$ in $\GL_{n,E_\lambda}$ by $\bG_\lambda^{\geo}$.
Let $\mathcal{G}_\lambda$ (resp. $\mathcal{G}_\lambda^{\geo}$) be the Zariski closure
of $\bG_\lambda$ (resp. the identity component of $\bG_\lambda^{\geo}$) in $\GL_{n,\mathcal{O}_\lambda}$
with special fiber $\mathcal{G}_{k_\lambda}$ (resp. $\mathcal{G}_{k_\lambda}^{\geo}$).
It suffices to prove that for almost all $\lambda$, 
$H_\lambda:=\mathcal{G}_\lambda(\mathcal{O}_\lambda)$ 
is a hyperspecial maximal compact subgroup
of $\bG_\lambda(E_\lambda)$. By Bruhat-Tits theory, this condition follows if 
we show that the $\mathcal{O}_\lambda$-group scheme
$\mathcal{G}_\lambda$ is reductive \cite[$\mathsection3.8.1$]{Ti79}.

By \cite[Thm. 7.3]{BGP19}, 
the $\mathcal{O}_\lambda$-group scheme $\mathcal{G}_\lambda^{\geo}$
is semisimple for almost all $\lambda$. Let $k_\lambda$ be the residue field of $E_\lambda$.
Since the $\mathcal{O}_\lambda$-group scheme $\mathcal{G}_\lambda$
is smooth with constant absolute rank for almost all $\lambda$ by Proposition \ref{gpscheme}(i)
and contains $\mathcal{G}_\lambda^{\geo}$ as a closed normal subgroup scheme,
the inequalities
\begin{align*}
\begin{split}
\dim \bG_\lambda&=\dim(\mathcal{G}_{k_\lambda})=\dim (\mathcal{G}_{k_\lambda}^{\geo}) 
+ \dim (\mathcal{G}_{k_\lambda}/\mathcal{G}_{k_\lambda}^{\geo})\\
&\geq \dim (\mathcal{G}_{k_\lambda}^{\geo}) 
+ \rk (\mathcal{G}_{k_\lambda}/\mathcal{G}_{k_\lambda}^{\geo})
=\dim \bG_\lambda^{\geo}
+ \rk (\bG_\lambda/\bG_\lambda^{\geo})=\dim \bG_\lambda.
\end{split}
\end{align*}
implies that the special fiber $\mathcal{G}_{k_\lambda}$
has trivial unipotent radical for almost all $\lambda$.
Therefore, the smooth group scheme $\mathcal{G}_\lambda$ is reductive over $\mathcal{O}_\lambda$ for almost all $\lambda$.
\end{proof}

\subsubsection{Proof of Corollary \ref{cor1}}
By Theorem \ref{thm1}(i), there is a connected reductive group $\bG$ defined over $E$ 
and an isomorphism $\phi_\lambda:\bG\times_E E_\lambda\to \bG_\lambda$ for each $\lambda\in\mathcal{P}$.
For almost all $\lambda$, the $\mathcal{O}_\lambda$-points $\bG(\mathcal{O}_\lambda)$ is 
well-defined (by finding some integral model $\mathcal{G}$ of $\bG$) 
and is a hyperspecial maximal compact subgroup of the $E_\lambda$-points $\bG(E_\lambda)$ \cite[$\mathsection 3.8.1$]{Ti79}.
Let $\bG_\lambda^{\ad}$ be the adjoint group of $\bG_\lambda$.
The subgroup $\bG_\lambda^{\ad}(E_\lambda)$ of $\Aut_{\overline E_\lambda}\bG_\lambda(E_\lambda)$ is transitive on the 
set of hyperspecial maximal compact subgroups of $\bG_\lambda(E_\lambda)$ \cite[$\mathsection 2.5$]{Ti79}.
Hence, by Proposition \ref{hyper} and adjusting $\phi_\lambda$ for almost all $\lambda$,
we assume $\phi_\lambda(\bG(\mathcal{O}_\lambda))=H_\lambda\subset\bG_\lambda(E_\lambda)$ for almost all $\lambda$.
Then the image of the map $\prod_{\lambda\in\mathcal{P}}\phi_\lambda^{-1}\circ\rho_\lambda$ is contained in the 
adelic points $\bG(\A_E^{(p)})$, which defines the desired $\bG$-valued adelic representation $\rho_\A^{\bG}$. This proves  assertion (i).

The proof of (ii) is exactly the same except we want to adjust the isomorphism of representation
$$\phi_\lambda: (\bG\hookrightarrow\GL_{n,E})\times_E E_\lambda \to (\bG_\lambda\hookrightarrow \GL_{n,E_\lambda})$$
in order to have
\begin{equation}\label{help}
\phi_\lambda^{-1}(H_\lambda)=\bG(\mathcal{O}_\lambda)\subset\GL_{n,E}(\mathcal{O}_\lambda)
\end{equation}
for almost $\lambda$ (the inclusion is defined by finding some integral 
model $\mathcal{G}\subset\GL_{n,\mathcal{O}_{E,S}}$ of $\bG\subset \GL_{n,E}$).
This can be achieved since $\bG_\lambda^{\ad}(E_\lambda)$ is a subgroup of 
$\Inn_{\overline E_\lambda}(\GL_{n,E_\lambda},\bG_\lambda)(E_\lambda)$ (see $\mathsection2.3.1$) as the representation $\rho_\lambda$ is
absolutely irreducible. \qed

\subsubsection{Proof of Corollary \ref{cor1.5}}
Find a smooth $\mathcal O_{E,S}$-model $\mathcal{G}\subset\GL_{n,\mathcal{O}_{E,S}}$ 
of $\bG\subset \GL_{n,E}$ for some finite $S\subset\mathcal{P}_{E,f}$.
Then by enlarging $S$ we obtain that 
the group scheme $\GL_{n,\mathcal{O}_{E,S}}\times\mathcal{O}_\lambda$ (resp. $\mathcal{G}\times\mathcal{O}_\lambda$)
is the group scheme associated to the hyperspecial maximal compact subgroup 
$\GL_{n,\mathcal{O}_{E,S}}(\mathcal{O}_\lambda)$ of $\GL_{n,\mathcal{O}_{E,S}}(E_\lambda)=\GL_n(E_\lambda)$
(resp. $\mathcal{G}(\mathcal{O}_\lambda)$ of $\bG(E_\lambda)$)
for all $\lambda\in\mathcal{P}\backslash S$ \cite[$\mathsection3.9.1$]{Ti79}.
We may assume that for all $\lambda\in\mathcal{P}\backslash S$, the inclusion
$$\mathcal{G}(\mathcal{O}_\lambda)\subset\GL_{n,\mathcal{O}_{E,S}}(\mathcal{O}_\lambda)=\GL_n(\mathcal{O}_\lambda)$$
gives the construction $\bG(\mathcal{O}_\lambda)\subset\GL_{n,E}(\mathcal{O}_\lambda)$ in 
\eqref{help}. Since the $\lambda$-component 
$$(\rho_\A^{\bG})_\lambda:\pi_1(X)\to \mathcal{G}(\mathcal{O}_\lambda)
\subset\GL_{n,\mathcal{O}_{E,S}}(\mathcal{O}_\lambda)=\GL_n(\mathcal{O}_\lambda)\subset\GL_n(E_\lambda)$$
of the adelic representation $\rho_\A^{\bG}$ is isomorphic to $\rho_\lambda$ by Corollary \ref{cor1}(ii), 
the representation $(\mathcal{G}\hookrightarrow \GL_{n,\mathcal{O}_{E,S}})\times\mathcal{O}_\lambda$
is isomorphic to $\mathcal{G}_\lambda\hookrightarrow \GL_{n,\mathcal{O}_\lambda}$, where $\mathcal G_\lambda$
is the Zariski closure of $\rho_\lambda(\pi_1(X))$ in $\GL_{n,\mathcal O_\lambda}$
after some choice of $\mathcal{O}_\lambda$-lattice in $V_\lambda$.  \qed

\begin{remark}\label{hyperremark}
The proofs of Corollaries \ref{cor1} and \ref{cor1.5} are standard in the sense that 
they only require the common $E$-forms $\bG$ and $\bG\subset\GL_{n,E}$ in Theorem \ref{thm1},
Proposition \ref{hyper}, and Bruhat-Tits theory \cite{Ti79}.
\end{remark}

\subsubsection{Proof of Corollary \ref{cor2}}
By Corollary \ref{cor1}(ii), there is a common $E$-form $\iota:\bG\hookrightarrow\GL_{n,E}$.
For each $\lambda\in\mathcal{P}$, choose an embedding $\overline E\to \overline E_\lambda$. 
We claim that the conjugacy class of the semisimple part $\rho_\lambda^{\bG}(\Fr_{\overline t})_s\in\bG(E_\lambda)$
is defined over $\overline E$ for all Frobenius element $\Fr_{\overline t}$ and all $\lambda\in\mathcal{P}$.
Indeed, by field extension, we obtain 
$$\rho_\lambda^{\bG}(\Fr_{\overline t})_s\in(\bG\times\overline E)(\overline E_\lambda)\stackrel{\iota_{\overline E}}{\hookrightarrow}\GL_{n,\overline E}(\overline E_\lambda).$$
It suffices to show that for any irreducible representation $\psi$ of $\bG\times\overline E$,
the trace of $\psi(\rho_\lambda^{\bG}(\Fr_{\overline t})_s)\in\overline E$. 
This is true because the roots $\alpha$ of the characteristic polynomial $P_t(T)$
of $\rho_\lambda^{\bG}(\Fr_{\overline t})\in \GL_{n,\overline E}(\overline E_\lambda)$ 
belong to $\overline E$ by $E$-compatibility and 
$\psi$ is a subrepresentation of $\otimes^r \iota_{\overline E} \otimes^s \iota_{\overline E}^*$ for some $r,s\in\Z_{\geq0}$.

The next step is to show that for a fixed Frobenius element $\Fr_{\overline t}$, 
the conjugacy class of $\rho_\lambda^{\bG}(\Fr_{\overline t})_s$ in $\bG$ is independent of $\lambda$.
By \cite[Thm. 4.3.2]{D'Ad20} (\cite[Thm. 6.8, Cor. 6.9]{Ch04} when $X$ is a curve), 
there is a finite extension $F$ of $E$ and a connected reductive subgroup 
$\bG^{\sp}\subset\GL_{n,F}$ such that for all $\lambda\in\mathcal{P}$,
if $F_\lambda$ is a completion of $F$ extending $\lambda$ on $E$, then there exists an 
isomorphism of representations:
\begin{equation}\label{Chin1}
f_{F_\lambda}: (\bG^{\sp}\hookrightarrow\GL_{n,F})\times_F F_\lambda\stackrel{\cong}{\rightarrow}
(\bG_\lambda\hookrightarrow\GL_{n,E_\lambda})\times_{E_\lambda}F_\lambda.
\end{equation}
Moreover by \cite[Proof of Thm. 4.3.2]{D'Ad20} (\cite[Thm. 6.12]{Ch04} when $X$ is a curve), the representations 
\begin{equation}\label{Chin2}
\rho_\lambda^{\bG^{\sp}}:\pi_1(X)\stackrel{\rho_\lambda}{\rightarrow} (\bG_\lambda\times F_\lambda)(F_\lambda)\stackrel{f_{F_\lambda}^{-1}}{\rightarrow}\bG^{\sp}(F_\lambda)
\end{equation}
for all $\lambda$ form an $F$-compatible system of $\bG^{\sp}$-representations.
Hence, the conjugacy class $[\rho_\lambda^{\bG^{\sp}}(\Fr_{\overline t})_s]$ in $\bG^{\sp}$ is independent of $\lambda$.
If $\beta_\lambda\in N_{\GL_{n,E}}(\bG)(E_\lambda)$, we obtain the isomorphisms below
\begin{align}\label{Chin3}
\begin{split}
&(\bG\times_E E_\lambda\hookrightarrow \GL_{n,E_\lambda})\times_{E_\lambda} F_\lambda \stackrel{\beta_\lambda^{-1}\times F_\lambda}{\longrightarrow} (\bG\times_E E_\lambda\hookrightarrow \GL_{n,E_\lambda})\times_{E_\lambda} F_\lambda \\
=&(\bG\hookrightarrow \GL_{n,E})\times_E F_\lambda \stackrel{\phi_\lambda\times F_\lambda}{\longrightarrow}
(\bG_\lambda\hookrightarrow\GL_{n,E_\lambda})\times_{E_\lambda} F_\lambda  \stackrel{f_{F_\lambda}^{-1}}{\longrightarrow}
(\bG^{\sp}\hookrightarrow\GL_{n,F})\times_F F_\lambda
\end{split}
\end{align}
by Corollary \ref{cor1}(ii) and \eqref{Chin1}.
Fix $\lambda'\in\mathcal{P}$, define $\beta_{\lambda'}=id$, and embed $F_\lambda$ into $\C$ for all $\lambda\in\mathcal{P}$.
 It suffices to find $\beta_\lambda$ for all $\lambda\in\mathcal{P}\backslash\{\lambda'\}$ such that
\begin{equation}\label{innerhope}
\Phi_\lambda:=[(\beta_\lambda\times F_\lambda)\circ(\phi_\lambda\times F_\lambda)^{-1}\circ f_{F_\lambda} \circ f_{F_{\lambda'}}^{-1}\circ(\phi_{\lambda'}\times F_{\lambda'})]\times\C\in \Inn_{\C}(\bG\times\C).
\end{equation}
Then $\Phi_\lambda([\rho_{\lambda'}^{\bG}(\Fr_{\overline t})_s])=[\rho_\lambda^{\bG}(\Fr_{\overline t})_s]$ 
is an equality of conjugacy class in $\bG$ for all $\Fr_{\overline t}\in\Fr$.

For (i), since $\bG_\lambda$ is split and is irreducible on the ambient space, 
$N_{\GL_{n,E}}(\bG)(E_\lambda)$ surjects onto $\theta_{\overline E_\lambda}$ in \eqref{ses3}.
Thus, there is 
$\beta_\lambda\in N_{\GL_{n,E}}(\bG)(E_\lambda)$
such that $\Phi_\lambda \in \Inn_\C(\bG\times\C)=\bG^{\ad}(\C)$.
For (ii), take $\beta_\lambda=id$ for all $\lambda$. 
Since the outer automorphism group $\Out(\bG^{\der}\times\C)$ is trivial and $\bG\times\C\hookrightarrow\GL_{n,\C}$ is irreducible,
the image of $\Phi_\lambda$ in $\Out(\bG\times\C)$ is also trivial.
Hence, we conclude that in both cases (i) and (ii), $\Phi_\lambda$ is inner and 
$[\rho_{\lambda'}(\Fr_{\overline t})_s]=[\rho_\lambda(\Fr_{\overline t})_s]$ for
all $\Fr_{\overline t}$.
For $\Fr_{\overline t}\in\Fr$, it follows that the $\overline E$-conjugacy class $[\rho_\lambda(\Fr_{\overline t})_s]$ is independent of $\lambda\in\mathcal{P}$. 

Let $R$ be the affine coordinate ring of $\bG$.
For any $f\in R^{\bG}$, $f_t:=f([\rho_\lambda^{\bG}(\Fr_{\overline t})_s])\in \overline E$ is independent of $\lambda$
and also belongs to $E_\lambda$ for all $\lambda\in\mathcal{P}$. Therefore, $f_t\in E$ and 
we conclude that $\{\rho_\lambda^{\bG}\}_{\lambda\in\mathcal{P}}$ is an $E$-compatible system of $\bG$-representations.
The last claim of the Corollary is immediate.\qed

\begin{remark}\label{outerweak}
In general, if we can find for each $\lambda$ an element 
$\beta_\lambda\in \Inn_{\overline E_\lambda}(\GL_{n,E_\lambda},\bG_\lambda)(E_\lambda)$ 
such that $\Phi_\lambda$ (defined in \eqref{innerhope}) belongs to $\Inn_\C(\bG\times\C)$,
then the conclusion of the corollary also follows.\end{remark}

\subsubsection{Proof of Theorem \ref{thmcrys}}
For each $\lambda\in\mathcal{P}_{E,f}^{(p)}$, we have the chain $\bT_{\bar t,\lambda}\subset\bG_\lambda^\circ\subset\GL_{V_\lambda}\cong\GL_{n,E_\lambda}$. For each $v\in\mathcal{P}_{E,p}$, we have the chain $\bT_{t,v}\subset\bG_{t,v}^\circ\subset\GL_{V_{t,v}}\cong\GL_{n,E_v}$
by condition (b) of Theorem \ref{thmcrys}.
By identifying the algebraic monodromy groups as subgroups of $\GL_n$,
we obtain a chain $\bT_\lambda\subset\bG_\lambda^\circ\subset\GL_{n,E_\lambda}$ for each finite place $\lambda$ of $E$.
Here we simplify our notation by representing places of $E$ extending $p$ also as $\lambda$.
To prove the theorem, it suffices to find a torus $\bT\subset\GL_{n,E}$ and a chain $\bT^{\sp}\subset\bG^{\sp}\subset\GL_{n,E}$
such that conditions (a),(b),(c),(d) of main theorem \ref{meta2} for the system
$$\{\bT_\lambda\subset\bG_\lambda^\circ\subset\GL_{n,E_\lambda}\}_{\lambda\in\mathcal{P}_{E,f}}$$
hold. Note that the last sentence of Theorem \ref{thmcrys}(ii) follows from Remark \ref{simdiff}(6).
The verifications rely on the following result of D'Addezio (enhancing Theorem \ref{Chin})
and the fact that $\bT_\lambda$ is a maximal torus of $\bG_\lambda$ for all $\lambda\in\mathcal{P}_{E,f}$ 
by condition (a) of Theorem \ref{thmcrys}.

\begin{customthm}{B'}\label{DAd}\cite[Construction 4.2.1 (Frobenius tori), Theorem 4.3.2 and its proof]{D'Ad20}
Let $\rho_\bullet$ be the $E$-compatible system in Theorem \ref{thmcrys} and 
$\bT_\lambda\subset\bG_\lambda^\circ\subset\GL_{n,E_\lambda}$ be the chain defined above for each $\lambda\in\mathcal{P}_{E,f}$.
Then the following assertions hold.
\begin{enumerate}[(i)]
\item (Common $E$-form of formal characters): There exists a subtorus $\bT:=\bT_t$ of $\GL_{n,E}$ such that for all 
$\lambda\in\mathcal{P}_{E,f}$,
 $\bT_\lambda:=\bT\times_E E_\lambda$ is a maximal torus  of $\bG_\lambda^\circ$.
\item ($\lambda$-independence over an extension): 
There exist a finite extension $F$ of $E$ and a chain of subgroups $\bT^{\sp}\subset\bG^{\sp}\subset\GL_{n,F}$
such that $\bG^{\sp}$ is connected split reductive, $\bT^{\sp}$ is a split maximal torus of $\bG^{\sp}$, and for all 
$\lambda\in\mathcal{P}_{E,f}$, 
if $F_\lambda$ is a completion of $F$ extending $\lambda$ on $E$, then there exists an isomorphism of chain representations:
\begin{equation*}
f_{F_\lambda} :(\bT^{\sp}\subset\bG^{\sp}\hookrightarrow\GL_{n,F})\times_F F_\lambda\stackrel{\cong}{\rightarrow} 
(\bT_{\lambda}\subset\bG_\lambda^\circ\hookrightarrow\GL_{n, E_\lambda})\times_{E_\lambda} F_\lambda.
\end{equation*}
\item (Rigidity) The isomorphisms $f_{F_\lambda}$ in (ii) can be chosen such that the restriction isomorphisms $f_{F_\lambda}:\bT^{\sp}\times_F F_\lambda\to\bT_{\lambda}\times_{E_\lambda} F_\lambda$ admit a common $F$-form $f_F:\bT^{\sp}\to \bT\times_E F$ 
for all $\lambda\in\mathcal{P}_{E,f}$ and $F_\lambda$.
\end{enumerate}
\end{customthm}

Then conditions \ref{meta2}(a),(b),(c) are just Theorem \ref{DAd}(i),(ii),(iii). 
For condition \ref{meta2}(d), let $\bT_t\subset\GL_{n,E}$ be the $E$-form in Theorem \ref{DAd}(i). 
By $\mathsection2.6.1$ and conditions \ref{meta2}(a),(b),(c), there exists an isomorphism of representations
$$f_{\overline E}:(\bT^{\sp}\hookrightarrow\GL_{n,E})\times_E\overline E 
\stackrel{\cong}{\rightarrow} (\bT_t\hookrightarrow\GL_{n,E})\times_E\overline E $$
which produces the cocycle $\mu$ as in \eqref{defmu}.
Consider the short exact sequence of $E$-groups
\begin{equation}\label{muses1}
1\to \bC\to \bT^{\sp}\to \bT^{\sp}/\bC\to 1.
\end{equation}
By Proposition \ref{Chebo}, $\mu$ as Galois representation
acts on $\bC$ and hence \eqref{muses1} in an equivariant way, inducing the short exact sequence of $E$-groups 
by twisting ($\mathsection2.4.1$)
\begin{equation}\label{muses2}
\xymatrix{
&&\bT_t\ar[d]^= &&\\
1\ar[r] &{}_\mu\!\bC\ar[r]&{}_\mu\!\bT^{\sp}\ar[r]&{}_\mu\!(\bT^{\sp}/\bC)\ar[r]& 1
}
\end{equation} 
Since $\mu$ is constructed from $f_{\overline E}$, it has values in $\Inn_{\overline E}(\GL_{n,E},\bT^{\sp})$.
It follows that $\mu$ as Galois representations acts on the surjection of $E$-groups 
\begin{equation}\label{muses3}
(\bT^{\sp}\cap\SL_{n,E})^\circ \twoheadrightarrow \bT^{\sp}/\bC
\end{equation}
in an equivariant way. Hence, we obtain the surjection of $E$-groups
\begin{equation}\label{muses3}
(\bT_t\cap\SL_{n,E})^\circ={}_\mu\!(\bT^{\sp}\cap\SL_{n,E})^\circ \twoheadrightarrow {}_\mu\!(\bT^{\sp}/\bC).
\end{equation}
By condition (c) of Theorem \ref{thmcrys}, $E$ has a real place.
Since $(\bT_t\cap\SL_{n,E})^\circ$ is anisotropic over each real place of $E$ 
by Proposition \ref{anisoreal}, Remark \ref{pureR}(1), and
the fact that $\rho_\bullet$ is pure of weight $w$, it follows that the twisted torus ${}_\mu\!(\bT^{\sp}/\bC)$
is also anisotropic over each real place of $E$ by the surjection \eqref{muses3}.\qed

\subsection{Proofs of characteristic zero results}
Let $\mathcal{P}$ be $\mathcal{P}_{\Q,f}$. 

\subsubsection{Proof of Theorem \ref{thm2}}
It suffices to check conditions (a),(b),(c),(d) of main theorem \ref{meta2}
for the system of algebraic monodromy representations
$$\{\bG_\ell\hookrightarrow\GL_{n,\Q_\ell}\}_{\ell\in\mathcal{P}}$$
and note Remark \ref{simdiff}(6).
Since the conditions remain the same after taking any finite extension $F$ of $K$, it is free to do so.\\

\textit{Condition \ref{meta2}(a)}: By condition \ref{thm2}(a) and Remark \ref{Dirichlet}(6), 
there is a place $v\in\mathcal{P}_{K,f}\backslash S$ such that the Frobenius torus $\bT_{\overline v,\ell}$ 
is a maximal torus of $\bG_\ell$ for all $\ell$ not equal to $p:=p_v$ and the local representation
$V_p$ of $\Gal(\overline K_v/K_v)$ is ordinary. It remains to check the condition for the places over $p$.

Let $Y_v$ be the special fiber of a smooth model of $Y$ over $\mathcal{O}_v$ and
$M_v:=H^w(Y_v/\mathcal{O}_v)\otimes_{\mathcal{O}_v}K_v$ be the crystalline cohomology group,
which belongs to the category $\textbf{MF}_{K_v}^f$ of \emph{weakly admissible filtered modules} over $K_v$.
There are algebraic subgroups $(\bH_{V_p}\subset\GL_{V_p})\times_{\Q_p} K_v$ and $\bH_{M_v}\subset\GL_{M_v}$
such that their tautological representations (via the \emph{mysterious functor} of Fontaine) 
are inner forms of each other, in particular isomorphic over $\overline \Q_p$,
$$(\bH_{V_p}\hookrightarrow\GL_{V_p})\times_{\Q_p} \overline\Q_p \stackrel{\iota \cong}{\longrightarrow}(\bH_{M_v}\hookrightarrow\GL_{M_v})\times_{K_v}\overline\Q_p,$$
where $\bH_{V_p}$ is the algebraic monodromy group of the local crystalline 
representation $\rho_p:\Gal(\overline K_v/K_v)\to \GL(V_p)$
and $\bH_{M_v}$ is the automorphism group of the fiber functor 
on the full Tannakian subcategory of $\textbf{MF}_{K_v}^f$ generated by $M_v$ 
that assigns a filtered $K$-module the underlying $K$-vector space (see \cite[$\mathsection2$]{Pi98}).
Let $m_v$ be the degree $[K_v:\Q_p]$ and $f_{M_v}$ the \emph{crystalline Frobenius}. 
By Katz-Messing \cite{KM74} 
(see \cite[Thm. 3.10]{Pi98}), $f_{M_v}^{m_v}$ is an element of $H_{M_v}(K_v)\subset\GL(M_v)$ with characteristic polynomial 
equal to $P_v(T)$, the characteristic polynomial of $\rho_\ell(\Fr_v)$ ($\ell\neq p$).

Let $\Phi_{V_p}$ be the element in $\bH_{V_p}(\overline \Q_p)$ corresponding to $f_{M_v}^{m_v}\in H_{M_v}(\overline\Q_p)$
via $\iota$. The group $\bH_{V_p}$ is generated by cocharacters (connected) and 
the smallest algebraic subgroup containing $\Phi_{V_p}$ \cite[Prop. 2.6]{Pi98}.
It is connected because the characteristic polynomial of $\Phi_{V_p}\in \GL_{V_p}$ is $P_v(T)$
and the (maximal) Frobenius torus $\bT_{\overline v,\ell}$ is equal to $\bS_{\overline v,\ell}$ by Remark \ref{Dirichlet}(1).
Let $V_p^{\ss}$ be the semisimplification of the representation $\bH_{V_p}\hookrightarrow\GL_{V_p}$. 
Since the local representation $V_p$ is ordinary,  
$\bH_{V_p}$ is solvable \cite[Prop. 2.9]{Pi98} and its image $\bH_{V_p}^{\red}$ in $\GL_{V_p^{\ss}}$ is a torus.
Since the conjugacy class of $\Phi_{V_p}$ in $\bH_{V_p}$ is defined over $\Q_p$ \cite[Prop. 2.2]{Pi98} and $\bH_{V_p}^{\red}$ is abelian,
the image of $\Phi_{V_p}$ in $\bH_{V_p}^{\red}$, denoted by $\Phi_{V_p}^{\red}$, 
belongs to $\bH_{V_p}^{\red}(\Q_p)$. 
By the splitting of the surjection $\bH_{V_p}\twoheadrightarrow\bH_{V_p}^{\red}$,
there is a semisimple element $\Phi_{\overline v}\in\bH_{V_p}(\Q_p)\subset\GL(V_p)$
with characteristic polynomial $P_v(T)$.
The smallest algebraic subgroup of $\bH_{V_p}\subset\bG_p$ containing 
$\Phi_{\overline v}$ is a $\Q_p$-maximal torus $\bT_{\overline v,p}$ of $\bG_p$ because
the absolute rank of $\bG_\ell$ is independent of $\ell$ by Corollary \ref{FMTcor}(ii) and
$\bT_{\overline v,\ell}$ is a maximal Frobenius torus.
Since the characteristic polynomials of 
$\Phi_{\overline v}$ and $\rho_\ell(\Fr_{\overline v})_s$ for all $\ell\neq p$ are equal to $P_v(T)$, the tori representations
$\bT_{\overline v,\ell}\hookrightarrow\GL_{V_\ell}$ for all $\ell$ admit a common $\Q$-form
$\bT_v\hookrightarrow\GL_{n,\Q}$.\qed \\

\textit{Condition \ref{meta2}(b)}: This is just condition \ref{thm2}(b).\qed \\

\textit{Condition \ref{meta2}(c)}: By Proposition \ref{c2c1} and  condition \ref{thm2}(c), it 
suffices to check condition (c'-bi) in $\mathsection2.2$. 
Identify $\GL_{V_\ell}$ as $\GL_{n,\Q_\ell}$ for all $\ell$.
We employ the technique in \cite[Prop. 3.18, Thm. 3.19]{Hu13}. 
Let $\{\psi_\ell\}_{\ell\in\mathcal{P}}$ be an $r$-dimensional
semisimple $\Q$-compatible system of abelian $\ell$-adic representations
of $\Gal(\overline K/K)$. Let $\bS_\ell\subset\GL_{r,\Q_\ell}$ be the algebraic monodromy group
of $\psi_\ell$ and assume $\bS_\ell$ is torus and with the largest possible dimension 
$d_K$ (\cite[Thm. 3.8]{Hu13}) for all $\ell$. Consider the semisimple $\Q$-compatible system
$\{\rho_\ell\oplus\psi_\ell\}_{\ell\in\mathcal{P}}$ 
and let $\bG_\ell'\subset \GL_{n,\Q_\ell}\times\GL_{r,\Q_\ell}$ be the algebraic monodromy group at $\ell$. Let 
\begin{equation}\label{proj}
p_{i,\ell}':\bG_\ell'\to \GL_{n,\Q_\ell}\times\GL_{r,\Q_\ell}
\end{equation}
be the projection to the $i$th factor, $i=1,2$.
By considering $p_{1,\ell}'$, there is a diagonalizable subgroup $\bD_\ell$ of $\bS_\ell$
with a short exact sequence
\begin{equation}\label{sesgg}
1\to \bD_\ell\to \bG_\ell'\to\bG_\ell\to 1.
\end{equation}

\noindent Let $k$ be the number of components of $\bD_{\ell'}$ for some prime $\ell'$. 
By replacing $\{\rho_\ell\oplus\psi_\ell\}_{\ell\in\mathcal{P}}$ with $\{\rho_\ell\oplus\psi_\ell^k\}_{\ell\in\mathcal{P}}$,
we assume that $\bD_{\ell'}$ is connected.
Since $\bG_{\ell'}$ is connected, $\bG_{\ell'}'$ is connected by \eqref{sesgg}.
Hence, $\bG_{\ell}'$ is connected for all $\ell$ by Proposition \ref{connect}. 
Since the dimension of the center of $\bG_\ell'$
is $d_K=\dim\bS_\ell$ for all $\ell$ \cite[Prop. 3.8, Thm. 3.19]{Hu13}, it follows that for all $\ell$
\begin{equation}\label{key}
\ker(p_{2,\ell}')^\circ=(\bG_\ell')^{\der}=\bG_\ell^{\der}.
\end{equation}

\begin{customprop}{P}\label{CM}\cite{Fa83},\cite[Chap. II]{Se98},\cite[Chap. 1, Thm. 4.1]{Sc88}
Fix a prime $\ell''$, there exist a finite extension $F$ of $K$ and an abelian variety $A$ over $F$ that is
a direct product of CM abelian varieties with the following properties. Let
$$\{\epsilon_{\ell}:\Gal(\overline F/F)\to \GL(W_{\ell})\}_{\ell\in\mathcal{P}}$$ 
be the semisimple compatible system of Galois representations with $W_\ell:=H^1(A_{\overline F},\Q_{\ell})$. 
Let $\bM_\ell$ and $\bG_\ell''$ be respectively the algebraic monodromy group of 
the Galois representation $\epsilon_\ell$ and $\rho_\ell\oplus \epsilon_\ell$ of $\Gal(\overline F/F)$.
Then the following assertions hold. 
\begin{enumerate}[(i)]
\item For all $\ell$, $\bG_\ell''$ is connected and $\bM_\ell$ is a torus with dimension independent of $\ell$.
\item The restriction map $\psi_{\ell''}:\Gal(\overline F/F)\to \GL_r(\overline\Q_{\ell''})$ factors through 
a morphism $\bM_{\ell''}\times\overline\Q_{\ell''}\to \GL_{r,\overline\Q_{\ell''}}$.
\end{enumerate}
\end{customprop}

Since $\bG_\ell'$ is connected for all $\ell$, it is again the algebraic monodromy group
of the restriction of $\rho_\ell\oplus \psi_\ell$ to $\Gal(\overline F/F)$.
Again, let
$p_{i,\ell}'':\bG_\ell''\to \GL_{n,\Q_\ell}\times\bM_\ell$
be the projection to the $i$th factor, $i=1,2$. Since there exists 
a surjective map $\bG_{\ell''}''\to \bG_{\ell''}'$ by Proposition \ref{CM}(ii), it follows from \eqref{key}
and the connectedness of $\bG_{\ell''}''$ (Proposition \ref{CM}(i)) that 
\begin{equation}\label{key2}
\ker(p_{2,\ell''}'')^\circ=\bG_{\ell''}^{\der}=(\bG_{\ell''}'')^{\der}
\end{equation}
is the semisimple part of $\bG_{\ell''}''$. Since $\{\rho_\ell\oplus\epsilon_\ell\}_{\ell\in\mathcal{P}}$
is a compatible system of representations of $\Gal(\overline F/F)$,
 the semisimple rank and the dimension of the center of $\bG_\ell''$ is independent of $\ell$ \cite[Thm. 3.19]{Hu13}.
This, together with \eqref{key2} and the $\ell$-independence of $\dim \bM_\ell$ (Proposition \ref{CM}(i)), imply that  
\begin{equation}\label{key3}
\ker(p_{2,\ell}'')^\circ=(\bG_{\ell}'')^{\der}
\end{equation}
for all $\ell$. Hence, if $\bT_{\ell}''$ is a maximal torus of $\bG_\ell''$, then
$$\ker (p_{2,\ell}'':\bT_\ell''\to \bM_\ell)^\circ\subset p_{1,\ell}''(\bT_\ell'')\hookrightarrow\GL_{n,\Q_\ell}$$
is a formal bi-character of $\bG_\ell$.

Finally, we follow the strategy in condition \ref{meta2}(a).
If $v$ is a finite place of $F$ such that $Y\times_K F$ and $A$ have good reduction,
then write $p:=p_v$ and the Frobenius element $\Fr_{\overline v}$ have
characteristic polynomials $P_v(T)\in\Q[T]$ on $V_\ell$ and $Q_v(T)\in\Q[T]$ on $W_\ell$ for all $\ell\neq p$. 
By condition \ref{thm2}(a), there exists $v\in\mathcal{P}_{F,f}$  such that
the Frobenius torus $\bT_{\overline v,\ell}''\subset\bG_\ell''$  is  maximal for all $\ell\neq p$ 
and the local representation $V_p$ of $\Gal(\overline F_v/F_v)$ is ordinary.
Then we let $\bH_{V_p\oplus W_p}\subset\GL_{V_p}\times\GL_{W_p}$ be 
the algebraic monodromy group of the local crystalline representation 
$$\rho_p\oplus\epsilon_p:\Gal(\overline F_v/F_v)\to \GL(V_p)\times\GL(W_p)$$ and 
$\bH_{V_p\oplus W_p}^{\red}$ its image (semisimplification)
in the (abelian) diagonalizable subgroup $\bH_{V_p}^{\red}\times \bM_p\subset \GL_{V_p}\times\GL_{W_p}$,
where $\bH_{V_p}^{\red}$ is defined in condition \ref{meta2}(a).
Since the local representation $V_p\oplus W_p$ is crystalline, we conclude by 
repeating the arguments in the second and third paragraphs of condition \ref{meta2}(a)
that there exists an element in $\bH_{V_p\oplus W_p}^{\red,\circ}(\Q_p)$ lifting to 
a semisimple element $\Phi_{\overline v}''\in \bH_{V_p\oplus W_p}^\circ(\Q_p)\in\bG_p''(\Q_p)$ with 
characteristic polynomials $P_v(T)$ on $V_p$ and $Q_v(T)$ on $W_p$.
The smallest algebraic subgroup $\bT_{\overline v,p}''$ of $\bG_p''$ containing $\Phi_{\overline v}''$
is also a maximal torus as the absolute rank of $\bG_\ell''$ is independent of $\ell$.
By using the polynomials $P_v(T),Q_v(T)\in\Q[T]$, we construct a common $\Q$-form 
$\bT_v''\hookrightarrow \GL_{n,\Q}\times\GL_{2\dim A,\Q}$ of 
the formal characters $\bT_{\overline v,\ell}''\hookrightarrow \GL_{n,\Q_\ell}\times\GL_{W_\ell}$ 
of $\bG_\ell''\subset \GL_{n,\Q_\ell}\times\GL_{W_\ell}$ for all $\ell$
such that  
\begin{equation}\label{formbi}
\ker (p_{2}:\bT_v''\to \GL_{2\dim A,\Q})^\circ\subset p_{1}(\bT_v'')\hookrightarrow\GL_{n,\Q}
\end{equation}
is a common $\Q$-form of formal bi-characters of $\bG_\ell\subset\GL_{n,\Q_\ell}$ for all $\ell$,
where $p_1,p_2$ are the obvious projections.
We may replace $\bT_v\hookrightarrow\GL_{n,\Q}$ constructed in 
condition \ref{meta2}(a) with $p_{1}(\bT_v'')\hookrightarrow\GL_{n,\Q}$ in \eqref{formbi}.\qed \\

\textit{Condition \ref{meta2}(d)}: Let $\bT_v\subset\GL_{n,\Q}$ be the $\Q$-form we found in condition \ref{meta2}(a). 
This part is exactly the same as the verification of condition \ref{meta2}(d) for Theorem \ref{thmcrys}
once we replace the field $E$ by $\Q$ and the $E$-torus $\bT_t$ by the $\Q$-torus $\bT_v$. \qed

\subsubsection{Proofs of Corollaries \ref{cor2.1} and \ref{cor2.2}}
Since Corollaries \ref{cor2.1} and \ref{cor2.2} (of Theorem \ref{thm2}) assume Hypothesis H, 
they follow along the same lines in 
the proofs of Corollaries \ref{cor1} and \ref{cor1.5} by Remark \ref{hyperremark}.\qed

\subsubsection{Galois maximality and Hypothesis H}
Let $K$ be a number field and $\{\rho_\ell:\Gal(\overline K/K)\to \GL_n(\Q_\ell)\}_{\ell\in\mathcal{P}}$
be a $\Q$-compatible system of $\ell$-adic representations.
Let $\Gamma_\ell$ be the image of $\rho_\ell$ and $\bG_\ell$ be the algebraic monodromy group of $\rho_\ell$.
Then $\Gamma_\ell$ is a compact subgroup of $\bG_\ell(\Q_\ell)$.
Suppose for simplicity that $\bG_\ell$ is connected for all $\ell$.
Denote by $\bG_\ell^{\ss}$ be the quotient of $\bG_\ell$ by its radical and 
by $\bG_\ell^{\sc}$ the simply-connected covering of $\bG_\ell^{\ss}$.
Denote by $\Gamma_\ell^{\ss}$ the image of $\Gamma_\ell$ in $\bG_\ell^{\ss}(\Q_\ell)$
and by $\Gamma_\ell^{\sc}$ the inverse image of $\Gamma_\ell^{\ss}$ 
in $\bG_\ell^{\sc}(\Q_\ell)$. When $\ell\gg 0$ compared to the absolute 
rank of $\bG_\ell^{\sc}$, a compact subgroup $H_\ell$ of $\bG_\ell^{\sc}(\Q_\ell)$
is hyperspecial maximal compact if the ``mod $\ell$ reduction'' 
of $H_\ell$ is ``of the same Lie type'' as the semisimple group $\bG_\ell^{\sc}$ (see \cite{HL16}).
In \cite{Lar95}, Larsen proved that 
the set of primes $\ell$ for which  $\Gamma_\ell^{\sc}\subset\bG_\ell^{\sc}(\Q_\ell)$ 
is hyperspecial maximal compact is of Dirichlet density one and 
conjectured the following.

\begin{customconj}{S}\label{Larsenconj}
For all $\ell\gg0$, $\Gamma_\ell^{\sc}$ is a hyperspecial maximal compact subgroup of $\bG_\ell^{\sc}(\Q_\ell)$.
\end{customconj}

This conjecture is also related to the conjectures of Serre on maximal motives \cite[11.4, 11.8]{Se94}.
Suppose the $\ell$-adic compatible system is 
$\{H^w(Y_{\overline K},\Q_\ell)\}_{\ell\in\mathcal{P}}$, 
where $Y$ is a smooth projective variety defined over a number field $K$.
When $Y$ is an elliptic curve without complex multiplication and $w=1$, a well-known 
theorem of Serre states that for $\ell\gg0$, $\Gamma_\ell\cong\GL_2(\Z_\ell)$ 
is maximal compact in $\GL(V_\ell)$ \cite{Se72}.
In general, by studying the \emph{mod $\ell$ compatible system} $\{H^w(Y_{\overline K},\F_\ell)\}_{\ell\gg 0}$,
we proved that $\Gamma_\ell\subset\bG_\ell(\Q_\ell)$ is large in the sense that 
its mod $\ell$ reduction has ``the same semisimple rank'' as 
the algebraic group $\bG_\ell$ for $\ell\gg0$ \cite[Thm. A]{Hu15}.
This result is crucial to the following.

\begin{customthm}{T}\label{hyperex}\cite{HL16,HL20}
Let $\rho_\bullet$ be the $\Q$-compatible system \eqref{cs} arising from a smooth projective variety $Y$ defined over $K$.
Conjecture \ref{Larsenconj} holds in the following cases.
\begin{enumerate}[(i)]
\item For $\ell\gg0$, $\bG_\ell^{\sc}$ is of type A, i.e., isomorphic to $\prod_i \SL_{n_i}$ over $\overline\Q_\ell$.
\item $Y$ is an abelian variety. 
\item $Y$ is a hyper-K\"ahler variety and degree $w=2$.
\end{enumerate}
\end{customthm}

Let $\Lambda_\ell$ be a $\Z_\ell$-lattice of $\Q_\ell^n$ that is invariant under $\Gamma_\ell$
and $\mathcal{G}_\ell$ (resp. $\mathcal{G}_\ell^{\der}$) the Zariski closure of $\Gamma_\ell$ 
(resp. the derived group $[\Gamma_\ell,\Gamma_\ell]$) in $\GL_{n,\Z_\ell}$, endowed 
with the unique reduced closed subscheme structure. Write 
$\mathcal{G}_{\F_\ell}$ (resp. $\mathcal{G}_{\F_\ell}^{\der}$)
as the special fiber.

\begin{thm}\label{QH}
Suppose $\bG_\ell$ is connected reductive for all $\ell$.
Then Conjecture \ref{Larsenconj} implies that $\mathcal{G}_\ell$ is a reductive group scheme over $\Z_\ell$ for $\ell\gg0$ and Hypothesis H.
\end{thm}

\begin{proof}
Let $\pi_\ell:\bG_\ell^{\sc}(\Q_\ell)\to\bG_\ell^{\der}(\Q_\ell)\to\bG_\ell(\Q_\ell)$
be the natural morphism. Consider the natural commutative diagram where each vertical map is the commutator map
\begin{equation*}
\xymatrix{
\bG_\ell^{\sc}\times\bG_\ell^{\sc} \ar[d] \ar[r] & \bG_\ell\times\bG_\ell \ar[d] \ar[r] & \bG_\ell^{\ss}\times\bG_\ell^{\ss} \ar[d]\\
\bG_\ell^{\sc} \ar[r]^{\pi_\ell} & \bG_\ell\ar[r] &\bG_\ell^{\ss}.}
\end{equation*}
Then by the definition of $\Gamma_\ell^{\sc}$, the inclusion 
$\pi_\ell([\Gamma_\ell^{\sc},\Gamma_\ell^{\sc}])\subset [\Gamma_\ell,\Gamma_\ell]$ holds.
Suppose Conjecture \ref{Larsenconj} holds, then the hyperspecial maximal compact 
$\Gamma_\ell^{\sc}$ is perfect for $\ell\gg0$ (see e.g., \cite[Proof of Cor. 11]{HL16}).
Thus, for $\ell\gg0$, it follows that
$\pi_\ell(\Gamma_\ell^{\sc})\subset[\Gamma_\ell,\Gamma_\ell]$.
The closed subscheme $\mathcal{G}_\ell\subset\GL_{n,\Z_\ell}$ 
is smooth by Proposition \ref{gpscheme} for $\ell\gg0$. 
Also, the subscheme $\mathcal{G}_\ell^{\der}$ is smooth for $\ell\gg0$ 
(see e.g., \cite[Thm. 9.1.1, $\mathsection9.2.1$]{CHT17}, note that $\bG_\ell^{\der}$ is connected).
Then for $\ell\gg0$,
\begin{equation}\label{hyperchain}
\pi_\ell(\Gamma_\ell^{\sc})\subset[\Gamma_\ell,\Gamma_\ell]\subset\mathcal{G}_\ell^{\der}(\Z_\ell)\subset\mathcal{G}_\ell(\Z_\ell)\subset\GL_n(\Z_\ell).
\end{equation}
If we can prove that $\mathcal{G}_\ell$ is a reductive group scheme over $\Z_\ell$, then 
$\mathcal{G}_\ell(\Z_\ell)\subset\bG_\ell(\Q_\ell)$ is hyperspecial maximal compact by Bruhat-Tits theory.
So it remains to prove that the special fiber $\mathcal{G}_{\F_\ell}$ is reductive.

By taking mod $\ell$ reduction of \eqref{hyperchain}, we obtain by Hensel's lemma that for $\ell\gg0$,
\begin{equation}\label{hyperchain2}
\overline{\pi_\ell(\Gamma_\ell^{\sc})}\subset\overline{\mathcal{G}_\ell^{\der}(\Z_\ell)}=\mathcal{G}_{\F_\ell}^{\der}(\F_\ell)\subset\GL_n(\F_\ell).
\end{equation}
For $\ell\gg n$, let $\bS_\ell\subset\GL_{n,\F_\ell}$ be the \emph{Nori envelope} \cite{Nor87}
of the finite subgroup  $\overline{\pi_\ell(\Gamma_\ell^{\sc})}\subset\GL_n(\F_\ell)$.
It is the connected algebraic subgroup of $\GL_{n,\F_\ell}$ generated 
by the one parameter unipotent subgroups $\{u^t: t\in\overline\F_\ell\}$ for 
all order $\ell$ elements of $\overline{\pi_\ell(\Gamma_\ell^{\sc})}$.
It is semisimple by unipotent.
Let $\overline{\pi_\ell(\Gamma_\ell^{\sc})}^+$ be the (normal) subgroup of $\overline{\pi_\ell(\Gamma_\ell^{\sc})}$
generated by the order $\ell$ elements. Then $\overline{\pi_\ell(\Gamma_\ell^{\sc})}^+$ is a subgroup of $\bS_\ell(\F_\ell)$
and $[\overline{\pi_\ell(\Gamma_\ell^{\sc})}: \overline{\pi_\ell(\Gamma_\ell^{\sc})}^+]$ is prime to $\ell$.
The Nori envelope $\bS_\ell$ approximates the finite subgroup $\overline{\pi_\ell(\Gamma_\ell^{\sc})}\subset\GL_n(\F_\ell)$
in the sense that the index $[\bS_\ell(\F_\ell):\overline{\pi_\ell(\Gamma_\ell^{\sc})}^+]$ 
is bounded by a constant depending only on $n$ for all prime $\ell$ large enough compared to $n$ \cite[Thm. B(1), 3.6(v)]{Nor87}. 

\begin{prop}\label{Noricontain}
For $\ell\gg0$, the smooth group scheme $\mathcal{G}_{\ell}^{\der}$ is reductive.
\end{prop}

\begin{proof}
Suppose $\ell\geq n$.
Since $\Gamma_\ell^{\sc}$ is maximal compact in $\bG_\ell^{\sc}(\Q_\ell)$ for $\ell\gg0$,
the equality $\pi_\ell^{-1}(\mathcal{G}_\ell^{\der}(\Z_\ell))=\Gamma_\ell^{\sc}$ holds for $\ell\gg0$.
Thus, there is a constant $c$ such that 
\begin{equation}\label{index1}
[\mathcal{G}_\ell^{\der}(\Z_\ell):\pi_\ell(\Gamma_\ell^{\sc})]\leq c
\end{equation}
for all $\ell\gg0$ \cite[Cor. 2.5]{HL20}. Hence, after reduction we also have
\begin{equation}\label{index2}
[\mathcal{G}_{\F_\ell}^{\der}(\F_\ell):\overline{\pi_\ell(\Gamma_\ell^{\sc})}]\leq c
\end{equation}
If the proposition is false, then the unipotent radical of the special fiber 
$\mathcal{G}_{\F_\ell}^{\der}$ is non-trivial 
for infinitely many primes $\ell$. Thus, \eqref{index2}
implies that $\overline{\pi_\ell(\Gamma_\ell^{\sc})}$ contains a non-trivial 
normal unipotent subgroup $U_\ell$ (consisting of order $\ell$ elements) for infinitely many primes $\ell$.
Let $\bS_\ell'$ be the Nori envelope of the semisimplification of $\overline{\pi_\ell(\Gamma_\ell^{\sc})}\hookrightarrow\GL_n(\F_\ell)$
(with image $\overline{\pi_\ell(\Gamma_\ell^{\sc})}^{\red}$)
for $\ell\gg0$. By the definition of Nori envelope \cite[$\mathsection\mathsection1,3$]{Nor87}, for all $\ell\gg0$ 
we have a short exact sequence 
\begin{equation}\label{Norises}
1\to \bU_\ell\to \bS_\ell\stackrel{\pi}{\rightarrow}\bS_\ell'\to 1
\end{equation}
where $\pi$ is induced by semisimplification. For infinitely many primes $\ell$, 
we have $\dim\bU_\ell\geq 1$ since $\bU_\ell$ contains a one parameter subgroup 
$t\mapsto u^t:=\text{exp}(t\text{log}(u))$ \cite{Nor87}
for some non-identity element $u\in U_\ell$.

Since $\bS_\ell'$ is semisimple, \cite[Prop. 4(iii)]{HL16} asserts that 
$\dim \bS_\ell'=\dim_\ell\bS_\ell'(\F_\ell)$ (the \emph{$\ell$-dimension} \cite[$\mathsection2$]{HL16}).
Since $\Gamma_\ell^{\sc}$ is hyperspecial maximal compact in $\bG_\ell^{\sc}(\Q_\ell)$,
there is a reductive group scheme $\mathcal{H}_\ell$ over $\Z_\ell$ such that 
the generic fiber is $\bG_\ell^{\sc}$ and $\mathcal{H}_\ell(\Z_\ell)=\Gamma_\ell^{\sc}$.
By the definition of $\ell$-dimension and \cite[Prop. 4(iii)]{HL16} again, we obtain
\begin{equation}\label{elldim}
\dim_\ell\bS_\ell'(\F_\ell)=\dim_\ell\bS_\ell'(\F_\ell)^+=\dim_\ell \overline{\pi_\ell(\Gamma_\ell^{\sc})}^{\red}=\dim_\ell \Gamma_\ell^{\sc}
=\dim_\ell \mathcal{H}_\ell(\F_\ell)=\dim \bG_\ell^{\sc}.
\end{equation}
It follows from \eqref{Norises} that $\dim \bS_\ell>\dim \bG_\ell^{\der}$ for infinitely many $\ell$, but
this contradicts \cite[Thm. 7]{Lar10}.
\end{proof}

Let $\mathcal{G}_{\F_\ell}^{\red}$ be the quotient of $\mathcal{G}_{\F_\ell}^{\circ}$ by its unipotent radical.
For $\ell\gg0$, the special fiber $\mathcal{G}_{\F_\ell}^{\der}$ (of $\mathcal{G}_{\ell}^{\der}$) is a normal connected semsimple 
subgroup of $\mathcal{G}_{\F_\ell}^{\circ}$ 
(Proposition \ref{Noricontain}), which injects into $\mathcal{G}_{\F_\ell}^{\red}$.
It follows that
\begin{equation}\label{eq4}
\begin{split}
\dim \mathcal{G}_{\F_\ell}\geq \dim\mathcal{G}_{\F_\ell}^{\red} &\geq \dim \mathcal{G}_{\F_\ell}^{\der} +\rk\mathcal{G}_{\F_\ell}^{\red}
-\rk \mathcal{G}_{\F_\ell}^{\der}\\
&=\dim \bG_{\ell}^{\der} +\rk\bG_\ell-\rk \bG_{\ell}^{\der}=\dim\bG_\ell.
\end{split}
\end{equation}
Therefore, \eqref{eq4} is an equality and the special fiber $\mathcal{G}_{\F_\ell}$ is reductive for $\ell\gg0$.
\end{proof}

\begin{remark}
Let $F$ be a finitely generated field of characteristic $p$ and $Y$ be a smooth projective variety defined over $F$.
Conjecture \ref{Larsenconj} holds for the $\Q$-compatible system $\{H^w(Y_{\overline F},\Q_\ell)\}_{\ell\neq p}$ \cite[Thm. 1.2]{CHT17}.
\end{remark}

\subsubsection{Proof of Theorem \ref{thm4}}
Embed $\Q_\ell$ into $\C$ for all $\ell$.
Since $\End_{\overline K}(A)=\Z$, the representations $\rho_\ell$ 
are all absolutely irreducible by the Tate conjecture proven by Faltings \cite{Fa83}.
Since the formal bi-character of $(\bG_\ell\hookrightarrow\GL(V_\ell))\times\C$ 
is independent of $\ell$ \cite[Thm. 3.19]{Hu13}, condition \ref{thm4}(b) and Theorem \ref{LP1}
imply that the tautological representation $(\bG_\ell\to\GL(V_\ell))\times\C$ is independent of $\ell$.
Since condition \ref{thm4}(b) and Theorem \ref{thm3}(ii) hold, 
the simple factors of $\bG_\ell\times_{\Q_\ell}\overline\Q_\ell$ are 
of the same type and the invariance of roots condition holds by Corollary \ref{LP2}.
We conclude that condition \ref{thm2}(a),(b),(c) hold.
Hence, Theorem \ref{thm2}(ii), Theorem \ref{hyperex}(ii), Theorem \ref{QH}, and Corollaries \ref{cor2.1}, \ref{cor2.2}
give Theorem \ref{thm4} except the last assertion.
It suffices to show that 
for $\ell\gg0$, the two $\Z_\ell$-representations
\begin{equation}
\begin{split}
V_{\Z_\ell}: \Gal(\overline K/K)&\to \bG(\Z_\ell)=\mathcal{G}(\Z_\ell)\to \GL_{2g}(\Z_\ell),\\
H^1(A_{\overline K},\Z_\ell): \Gal(\overline K/K)&\to \mathcal{G}_\ell(\Z_\ell)\to \GL(H^1(A_{\overline K},\Z_\ell))
\end{split}
\end{equation}
are isomorphic. 

Since $V_{\Z_\ell}\otimes \Q_\ell\cong H^1(A_{\overline K},\Q_\ell)\cong H^1(A_{\overline K},\Z_\ell)\otimes \Q_\ell$,
there is an element $\Phi_\ell$ in the free $\Z_\ell$-module $\Hom_{\Gal(\overline K/K)}(V_{\Z_\ell},H^1(A_{\overline K},\Z_\ell))$
that is non-zero after mod $\ell$ reduction.
Since $\End_{\overline K}(A)=\Z$, the representation $H^1(A_{\overline K},\F_\ell)$ is absolutely irreducible for $\ell\gg0$ \cite[Theorem 4.2]{FW84}.
Thus, the non-zero $\Gal(\overline K/K)$-equivariant map
$$\Phi_\ell\times\F_\ell: V_{\Z_\ell}\otimes \F_\ell\to H^1(A_{\overline K},\F_\ell)$$ 
is surjective for $\ell\gg0$.
By Nakayama's lemma, $\Phi_\ell$ is surjective for $\ell\gg0$. 
Therefore, $\Phi_\ell$ is bijective and induces an isomorphism of the Galois representations $V_{\Z_\ell}$
and $H^1(A_{\overline K},\Z_\ell)$ for $\ell\gg0$. \qed

\begin{remark}
Embed $\Q_\ell$ into $\C$.
Let $\{(\bH_i,V_i): 1\leq i\leq k\}$ be the irreducible factors of the
 irreducible representation $(\bG_\ell^{\der}\to\GL(V_\ell))\times\C$,
i.e., $\bH_i$ is almost simple and $V_i$ is irreducible ($\mathsection2.2.2.1$).
By \cite[Thm. 3.18]{Pi98}, 
the irreducible representation $(\bG_\ell\to\GL(V_\ell))\times\C$ 
is a strong Mumford-Tate pair of weight $\{0,1\}$.
Then \cite[Prop. 4.5]{Pi98}
and \cite[Table 4.6]{Pi98} imply that $k$ is odd and 
the representations $(\bH_i,V_i)$ have the following possibilities:
\begin{enumerate}
\item[$A_r$:] $\bigwedge^r$(standard), $r\equiv 1$ mod $4$, $r\geq 1$.
\item[$B_r$:] Spin, $r\equiv 1,2$ mod $4$, $r\geq 2$.
\item[$C_r$:] Standard, $r\geq 3$.
\item[$D_r$:] Spin$^+$, $r\equiv 2$ mod $4$, $r\geq 6$.
\end{enumerate}

\noindent One observes that each simple Lie algebra has at most one possible representation.
\end{remark}

\subsection{Final remarks}
\begin{enumerate}[(1)]
\item We construct a common $E$-form $\bG\hookrightarrow\GL_{n,E}$ 
of the algebraic monodromy representations $\bG_\lambda\hookrightarrow\GL_{n,E_\lambda}$ of the system  \eqref{rhobullet}
in case it is absolutely irreducible and $\bG_\lambda$ is connected (for all $\lambda$) in Theorem \ref{thm1}(ii). 
The non-absolutely irreducible case and the non-connected case remain open.
\item Let $\rho_\bullet$ be the system in Theorem \ref{thm2} and assume Conjecture \ref{Larsenconj}.
Then Corollary \ref{cor2.1}(i) produces an adelic representation $\rho_\A^{\bG}:\Gal(\overline K/K)\to \bG(\A_\Q)$.
Let $\rho_{\F_\ell}^{\bG}$ be the mod $\ell$ reduction of the $\ell$-component $\rho_{\ell}^{\bG}$ of $\rho_\A^{\bG}$ for $\ell\gg0$.
One can deduce by \cite[Thm. A, Cor. B]{Hu15} that there is a constant $C>0$ such that 
the index satisfies
\begin{equation*}
\hspace{.55in}[\bG(\F_\ell):\rho_{\F_\ell}^{\bG}(\Gal(\overline K/K))]\leq C,\hspace{.1in}\forall\ell\gg0.
\end{equation*}
Thus, the composition factors  of Lie type in characteristic $\ell$ of $\rho_{\F_\ell}^{\bG}(\Gal(\overline K/K))$
can be described when $\ell\gg0$, see a similar result \cite[Cor. 1.5]{Hu18} for certain type A compatible system. 
\item The smooth subgroup scheme $\mathcal G_\lambda\subset\GL_{n,\mathcal{O}_\lambda}$ in Corollary \ref{cor1.5} 
depends on the choice of an $\mathcal O_\lambda$-lattice of $V_\lambda$. 
It is shown in \cite{Ca17} that for almost all $\lambda$, the subscheme
$\mathcal G_\lambda\subset\GL_{n,\mathcal{O}_\lambda}$ is unique up to isomorphism.
\item The $E$-forms $\bG$ and $\bG\subset\GL_{n,E}$ we constructed in Theorem \ref{thm1} are not unique
for the simple reason that 
$\Sh^1(E,\bG^{\ad})$ in Theorem \ref{local-globalred} may not be trivial, where $\bG^{\ad}$ denotes the adjoint quotient of $\bG$.
\item Let $S'$ be a non-empty finite subset of $\mathcal{P}_{E,f}$. 
Actually, by examining the proof, main theorem \ref{meta1} holds if we 
replace $\mathcal{P}_{E,f}^{(p)}$ with $\mathcal{P}_{E,f}\backslash S'$.
\item In the characteristic zero case, Question Q in $\mathsection3.3.2$ should be addressed if one wants to apply
main theorem \ref{meta2} to an $E$-compatible system when $E$ is totally imaginary.
However, one can always use main theorem \ref{meta1} by omitting a finite place of $E$ if one knows that $\bG_\lambda$
is quasi-split for almost all $\lambda$, or, 
one can take the Weil restriction $\mathrm{Res}_{E/\Q}$ ($\mathsection3.2.2$) to obtain a $\Q$-compatible system 
and see if main theorem \ref{meta2} can be applied.
\end{enumerate}

\section*{Acknowledgments} 
I would like to thank Akio Tamagawa, Anna Cadoret, Gebhard B$\ddot{\text{o}}$ckle for their interests and comments, 
 Macro D'Addezio and Ambrus P\'al for their interests and introducing their works \cite{D'Ad20} and \cite{Pa15} to me,
and the referee for his/her comments and suggestions.

\vspace{.1in}

\end{document}